\documentclass[10pt,a4paper]{article}
\usepackage[utf8]{inputenc}
\usepackage{amsmath}
\usepackage{yfonts}
\usepackage{amsthm}
\usepackage{amsfonts}
\usepackage{amssymb}
\usepackage{geometry}                		  \usepackage{tikz}
       		
\usepackage{graphicx}	
\usepackage{subfigure}
\usepackage{float}	
\usepackage{colortbl}
\usepackage{xcolor}
\usepackage{comment}
\usepackage{hyperref}
\usepackage{soul,color}
\usepackage{tabularx}

\theoremstyle{definition}
\newtheorem{defi}{Definition}[section]

\newtheorem{exa}[defi]{Example}

\theoremstyle{remark}

\theoremstyle{plain}
\newtheorem*{thm*}{Theorem 7.12b}

\newtheorem*{fact}{Fact}
\newtheorem{lem}[defi]{Lemma}
\newtheorem{thm}[defi]{Theorem}
\newtheorem{cor}[defi]{Corollary}
\newtheorem{prop}[defi]{Proposition}
\newtheorem{rem}[defi]{Remark}

\title{Transfinite Topological Dynamics}
\author{Alessandro Della Corte\footnote{Mathematics Division, School of Sciences and Technology, University of Camerino (Camerino, MC, Italy). ORCID: 0000-0002-1782-0270. E-mail: alessandro.dellacorte@unicam.it (corresponding author).}, Marco Farotti\footnote{Doctoral School in Computer Sciences and Mathematics, University of Camerino (Camerino, MC,  Italy). ORCID: 0009-0001-5000-2827.}}
\date{}
\begin{document}
\maketitle

\begin{abstract}
\noindent We present a canonical extension of topological dynamics to transfinite iterations, which makes precise the idea of dynamical phenomena stabilizing at different time-scales. 

\noindent Specifically, consider a sequence of self-maps $F=\{f_n\}$ of a compact metric space $X$. If $F$ is finitely convergent, i.e.  $f_n(x)=f(x)$ for $n>N(x)$, the $f_n$-orbits exhibit an emergent poset structure. A maximal initial segment of this poset is isomorphic to a countable ordinal $\ge\omega$. The construction is canonical: every finitely convergent sequence induces, at each point, a unique maximal transfinite orbit that is independent of any finite initial segment of the sequence and invariant under step-by-step conjugacy at each $n$.
For $\lambda$ a countable limit ordinal, we study orbits, recurrence, limit sets and attractors at level $\lambda$, and the interplay of different ordinal levels. 
Moreover, we introduce the natural notion of transfinite conjugacy, that refines conjugacy of limit maps alone but is strictly weaker than step-by-step conjugacy. We describe a family of invariants of transfinite conjugacy that detect recurrence and attraction phenomena at each ordinal level. Particularizing to $\lambda=\omega$ recovers (and in some cases refines) classical results of topological dynamics.

\vspace{5mm}



\end{abstract}

\section{Introduction}
\subsection{Dynamics beyond the limit evolution law}\label{1_1}

In discrete-time dynamics, the basic piece of information one is given is ``what happens next" to a system once its current state and evolution law are known. 
In topological dynamics, the state is typically a point in a compact metric space \(X\), while the evolution law is a continuous self-map of \(X\). 
Sometimes, however, one is led to consider not a single map, but an entire family \(F = \{f_n\}_{n\in\mathbb{N}}\) of evolution laws. 
In many cases, the sequence of dynamical laws \(f_n\) can converge to a limit map \(f\), yet exhibit emergent stages of evolution that cannot be recovered, even approximately, from \(f\) alone. 

To isolate the phenomenon in the simplest form, let us select, for the sequence $F$, a particularly simple form of convergence, that is $f_n(x)=f(x)$ for $n$ larger than a certain $N(x)$ (this is just pointwise convergence to $f$ with respect to the discrete topology on $X$). When looking at the dynamics of a point $x\in X$, one may see that a certain point $y$ appears eventually in all the orbits determined by the maps $f_n$, but it shows up further and further away in time when $n$ goes to infinity. In the simplest possible case, the point $y$ has the following property: 

\begin{itemize}\label{*}
    \item[ ] \textit{For every $k$, the point $y$ appears, for large $n$, only after the $k$th iteration of $f_n$ at $x$. \emph{(*)}}
\end{itemize}

Then it is natural to assign to \(y\) a generalized iteration order larger than any finite one. 
If \(y\) is the first point with property (*), its iteration order must be \(\omega\), the smallest ordinal exceeding all finite ordinals. 
A point that first appears, relative to \(y\), with property (*), must in turn have order \(\omega\cdot 2\) from the viewpoint of \(x\).
The starting point of this work is the realization that: 1. this assignment can be made precise in a natural way, and it extends unambiguously and canonically up to a certain countable ordinal, depending on local properties of the system at \(x\) (see Section~\ref{sec3}); 2. composition of iterations under this assignment obeys the ordinal addition law.
We can summarize this saying that a transfinite orbit emerges: a well-ordered set of iterates of $x$, whose order type is a countable ordinal. 

 Note that, in this framework, the basic piece of information we mentioned at the beginning, i.e., ``what happens next" to each given state of the system, is still provided, because every ordinal has a successor. This was the key motivation for focusing on \textit{well-ordered} initial segments, rather than just ordered (or partially ordered) ones. 
 
 Moreover, there can also be something (that is, iterations corresponding to limit ordinals) happening ``after" a certain ordered set of states, which nevertheless does not happen ``next after" any particular state. 

\begin{table}[H]\label{table_1}
\centering
\fbox{
\begin{tabularx}{0.85\textwidth}{cccccccc}
\arrayrulecolor{red}
\\
\cline{4-4}
$f_1(x),$ & $f_1^2(x),$  & \multicolumn{1}{c}{$\ \ldots$} & \multicolumn{1}{!{\color{red}\vline}c!{\color{red}\vline}}{$f_1^{k_1}(x)=y,$} & $\ldots$ & & & \\ 
\cline{4-5}
$f_2(x),$ & $f_2^2(x),$ & $\ldots$ & \multicolumn{1}{c|}{$\ \ldots \ldots\ldots$} & \multicolumn{1}{c|}{$f_2^{k_2}(x)=y,$} & $\ldots$ & & \\ 
\cline{5-5}
$\vdots$ & $\vdots$ & $\cdots$ & $\cdots\cdots\cdots$ & $\vdots$ & $\textcolor{red}{\cdots} $ & $\cdots$ & \\ \cline{7-7}
$f_n(x),$ & $f_n^2(x),$ & $\ldots $ & $\ldots\ldots\ldots$ & $\ldots$ &\multicolumn{1}{c|}{$\ \ldots$} & \multicolumn{1}{c|}{$f_n^{k_n}(x)=y,$} & $\ldots$ \\ 
\cline{7-7}
$\vdots$ & $\vdots$ & $\vdots$ & $\vdots$ &$\vdots$ &$\vdots$ & $\vdots$ & $\textcolor{red}{\cdots}$ \\
\multicolumn{8}{m{12cm}}{\upbracefill}\\
\multicolumn{8}{c}{\Large$\scriptstyle\{f\}^\beta(x)\, =\, y$}
\end{tabularx}}\caption{\small{Schematic representation of how an iteration of order $\beta\ge\omega$ emerges from the sequence of orbits generated by the maps $f_n$. Here $f_n$ converges finitely to $f$ (see Def. \ref{1}), so, for every $k\in\mathbb{N}$ and $x\in X$, one has $f_n^k(x)=f^k(x)$ for $n$ large enough. At the same time, the iteration $k_n$ at which the point $y$ is visited by $x$ diverges with $n$. The point $y$ belongs to the $f_n$-orbit of $x$ for every $n$, but every point in the $f$-orbit of $x$ appears ``before" $y$ if you go down enough in the table. The particular countable ordinal $\beta$ measures the depth of the time-scale at which $y$ is visited, starting from $x$, for the given sequence of maps.}}
\end{table}

Transfinite structures connected with topological dynamics have of course appeared before. 
Already in 1928, Birkhoff introduced the notion of the \emph{center} of a dynamical system \cite{birkhoff1928structure}, defined via transfinite induction and later generalized 
(\cite{mathias1996long,alseda1999structure}). 
Ordinal structures have also been considered to describe hierarchically recurrence (\cite{auslander1964generalized}) and in entropy theory (\cite{downarowicz2005entropy,mcgoff2011orders,DavidBurguet2012}). Transfinite extensions of discrete iterative processes also arise in combinatorial settings, for instance in infinite generalizations of classical games(\cite{evans2013transfinite,evans2015position,leonessi2021transfinite}). Here the ordinal hierarchy arises at the most elementary level, that of pointwise iteration itself. 
This leads naturally to the concepts of transfinite orbits and, more generally, of \emph{Transfinite Dynamical Systems} (TDSs). 
Our goal is to lay the foundations of the topological dynamics of such systems, beginning from their orbits and proceeding to recurrence, limit sets, attractors, and conjugacy. As a byproduct, in some cases (for instance see Remarks \ref{akin_ref}, \ref{__refin__}, \ref{anomaly} and Corollary \ref{akin_ref1}), the particular case $\lambda=\omega$  will slightly refine certain well-known results and constructions of classical topological dynamics.


\subsection{Outline of the Paper}
We provide an overview of the paper, to highlight the main results and describe the overall flow of the treatment.

\noindent \textbf{In Section \ref{sec3}}, the basic concepts are introduced, in particular transfinite iterations and transfinite dynamical system (TDS), proven to obey to the ordinal addition law for composition of iterations (Theorem \ref{lem_group}). Transfinite cycles are closed orbits of countable ordinal length. We completely characterize them in terms of finite orbits: they are always   set-theoretic $\liminf$ of nested finite cycles of increasing order; this is proved in  Theorem \ref{C_n}.


A measure of ``how far" in time a transfinite system goes respectively locally and globally is formalized through the concept of ordinal degree of a point and of the system itself. This will be later (Section \ref{sec8}) proven to be a (transfinite) topological invariant.

\textbf{In Section \ref{sec4}} we look at transfinite orbits through examples. The main result is a realization theorem on a standard continuum: there are sequentially continuous transfinite dynamical systems defined on the interval having arbitrarily large countable ordinal degree; this is proved in Theorem \ref{th_existence}.


\textbf{In Section \ref{sec5}} we introduce the transfinite analog of regularity properties. In particular, we define the concept of transfinite continuity and call a transfinite dynamical system $\lambda$-regular (in short, a $\lambda$-TDS) if all iterations of order $<\lambda$ are continuous. We start then approaching the theory of transfinite attractors by generalizing E. Akin's dynamical relations to transfinite systems. Transfinite orbit, recurrence, non-wandering and chain recurrence relations are increasingly general, and the latter is closed and transitive. For $\lambda$ a countable limit ordinal, $\lambda$-minimal and $\lambda$-transitive systems are precisely those in which every pair of points is, respectively, in $\lambda$-recurrence or $\lambda$-non-wandering relation. In every TDS there is a point with a dense  orbit, but the converse is not true, because a proper subset of a transfinite orbit can be dense in the space.

\textbf{In Section \ref{sec6}} transfinite limit sets and attractors are introduced. A $\lambda$-attractor is the limit set at level $\lambda$ of a transfinitely inward set. We identify two phase transitions along the ordinal scale.
\begin{itemize} 
\item At $\omega^2$: the classical link between attractiveness and strong invariance breaks:
for $\lambda \ge \omega^2$ there exist proper $\lambda$-attractors $Y$ with
$f(Y)\subsetneq Y$, whereas for $\lambda<\omega^2$ every proper $\lambda$-attractor is
strongly invariant.
\item At $\omega^\omega$: the implication “inward $\Rightarrow$ uniformly inward” fails:
for $\lambda \ge \omega^\omega$ one can have inward $V$ with $d\,\!\Big(\!\bigcup_{\beta<\lambda}\{f\}^{\beta}(\overline V),\,\partial V\Big)=0$ (here $d$ is the metric and $\{f\}^\beta$ indicates the iteration of order $\beta$), while for $\lambda<\omega^\omega$ inwardness is always uniform.
\end{itemize}

Uniformly inward sets converge to uniform transfinite attractors, which are more tractable than general ones. Assuming uniformity, the straightforward generalization of an ordinary attractor, that is proper $\lambda$-attractors, are closed, invariant and attractive (at the corresponding ordinal level), but in general not stable, while larger objects, extended $\lambda$-attractors, are closed, invariant and stable, but in general not attractive; this is proved in Theorem \ref{themattractors}.


\textbf{Section \ref{sec7}} and \textbf{\ref{sec8}} are the core of the paper. In Section \ref{sec7}, we addresses deeper dynamical properties of $\lambda$-attractors, aiming at a general classification in terms of invariant properties. We will see that finite (that is, usual) attractors occupy a small region of a much richer structural landscape (graphically represented in Figs. \ref{scheme_attrac_1} and \ref{scheme_attrac_2}).
Uniform extended attractors verify (under some conditions) a  generalization of a well-known structure result on attractors by Akin, in which  suitable transfinite versions of the non-wandering and chain recurrence relations (denoted by $\lambda\{\mathcal{D}\}$ and $\lambda\{\mathcal{G}\}$) appear; this is proved in Theorem \ref{__akin_}.


Uniform attractors are further particularized by complete attractors, that are limit sets of completely (not just uniformly) inward sets: larger-order iterations of the set are always included in smaller-order ones. Complete attractors share several properties of standard attractors in ordinary dynamical systems, but are not, in general, transfinitely strongly invariant: the transfinite images of a complete attractor $A$ are closed and included in $A$, but do \textit{not} necessarily coincide with $A$. Perfect attractors (a special class of complete attractors) are the truly well-behaved objects: in addition to being stable and attractive, they are also strongly invariant at every transfinite level. A sufficient condition for perfectness is $\lambda$-reachability: if a point $y$ belonging to the $\lambda$-attractor is reached by a sequence of points admitting a certain transfinite iteration $\beta<\lambda$, then $y$ also admits an iteration of order $\beta$; this is proved in Theorem \ref{strong_inv_perf_}.
Perfect attractors obey a neat generalization of Akin's Theorem involving two further transfinite versions of the non-wandering and the chain recurrence relations (namely $\lambda\{\mathcal{B}\}$ and $\lambda\{\mathcal{F}\}$); this is proved in Theorem \ref{akin_perf_}.

In \textbf{Section \ref{sec8}} we prove that the main concepts introduced in Sections \ref{sec6} and \ref{sec7} are invariant. Indeed, for $\lambda$ a countable limit ordinal, two topologically $\lambda$-conjugate systems have, up to a homeomorphism, the same ordinal degree, orbits, limit sets, attractors, basins; this is proved in Theorem \ref{_conj__}. The transfinite conjugacy between two systems can be checked by looking at the sequences defining the systems: if they are element-by-element conjugate (in the ordinary sense), the systems are $\lambda$-conjugate. The converse is in general not true; this is proved in Theorem \ref{sequenc_conj__},  and it shows that in fact transfinite conjugacy is more meaningful than literal one-by-one conjugacy of the maps in the sequences. 

\vspace{0.3cm}

For the sake of clarity,
we have been abundant with examples and counter-examples. They have been selected to be always as simple as possible, so as to illustrate the concepts clearly. This deliberate simplicity makes transparent that the phenomena under study already occur in minimal settings.

\section{Transfinite dynamical systems}
\label{sec3}

Throughout, $X$ is a compact metric space with metric $d$ and such that $|X|\ge \aleph_0$. When we say that $(X,f)$ is a \emph{topological dynamical system}, we assume that $f$ is a continuous self-map of $X$.

We will indicate by $\mathbb{N}$ and $\mathbb{N}_0$, respectively, the set of positive and non-negative integers, and by $\mathbb{Z}$, $\mathbb{R}$ and $\mathbb{C}$ the sets of integers, real numbers and complex numbers.  We will generally use the last letters of the Latin alphabet ($x,y,z,w,\dots$) for points in $X$, mid-alphabet letters $(n,m,h,j,\ell,\dots)$ for natural numbers and Greek letters ($\lambda,\alpha,\beta,\eta,\iota,\dots$) for countable ordinals. As usual, $\omega$ is the first infinite ordinal and $\omega_1$ is the first uncountable ordinal. 

Given two subsets $A,B\subseteq X$, we set $d(A,B)=\inf\{d(x,y):x\in A,\ y\in B\}$.  With a slight abuse of language, we will follow the habit of calling $d(A,B)$ the  ``distance" between  $A$ and $B$; we also set $d(A,\emptyset)=d(\emptyset,A)=+\infty$ for every nonempty set $A$.
For $\epsilon>0$ and $A\subseteq X$, we indicate by $B_\epsilon(A)$ the open ball of radius $\epsilon$ around $A$, that is $B_\epsilon(A):=\{x\in X: d(x,A)<\epsilon\}$. We indicate by $\overline{B}_\epsilon(A)$ the closed ball of radius $\epsilon$ around $A$, that is $\overline{B}_\epsilon(A):=\{x\in X: d(x,A)\le\epsilon\}$. The topological closure of a set $S\subseteq X$ will be indicated by $\overline{S}$, so that the closure of the open ball of radius $\epsilon$ around $A$ will by indicated by $\overline{B_\epsilon(A)}$. When we say that $U(x)$ is an open neighborhood of $x\in X$ we mean that $U\subseteq X$ is an open set such that $x\in U$. 
\begin{defi}
We say that the sequence of maps $\{f_n\}_{n\in\mathbb{N}}$ is \emph{finitely convergent} to a map $f$, in symbols
\[
f_n \dot{\longrightarrow} f \quad \text{for }n\to \infty,
\]
if, for every $x\in X$, there is $N\in\mathbb{N}$ such that $f_n(x)=f(x)$ for every $n>N$. 
\label{1}
\end{defi}



Let us indicate by $\{f_n\}_{n\in\mathbb{N}}$ and $f$, respectively, a finitely convergent sequence of self-maps of $X$ and its pointwise limit. For every point $x\in X$, we denote by $\mathcal{O}_{f_n}(x)$ the $f_n$-orbit of $x$, that is: $$\mathcal{O}_{f_n}(x):=\{f_n(x),f_n^2(x),\dots\}.$$ We will write simply $\mathcal{O}_n(x)$ instead of $\mathcal{O}_{f_n}(x)$ if the sequence of maps is clear from the context. Moreover, we denote by $\mathcal{O}_f(x)$ (or simply by $\mathcal{O}(x)$ if the limit map is clear from the context) the $f$-orbit of $x$, that is the set $\{f(x),f^2(x),\dots\}$. When considering an invertible map $h$, we will use the symbol $\mathcal{O}_h^{\mathbb{Z}}(x)$ (or simply $\mathcal{O}^{\mathbb{Z}}(x)$ if the map is clear from the context) to indicate the full backward and forward orbit: $$\mathcal{O}_h^{\mathbb{Z}}(x):=\{h^k(x)\}_{k\in\mathbb{Z}}.$$

Our aim is to extract $n$-asymptotic dynamics from the sequence $\{\left(X,f_n \right)\}_{n\in\mathbb{N}}$. This will lead to the definition of a \textit{transfinite dynamical system}, that we will indicate by $\left(X,\{f\}\right)$. For this, we need first to introduce some technical tools of order-theoretic nature.

\begin{defi}\label{defSTlim}
We recall here some elementary set-theoretic concepts. Let $\{S_n\}_{n\in\mathbb{N}}$ be a sequence of sets.  The set-theoretic limit inferior ($\liminf$) and limit superior ($\limsup$) of the sequence are defined as follows:
$$\liminf_{n\to\infty} S_n=\bigcup_{n\ge 1}\bigcap_{j\ge n}S_j \quad , \quad \limsup_{n\to\infty} S_n=\bigcap_{n\ge 1}\bigcup_{j\ge n}S_j .$$
Moreover, if there exists a set $S$ such that $\underset{n\to\infty}\liminf\, S_n=\underset{n\to\infty}\limsup\, S_n=S$, then $S$ is called the (set theoretic) limit of $\{S_n\}_{n\in\mathbb{N}}$, denoted by 
 $S=\underset{n\to\infty}\lim\  S_n$.
 \end{defi}
 
\begin{defi}\label{order}
For every $x\in X$ and $n\in\mathbb{N}$ and for every pair of distinct points $z,y\in\mathcal{O}_n(x)$, we set: $$y<_{x,n} z\Longleftrightarrow\min\{k\in\mathbb{N}:f_n^k(x)=y\}<\min\{k\in\mathbb{N}:f_n^k(x)=z\}.$$  
Set $$\textswab{O}(x):=\bigcup_{n=1}^\infty \mathcal{O}_n(x) \quad\quad,\quad\quad \mathcal{O}_\infty(x):= \underset{n\to\infty}\liminf\,  \mathcal{O}_n(x).$$ Let us introduce a strict partial order on $\mathcal{O}_\infty(x)$ by taking the direct limit of the orders $<_{x,n}$. We set thus, for every pair of distinct $z,y\in \mathcal{O}_\infty(x)$,
 \begin{equation}
y<_{x,\infty} z\Longleftrightarrow \exists\, N\in\mathbb{N} : y<_{x,n} z \quad\forall n> N.
\end{equation}
\end{defi}
From now on, we will write simply $<_n$ and $<_\infty$ when the point $x$ with respect to which the relation holds is clear from the context, while we will use the full notation when confusion may arise. We will also write $y\le_{\infty} z$ to mean ``either $y<_{\infty} z$ or $y=z$". 

Using the strict order relation $<_\infty$, we define now inductively the transfinite iterations of the point $x$ associated to the finitely convergent sequence $\{f_n\}_{n\in\mathbb{N}}$. This will be done in Def. \ref{cycle}, but we will arrive at it in two steps, starting from Def. \ref{transfinite}, which covers what we call \textit{basic transfinite iterations}. Those are proven, in Theorem \ref{lem_group}, to obey the (ordinal version of the) additive law for composition of iterations, which in turn motivates Def. \ref{cycle}. 

Throughout the work, we use a few elementary facts about ordinals without mentioning them explicitly. In particular, we exploit systematically that: i) every ordinal smaller than $\omega_1$ has countable cofinality;
    ii) ordinal sum is associative;
    iii) the additively indecomposable countable ordinals are precisely the ones of the form $\omega^\lambda$ for some $\lambda<\omega_1$;
    iv) every ordinal admits a unique Cantor Normal Form.
For a standard introduction to ordinal numbers we refer, for instance, to \cite{jech2003set} (p.12 onwards), and \cite{ciesielski1997set} (Chapters 4-5).

\begin{defi}
\label{transfinite}
For every $x\in X$, set $[f]^1(x):=f(x)$.

Assume $1<\beta<\omega_1$ and suppose that $[f]^\alpha(x)$ has been defined for every $\alpha$ such that $1\le\alpha<\beta$. Set: 
\begin{equation}
\label{Obeta}
    \mathcal{O}_{[\beta]}(x):=\{[f]^\alpha(x):1\le\alpha<\beta\ \text{and}\  [f]^\alpha\ \text{is defined}\}
\end{equation}
and suppose that there is $z\in\mathcal{O}_\infty(x)$ such that 

\begin{equation}
\forall y\in\mathcal{O}_{[\beta]}(x),\ \forall w\in\mathcal{O}_\infty(x)\setminus\left(\mathcal{O}_{[\beta]}(x)\cup\{z\}\right),\ \text{we have}\ y<_\infty z<_\infty w.
\label{iteration}
\end{equation} 

Such a point $z\in\mathcal{O}_\infty(x)$, if it exists, is clearly unique. Then we define 
$$[f]^\beta(x):=z.$$
Notice that this definition is equivalent to requiring that $[f]^\beta(x)$ is the least element with respect to $<_\infty$ of the set $\mathcal{O}_\infty(x)\setminus\mathcal{O}_{[\beta]}(x)$, so that a necessary and sufficient condition for the existence of $[f]^\beta(x)$ is that $\mathcal{O}_\infty(x)\setminus\mathcal{O}_{[\beta]}(x)$ has a least element.
If instead there is no $z\in\mathcal{O}_\infty(x)$ verifying \eqref{iteration} (equivalently, $\mathcal{O}_\infty(x)\setminus\mathcal{O}_{[\beta]}(x)$ has no least element), then $[f]^\eta(x)$ is not defined (we may say as well \emph{does not exist}) for every $\eta\ge\beta$. Finally, we set 
\begin{equation}\label{set_iter}
    [\mathcal{O}](x)=\{[f]^\alpha(x):\ 1\le\alpha<\omega_1\text{ and } [f]^\alpha(x)\text{ is defined}\}.
\end{equation}
\end{defi}

\begin{rem}
Let be $\lambda=\sup\{\beta<\omega_1 : [f]^\beta(x) \text{ exists}\}$.
By construction, the set $[\mathcal{O}](x)$ is order-isomorphic to $\lambda$. 
Hence, the relation $<_\infty$ is a strict well-ordering of $[\mathcal{O}](x)$. 
\end{rem}

The above definition allows for the existence of basic transfinite iterations, that is points belonging to the space $X$ which can be written as $[f]^{\beta}(x)$ for some $\beta\ge\omega$ and $x\in X$. We call them ``basic" because they are not, as such, a satisfactory transfinite version of ordinary iterations of a map. More specifically, the iterations defined in Def. \ref{transfinite} \textit{do not} generalize finite iterations. Indeed, if $x$ belongs to a periodic $f$-orbit of order $k\in\mathbb{N}$, we have that $[f]^{k+1}(x)$ is not defined, while of course $f^{k+1}(x)=f(x)$. For this reason, the introduction of Def. \ref{cycle} is required, which is indeed a generalization of both Def. \ref{transfinite} and ordinary finite iterations, allowing in particular for the ``completion of the cycles", whether they are finite or, as we will see, transfinite cycles. 

The first result we want to prove, Theorem \ref{lem_group}, provides the technical justification Def. \ref{cycle}, showing that whenever certain  of basic transfinite iterations exist, they behave
with respect to countable ordinals as finite compositions do with
nonnegative integers: the iteration orders are added,
with the addition being the usual ordinal addition (hence associative but, in
general, noncommutative). We exploit this to define transfinite iterations in
situations where Definition~\ref{transfinite} alone would not be applicable. 
To prove Theorem \ref{lem_group} we need some preliminary lemmas.

\begin{lem}\label{lem_inclus_}
\label{inclusion}
If $[f]^\beta(x)=y$, then \begin{equation} \label{inclus_}
\mathcal{O}_{\infty}(y)\subseteq \mathcal{O}_\infty(x).
\end{equation}
Moreover, we have: \begin{equation} \label{setz_}\mathcal{O}_\infty(x)\setminus\mathcal{O}_\infty(y)\subseteq \mathcal{O}_{[\beta+1]}(x).
\end{equation}
\end{lem}
\begin{proof}
If 
$z\in\mathcal{O}_n(y)$ for every $n>N_1$ and $y\in\mathcal{O}_n(x)$ for every $n>N_2$, then $z\in \mathcal{O}_n(x)$ for every $n>\max\{N_1,N_2\}$.
Moreover, if $z\in \mathcal{O}_\infty(x)\setminus\mathcal{O}_\infty(y)$ and $z\notin \mathcal{O}_{[\beta+1]}(x)$, since $y=[f]^\beta(x)$, it follows that $y<_{\infty,x} z$. Hence, there exists $N>0$ and $k_n,h_n\in\mathbb{N}$ with $h_n<k_n$ such that $f_n^{k_n}(x)=z$ and $f_n^{h_n}(x)=y$ for all $n>N$. Then $f_n^{k_n-h_n}(y)=z$ for all $n>N$, which means that $z\in\mathcal{O}_\infty(y)$, and that is a contradiction.
\end{proof}

\begin{lem}\label{nec_cond_exist}
   For every countable ordinal $\beta$, we have $[f]^\beta(x)\notin \mathcal{O}_{[\beta]}(x)$. 
\end{lem}
\begin{proof}
    Set $y:=[f]^\beta(x)$. Then assuming $\mathcal{O}_{[\beta]}(x)\ni y=[f]^\eta(x)$ for some $\eta<\beta$, it follows $y=[f]^\eta(x)<_\infty [f]^\beta(x)=y$,
    which is impossible.
\end{proof}

\begin{lem}\label{lem_exist_iter}
Assume that there exists $[f]^{\alpha}(x)$ for some countable ordinal $\alpha$. Then $[f]^{\alpha+1}(x)$ exists if and only if $f([f]^\alpha(x))\notin \mathcal{O}_{[\alpha+1]}(x)$, and in that case we have: $$[f]^{\alpha+1}(x)=f([f]^\alpha(x)).$$
\end{lem}

\begin{proof} 
Set $y:=[f]^\alpha(x)$. 
Then there exist a positive integer $N_1$ and a sequence of positive integers $\{k_n\}_{n\in\mathbb{N}}$ such that $f_n^{k_n}(x)=y$ for all $n>N_1$, with $k_n=\min\{k\in\mathbb{N}:f_n^k(x)=y\}$. 
If $f(y)=f([f]^\alpha(x))\in \mathcal{O}_{[\alpha+1]}(x)$, then there exists $\beta\le \alpha$ such that $f(y)=[f]^\beta(x)$, so there is a sequence of positive integers $\{h_n\}_{n\in\mathbb{N}}$  and  $N_2$ so large that, for all $n>N_2$, $$f(y)=f_n(y)=f_n(f_n^{k_n}(x))=f_n^{k_n+1}(x)=f_n^{h_n}(x),$$ for some $h_n=\min\{h\in\mathbb{N}:f_n^h(x)=f(y)\}\le k_n$.
Hence, for every $n>\max\{N_1,N_2\}=:N$, the set $\{f_n^{h_n}(x),\ldots,f_n^{k_n}(x)\}$ is a cycle of order $k_n-h_n+1$ with respect to the map $f_n$, which means that $\mathcal{O}_n(x)$ is a finite set. Then, for every $n>N$, there is no $z\in \mathcal{O}_n(x)$ such that $y<_n z$. Indeed, for every $z\in\mathcal{O}_n(x)$ and every $n>N$, we have: 
$$k_n=\min \{k\in\mathbb{N} : f_n^k(x)=y\}\ge \min \{k\in\mathbb{N} : f_n^k(x)=z\}.$$ 
It follows that there is no point $z$ in $\mathcal{O}_\infty(x)$ verifying \eqref{iteration} with $\beta=\alpha+1$, so $[f]^{\alpha+1}$ is not defined.

Otherwise, suppose that $f([f]^\alpha(x))\notin \mathcal{O}_{[\alpha+1]}(x)$. By the finite convergence of $\{f_n\}_{n\in\mathbb{N}}$, it follows that, for $n$ large enough, $f([f]^\alpha(x))=f_n([f]^\alpha(x))$, so that $f([f]^\alpha(x))\in \mathcal{O}_\infty (x)$ and 
\begin{equation*}
f([f]^\alpha(x))=\min_{<_{_\infty}}\{ z\in X : [f]^\alpha(x) <_\infty z\}.
\end{equation*}
This implies that $[f]^{\alpha+1}(x)=f([f]^\alpha(x))$.
\end{proof}

\begin{thm}\label{lem_group}
If $[f]^\beta(x)=y$ and $[f]^\eta(y)=s$, then $[f]^{\beta+\eta}(x)$ exists if and only if $s\notin \mathcal{O}_{[\beta+\eta]}(x)$, and in that case we have: 
\begin{equation}
\label{group}
[f]^{\beta+\eta}(x)=[f]^\eta\left([f]^\beta(x)\right).
\end{equation}
\end{thm}

\begin{proof} 
Let us first prove the equality \eqref{group} assuming that there exists $[f]^{\beta+\eta}(x)$.  

By Lemma \ref{lem_exist_iter}, it follows that $[f]^{\beta+1}(x)$, if it exists, verifies $$[f]^{\beta+1}(x)=f([f]^\beta(x))=[f]^1([f]^\beta(x)).$$
Now we proceed by transfinite induction up to $\eta$. 
Take $\gamma\le \eta$ and suppose that $[f]^{\beta+\gamma}(x)$ exists and that, for every $\alpha$ verifying $1\le\alpha<\gamma$, we have $[f]^{\beta+\alpha}(x)=[f]^\alpha([f]^\beta(x))=[f]^\alpha(y)$.  We want to show that $[f]^{\beta+\gamma}(x)=[f]^\gamma(y)$. 

By Lemma \ref{inclusion}, we know that $[f]^\gamma(y)\in \mathcal{O}_\infty(x)$. 

\begin{itemize}
\item Let us prove that \begin{equation}\label{ineqq__}
[f]^\gamma(y)\ge_{\infty,x} [f]^{\beta+\gamma}(x).   
\end{equation}

From the inductive hypothesis, it follows that $<_{\infty,x}$ extends the order $<_{\infty,y}$ to $\mathcal{O}_{[\beta+\gamma]}(x)$. More precisely, we have:
\begin{enumerate}
    \item[(i)] if $a<_{\infty,y}b$ with $a,b\in \mathcal{O}_{[\gamma]}(y)$, then there exist $1\le\alpha_1<\alpha_2<\gamma$ such that $a=[f]^{\alpha_1}(y)=[f]^{\beta+\alpha_1}(x)$ and $b=[f]^{\alpha_2}(y)=[f]^{\beta+\alpha_2}(x)$. Hence $\beta+\alpha_1<\beta+\alpha_2$ and $a<_{\infty,x}b$.
    \item[(ii)] If $a\in \mathcal{O}_{[\beta]}(x)$, $b\in \mathcal{O}_{[\gamma]}(y)$, then $a<_{\infty,x} y$ and $y<_{\infty,y}b$. By item (i) we have that $y<_{\infty,x} b$. By transitivity, it follows that $a<_{\infty,x} b$.
\end{enumerate} 
In particular, it follows that:
\begin{equation}\label{transitivity__}
u<_{\infty,x} [f]^\gamma(y)\ \text{for all }u\in \mathcal{O}_{[\beta+\gamma]}(x).
\end{equation}
Since $[f]^{\beta+\gamma}(x)$ exists, we have by definition $u<_{_\infty,x}[f]^{\beta+\gamma}(x)$ for $u\in \mathcal{O}_{[\beta+\gamma]}(x)$. It cannot be $[f]^\gamma(y)<_{\infty,x} [f]^{\beta+\gamma}(x)$, because otherwise by definition $[f]^\gamma(y)\in \mathcal{O}_{[\beta+\gamma]}(x)$, contradicting \eqref{transitivity__}. Therefore,
the inequality \eqref{ineqq__} follows.

\item Let us prove that $$[f]^\gamma(y)\le_{\infty,x} [f]^{\beta+\gamma}(x).$$ 

Assume the contrary, that is $[f]^{\beta+\gamma}(x) <_{\infty,x} [f]^\gamma(y)$.
Again, using \eqref{iteration} with $z=[f]^\gamma(y)$ and recalling that $<_{\infty,x}$ extends the order $<_{\infty,y}$,  it follows that $[f]^{\beta+\gamma}(x)\in \mathcal{O}_{[\gamma]}(y)$. Using the inductive hypothesis, this is equivalent to $[f]^{\beta+\gamma}(x) \in \mathcal{O}_{[\beta+\gamma]}(x)$, which is impossible by Lemma \ref{nec_cond_exist}.
\end{itemize}

We can conclude that, if $[f]^{\beta+\eta}(x)$ exists, we have $[f]^\gamma(y)=[f]^{\beta+\gamma}(x)$ for every $\gamma\le \eta$. 

\vspace{0.2cm}

Suppose now that $[f]^{\beta+\eta}(x)$ is not defined and assume that $s\notin \mathcal{O}_{[\beta+\eta]}(x)$. Set
$$S:=\{w\in \mathcal{O}_\infty(y)\setminus \mathcal{O}_{[\eta]}(y) : u<_{\infty,y} w, \ \forall u\in\mathcal{O}_{[\eta]}(y)\}.$$
Using \eqref{iteration} with $z=s$, we have that $s$ is the least element of the set $S$ with respect to $<_{\infty,y}$. Applying item (ii), it follows that $u<_{\infty,x} s$ for all $u\in \mathcal{O}_{[\beta+\eta]}(x)$. By the inclusion \eqref{inclus_}, and since $s\notin \mathcal{O}_{[\beta+\eta]}(x)$, we have that $s\in\mathcal{O}_\infty(x)\setminus\mathcal{O}_{[\beta+\eta]}(x)$. Moreover, using the inclusion \eqref{setz_}, we obtain: $$\big(\mathcal{O}_\infty(x)\setminus\mathcal{O}_{[\beta+\eta]}(x)\big)\setminus
(\mathcal{O}_\infty(y)\setminus\mathcal{O}_{[\eta]}(y))=\mathcal{O}_\infty(x)\setminus\big(\mathcal{O}_{[\beta+\eta]}(x) \cup \mathcal{O}_\infty(y)\big)=\emptyset.$$
It follows that $s$ is the least element of $\mathcal{O}_\infty(x)\setminus\mathcal{O}_{[\beta+\eta]}(x)$ with respect to $<_{\infty,x}$. 
This means that $[f]^{\beta+\eta}(x)=s$, which is a contradiction because we assumed that $[f]^{\beta+\eta}(x)$ does not exist. 
\end{proof}

Motivated by Theorem \ref{lem_group}, we now extend the validity of \eqref{group} making it a \textit{definition} of composition of transfinite iterations. 

\begin{defi}[\textbf{Transfinite iterations}]
\label{cycle}
For every $x\in X$, we set: $$\{f\}^0(x):=x,$$
$$\{f\}^\beta(x):=[f]^\beta(x)\text{ if $[f]^\beta(x)$ is defined for  $\beta<\omega_1$}.$$ Moreover, whenever there are ordinals $\beta,\eta<\omega_1$ and points $x,y\in X$ such that
\begin{equation}\label{composition1}
\{f\}^\beta(x)=y\quad,\quad
\{f\}^\eta(y)=z,
\end{equation}
we set:
\begin{equation} \label{composition2}
\{f\}^{\beta+\eta}(x)=z.
\end{equation}
Therefore, every point of type $y=\{f\}^\beta(x)$ for some countable ordinal $\beta$ and some $x\in X$ can be obtained starting from the basic transfinite iterations and applying a finite number of times  the composition rules \eqref{composition1}-\eqref{composition2}.
\end{defi}

For $\alpha,\beta<\omega_1$, $x\in X$ and $A\subseteq X$, we also define the following sets:
\begin{align}\label{def_orbits_}
    &\mathcal{O}_{\{\beta\}}(x):= \{\{f\}^\alpha(x):1\le\alpha<\beta\},\\
    &\{\mathcal{O}\}(x)\ := \{\{f\}^\alpha(x):\ 1\le\alpha<\omega_1\text{ and } \{f\}^\alpha(x)\text{ is defined}\},\label{def_orbits_2}\\
    &\{f\}^{-\beta}(A):=\left(\{f\}^\beta\right)^{-1}(A)=\{x\in X: \{f\}^\beta(x)\in A\}.
\end{align}

\begin{rem}
Def. \ref{cycle} is well-posed thanks to the associativity of multiplication of ordinals. Indeed, take $x,y,z,w \in X$ such that $\{f\}^\beta(x)=y$, $\{f\}^\eta(y)=z$ and $\{f\}^\gamma(z)=w$ for some $\beta,\eta,\gamma<\omega_1$. We have $\{f\}^{\eta+\gamma}(x)=w$ and $\{f\}^{\beta+\eta}(y)=z$. Moreover, $$w=\{f\}^{\eta+\gamma}(\{f\}^\beta(y))=\{f\}^{\beta+(\eta+\gamma)}(y)=\{f\}^{(\beta+\eta)+\gamma}(y)=\{f\}^\gamma(\{f\}^{\beta+\eta}(y))=w.$$
\end{rem}

\begin{defi}\label{Xbeta}
    For a countable ordinal $\lambda$, we set:
    \begin{equation}
      X^\lambda:=\{x\in X : \text{ $\forall\,$} \beta<\lambda,\ \exists\,y\in X \text{ such that }\{f\}^\beta(x)=y  \}.  
    \end{equation}
    We assume on $X^\lambda$ the relative topology induced from that of $X$.
    For every $\beta<\omega_1$, we say that a set $S\subseteq X$ is \emph{$\beta$-saturated}  if $S\subseteq X^\eta$ for every $\eta<\beta$.
\end{defi}

\begin{rem}\label{emptysetnotation}
Notice that we have always $X^\omega=X$ and, more generally, we have: $$X^{\beta+\omega}=X^{\beta+1}\text{ for every $\beta<\omega_1$}.$$
For every countable ordinal $\beta$ we have thus a map 
\begin{equation}\label{continuity__}
\{f\}^\beta:X^{\beta+1}\to X,
\end{equation} 
defined on a suitable subset of the whole space $X$, which collapses to the empty function if there are no points $x\in X$ for which the iteration of order $\beta$ is defined. In order to simplify the notation, for $A\subseteq X$ and $\beta<\omega_1$, we set $$\{f\}^\beta(A):=\{f\}^\beta(X^{\beta+1}\cap A).$$
In particular, we have $\{f\}^\beta(\{x\})=\emptyset$ if and only if $x\notin X^{\beta+1}$, and in this case, with a slight abuse of notation, we will write as well $\{f\}^\beta(x)=\emptyset$. Notice that this simplifies Def. \eqref{def_orbits_2} to $$\{\mathcal{O}\}(x)\ := \{\{f\}^\alpha(x):\ 1\le\alpha<\omega_1\}$$
\end{rem}

From Theorem \ref{lem_group} and Def. \ref{cycle} it follows immediately the following

\begin{cor}\label{cor_lem_group}
    If $[f]^\beta(x)$ is not defined while $\{f\}^\beta(x)\neq \emptyset$, then $\{f\}^\beta(x)\in \mathcal{O}_{[\beta]}(x)$. 
\end{cor}

Let us prove now that Def. \ref{cycle} can define new \textit{iterations} of a certain point $x$, but  no new \textit{point} can be reached from $x$ by means of it. More precisely, we have the following result.
\begin{prop}\label{orbit_eq}
For every countable ordinal $\beta$, we have $\mathcal{O}_{[\beta]}(x)= \mathcal{O}_{\{\beta\}}(x)$.
\end{prop}

\begin{proof}
If $[f]^\eta(x)=\{f\}^\eta(x)$ for every $\eta<\beta$, the claim follows.
Otherwise, observe first that, since Def. \ref{cycle} extends Def. \ref{transfinite}, it follows that $\mathcal{O}_{[\beta]}(x)\subseteq \mathcal{O}_{\{\beta\}}(x)$.
To prove the converse inclusion,
let $\eta<\beta$ be the smallest ordinal such that $[f]^\eta(x)\ne \{f\}^\eta(x)$. Clearly $\mathcal{O}_{[\eta]}(x)=\mathcal{O}_{\{\eta\}}(x)$. By Def. \ref{cycle}, we have that $[f]^\eta(x)$ is not defined, and it follows from Corollary \ref{cor_lem_group} that $\{f\}^\eta(x)\in\mathcal{O}_{[\eta]}(x)$. Then
$\{f\}^\eta(x)=[f]^\gamma(x)$ for some $\gamma<\eta$, so $\{f\}^\eta(x)\in\mathcal{O}_{[\eta]}(x)$.  Again recalling Corollary \ref{cor_lem_group}, and by the composition rules \eqref{composition1}-\eqref{composition2}, this means that 
$$\{f\}^{\eta+\delta}(x)=[f]^{\gamma+\delta}(x)\in\mathcal{O}_{[\eta]}(x)$$
for every $\delta$ such that $[f]^{\gamma+\delta}(x)$ exists, that is until $\gamma+\delta<\eta$. If $\overline{\delta}$ is such that $\gamma+\overline{\delta}=\eta$,  again by the composition rules \eqref{composition1}-\eqref{composition2}, we have
$$\{f\}^{\eta+\overline{\delta}}(x)=\{f\}^{\gamma+\overline{\delta}}(x)=\{f\}^\eta(x)$$
so that, for every $k\in\mathbb{N}_0$ and every $\delta$ such that $\gamma+\delta<\eta$, we have:
\begin{equation}\label{finite_number__}\{f\}^{\eta+(\overline{\delta}\cdot k)+\delta}(x)=\{f\}^{\eta+\delta}(x)=[f]^{\gamma+\delta}(x)\in\mathcal{O}_{[\eta]}(x)\subseteq \mathcal{O}_{[\beta]},
\end{equation}
where the last inclusion is true because $\eta<\beta$. Observe finally that 
$$
\sup_{k\in\mathbb{N}_0, \delta<\overline{\delta}}\{\eta+(\overline{\delta}\cdot k)+\delta\}=\gamma + (\overline{\delta}\cdot\omega),
$$
but the transfinite order of iteration on the right hand side cannot be reached through a finite number of applications of the composition rules \eqref{composition1}-\eqref{composition2}. Therefore, $\{f\}^\iota=\emptyset$ for every $\iota\ge \gamma+(\overline{\delta}\cdot \omega)$, so the iterations appearing in Eq. \eqref{finite_number__} exhaust all the transfinite iterations of $x$.
\end{proof}

\begin{prop}\label{periodic_finite}
    If $x$ is a periodic or pre-periodic point for $f$, that is $|\mathcal{O}(x)|=N<\aleph_0$, then for every $\beta\ge\omega$ there is no $y\in X$ such that $\{f\}^\beta(x)=y$ .
\end{prop}
\begin{proof}
    The induction procedure defining basic transfinite iterations (Def. \ref{transfinite}) ends at step $N$, since $f^{N+1}(x)\in \mathcal{O}(x)$, so, by Proposition \ref{orbit_eq}, $\mathcal{O}_{\{N\}}(x)=\mathcal{O}_{[N]}(x)$ is a finite set. Since $\omega$ is additively indecomposable, all the transfinite iterations of $x$ have finite order, as the sums appearing in \eqref{composition2} can only reach finite ordinals.  
\end{proof}

\vspace{0.2cm}

\emph{From now on, unless specified otherwise, by transfinite iterations we will always mean the ones defined in Def. \ref{cycle}}.

\begin{defi}[\textbf{Transfinite Dynamical Systems}]
\label{defi_tds}
A pair $(X,\{f_n\}_{n\in\mathbb{N}})$, with $X$ a compact metric space and $\{f_n\}_{n\in\mathbb{N}}$ a sequence of self-maps of $X$ finitely converging to a continuous limit map $f$, is called 
a \emph{transfinite dynamical system} (TDS).
If the maps $f_n$ are continuous for every $n\in\mathbb{N}$, the pair $(X,\{f_n\}_{n\in\mathbb{N}})$ is called a \emph{ sequentially continuous transfinite dynamical system}.  

For $x\in X$, the set $\{\mathcal{O}\}(x)$ defined in Eq. \eqref{def_orbits_2} is called the \emph{transfinite orbit} of $x$, and 
\begin{equation}\label{def_traject}
\beta \mapsto \{f\}^\beta(x)\in \{\mathcal{O}\}(x)
\end{equation}
is called the \emph{transfinite trajectory} of $x$.

We say that the system $(X,\{f_n\}_{n\in\mathbb{N}})$ is \emph{finite} if there is $N$ such that $f_n\equiv f$ for all $n\ge N$.
\end{defi}

We will usually be a bit less precise and denote the transfinite system more compactly by $(X,\{f\})$.
If $(X,\{f\})$ is finite, then, for every $x$, we have $\mathcal{O}_n(x)=\mathcal{O}(x)=\mathcal{O}_\infty(x)$ for large enough $n$. Therefore, no iterations of order $\ge \omega$ can exist, and of course $\{f\}^k(x)=f^k(x)$ for every $x\in X$ and every $k\in\mathbb{N}$, so that we can identify the system $(X,\{f\})$ with the (ordinary) dynamical system $(X,f)$: everything is determined by the limit map $f$ alone, there are no emergent transfinite dynamical phenomena. 
For this reason, in the following, when we want to consider the particular case of an ordinary dynamical system, we may refer to it as a \textit{finite} system.

\begin{rem} \label{_subs_}
TDSs are in general not stable up to subsequences, in the sense that the transfinite orbits of the TDS $(X,\{f_n\}_{n\in\mathbb{N}})$ may not coincide with the ones of $(X,\{f_{n_j}\}_{j\in\mathbb{N}})$. An example of this phenomenon will be shown later (see Example \ref{noAkin_}, and in particular Remark \ref{subs___}).
\end{rem}

\begin{defi}
    For a countable ordinal $\lambda$, we say that the transfinite orbit $\{\mathcal{O}\}(x)$ of a point $x\in X$ is $\lambda$-\emph{open} if 
    $$\{f\}^\alpha(x)\ne \{f\}^\beta(x)\ \text{for every } 0\le \alpha< \beta <\lambda\text{ such that }\{f\}^\alpha(x) \ne\emptyset\text{ and } \{f\}^\beta(x)\ne\emptyset.$$
    We say that $\{\mathcal{O}\}(x)$ is an \emph{open orbit} if it is $\lambda$-open for every $\lambda<\omega_1$. 
\end{defi}
By Theorem \ref{lem_group}, $[f]^\beta(x)=\{f\}^\beta(x)$ for every $\beta<\omega_1$ if $\{\mathcal{O}\}(x)$ is an open orbit.

A consequence of Def. \ref{cycle}, as already said, is the onset of transfinite cycles, a natural concept in that they are, with respect to the sequence $\{f_n\}_{n\in\mathbb{N}}$, the limit of finite cycles of increasing order. This is formalized in the following Def. \ref{cycles} and Theorem \ref{C_n}. The proof of this result requires some intermediate steps, provided in Lemma \ref{nbound_seq} and Corollary \ref{cor_nbound_seq}. 

\begin{defi} [\textbf{Transfinite cycles}]
\label{cycles}
If there is a point $x\in X$ and an ordinal $\beta$ such that $\beta=\min\{\eta :\{f\}^\eta(x)=x,\  \omega\le\eta<\omega_1\}$, we say that $\{ \{f\}^\eta (x): 1\le\eta<\beta \}$ is a \textit{transfinite cycle} of $x$-order $\beta$. We say that a transfinite cycle $C$ has order $\beta$ if 
$\beta=\min\{\eta:\exists\, x\in C: C$ has $x$-order $\eta\}$.
\end{defi}
In case of a finite dynamical system, the order of every point in a cycle is the same, so that one can talk indifferently about the order of a point or the order of a cycle. This difference between the finite and the transfinite case can be seen as a consequence of the fact that addition is commutative below $\omega$ and generally not commutative above $\omega$: if $f^h(x)=y$, $f^k(y)=x$, then $f^{h+k}(x)=x$ and $f^{k+h}(y)=y$, where of course $h+k=k+h$ assuming $h,k<\omega$.
\begin{defi}\label{def_kn}
    Let $\beta<\omega_1$ be such that $[f]^\beta(x)$ is defined. Let us define a map $N^\beta: X\rightarrow \mathbb{N}$ where $N^\beta(x)$ for $x\in X$ is the least $n\in\mathbb{N}$ such that: 
    \begin{enumerate}
        \item $[f]^\beta(x)\in \mathcal{O}_m(x)$ for all $m\ge n$;
        \item for every $m\ge n$, we have that
        $$u<_m [f]^\beta(x)<_m w,$$ 
        for all $u\in\mathcal{O}_{[\beta]}(x)\cap \mathcal{O}_m(x),$ and for all $w\in(\mathcal{O}_\infty(x)\cap \mathcal{O}_m(x))\setminus\left(\mathcal{O}_{[\beta]}(x)\cup\{[f]^\beta(x)\}\right)$.
        
    \end{enumerate}
    
    Therefore, starting from $n=N^\beta(x)$, the point $[f]^\beta(x)$ appears in the correct ``place" in all the orbits $\mathcal{O}_n(x)$. We will write simply $N^\beta$ when the point $x$ with respect to which the map is applied is clear from the context. 
    
    For $x\in X$, let the map $i^\beta_x:\{n\in\mathbb{N} : n>N^\beta(x)\}\longrightarrow \mathbb{N}$ be defined by 
    \[
    i^\beta_x(n):=\min \{k\in\mathbb{N} : f_n^k(x)=[f]^\beta(x)\}.
    \]
    Let us define the sequence $\{k^{\beta,x}_n\}_{n>N^\beta(x)}$ where $k^{\beta,x}_n:=i^\beta_x(n)$. From now on, we will simply write $\{k^\beta_n\}_{n>N^\beta}$ when the point $x$ is clear from the context.
\end{defi}

\begin{lem}\label{nbound_seq}
If $[f]^\beta(x)$ is defined for some $\beta\ge \omega$, then the sequence $\{k^\beta_n\}_{n>N^\beta}$ is unbounded.
\end{lem}
\begin{proof}
Set $y:=[f]^\beta(x)$ and assume that $\{k^\beta_n\}_{n>N^\beta}$ is bounded, that is there exists $N_1>N^\beta$ such that for all $n>N_1$, $f_n^{k^\beta_n}(x)=f_n^k(x)=y$ for some $k\in\mathbb{N}$. Since $f_n \dot{\longrightarrow} f$, there exists $N_2>N_1$ such that $f_n^k(x)=f^k(x)=y$ for all $n>N_2$. Then $y\in \mathcal{O}(x)$ and $[f]^k(x)=y$.
\end{proof}

\begin{cor}\label{cor_nbound_seq}
Let $\omega\le \alpha<\beta<\omega_1$ be limit ordinals such that $[f]^\alpha(x)$ and $[f]^\beta(x)$ exist. Then there exists $N>0$ such that $\{h_n\}_{n>N}$, where $h_n=k^{\beta,x}_n-k^{\alpha,x}_n$, is positive and unbounded.
\end{cor}
\begin{proof}
    Take $\eta<\omega_1$ such that $\alpha+\eta=\beta$. Since $\alpha,\beta$ are limit ordinals, $\eta\ge \omega$. By Theorem \ref{lem_group}, we have  $[f]^\beta(x)=[f]^{\eta}([f]^\alpha(x))$. Set $y:=[f]^\alpha(x)$. By Lemma \ref{nbound_seq}, the sequence $\{k^{\eta,y}_n\}_{n>N^\eta(y)}$ is unbounded and we conclude by observing that $k^{\beta,x}_n=k^{\eta,y}_n+k^{\alpha,x}_n$ for $n>N^\beta(x)$.
\end{proof}

\begin{thm}
\label{C_n}
Let $(X,\{f\})$ be a TDS. Then every transfinite cycle in $(X,\{f\})$ is the limit inferior of finite cycles of increasing order. 
More precisely, if for $x\in X$ there exists a cycle $$C=\{\{f\}^\eta(x): 1\le\eta \le\beta\}$$ of $x$-order $\beta\ge \omega$, then, for every $n>N^\beta$, $x$ belongs to a cycle $C_n= \{f_n(x),f_n^2(x),\ldots,f_n^{k^\beta_n-1}(x),x\}$ of order $k^\beta_n$, that is $f_n^{k^\beta_n}(x)=x$. Moreover, we have: 
\begin{equation}\label{cycliminf}
C=\liminf_{n\to\infty}C_n.
\end{equation}
\end{thm}
\begin{proof}
By Lemma \ref{nbound_seq}, the sequence $\{k^\beta_n\}_{n>N^\beta}$ is unbounded. Now observe that $\beta$ is the smallest ordinal such that $\{f\}^\beta(x)=[f]^\beta(x)=x$, so it follows that $C_n= \{f_n(x),f_n^2(x),\ldots,f_n^{k^\beta_n-1}(x),x\}$ is a cycle of least order $k^\beta_n$ for all $n>N^\beta$. 

To prove \eqref{cycliminf}, notice first that $x\in C$ and  $x\in C_n$ for all sufficiently large $n$. If $y\in C$ and $y\neq x$, then $y=\{f\}^\eta(x)$ for some $\eta<\beta$. Hence there exists $N>N^\beta$ such that $y<_n x$ for all $n>N$. This means that there exists a sequence $\{h_n\}_{n\in\mathbb{N}}$ of natural numbers such that $f_n^{h_n}(x)=y$ and $h_n<k^\beta_n$ for all $n>N$. Then $y\in C_n$ for all $n>N$, that is $y\in \liminf_{n\to\infty}C_n$. On the other hand, if $y\in \liminf_{n\to\infty}C_n$ and $y\neq x$, then $y\in \mathcal{O}_{\{\beta\}}(x)$, which implies $y\in C$.
\end{proof}

\begin{defi}[\textbf{Ordinal degree}]\label{degree_}
We indicate by $D(x)$ the ordinal degree of $x\in X$, that is the largest ordinal $\beta\le \omega_1$ such that $x$ has transfinite iterations for every smaller ordinal:
\begin{equation}
    D(x)=\beta \iff \{f\}^\eta(x)\ne\emptyset\ \text{for every}\ 0<\eta<\beta.
\end{equation}
We further define the ordinal degree $\mathfrak{D}$ of the whole system as:
\begin{equation}
    \mathfrak{D}(X,\{f\}):=\sup_{x\in X} D(x)
\end{equation}
Since $\vert \textswab{O}(x)\vert \le \aleph_0$, it is clear that $D(x)$ is always less than $\omega_1$.  Moreover, since for every $\beta<\omega_1$ such that $\{f\}^\beta(x)\ne\emptyset$ we have $\{f\}^{\beta+1}(x)=f(\{f\}^\beta(x))\ne\emptyset$, it follows that $D(x)$ is always a limit ordinal.
Notice that the limit ordinal $D(x)$ is the domain of the trajectory map defined in Eq. \eqref{def_traject}, and that $X^{\lambda}\neq \emptyset \implies \mathfrak{D}(X,\{f\}) \ge \lambda$.   
\end{defi}
Finally let us point out that, for every $x\in X$ and every $\beta$ such that $\omega\le \beta \le D(x)$, we have:
\begin{equation}
\label{hierarchy}
    \mathcal{O}(x)\subseteq \mathcal{O}_{\{\beta\}}(x)\subseteq \{\mathcal{O}\}(x)\subseteq \mathcal{O}_\infty(x)\subseteq \textswab{O}(x). 
\end{equation}
In case of a finite dynamical system, we have $\beta=\omega$ and the  inclusions become all equalities.

\vspace{0.3cm}

We remark that only in very special cases can the family of transfinite iterations be organized as a semigroup (or monoid, upon adjoining the identity) of self-maps of~$X$. 
Indeed, this requires that the entire space is $\omega^{\beta}$-saturated for some countable ordinal~$\beta$, which is a very strong condition. 
Even then, the induced action of the additive monoid~$\omega^{\beta}$ on~$X$ is in general neither commutative nor continuous, even if all the maps~$f_n$ and the limit map~$f$ are continuous. 
For spaces homeomorphic to, or retracts of, closed Euclidean balls, one always has $X\setminus X^{\omega+1}\neq\emptyset$, 
since fixed points of continuous maps have ordinal degree~$\omega$ by Proposition~\ref{periodic_finite}. On the other hand, if~$f$ has no periodic points, one can construct examples in which every $x\in X$ satisfies $\omega<D(x)$ and even $D(x)=\mathfrak{D}(X,\{f\})$ for all~$x$ (see Example~\ref{irrational__}).

Our focus is on 
how dynamical behavior unfolds across ordinal scales and how attractors or recurrence phenomena arise at distinct ordinal levels. 
The transfinite level of interest is dictated by the sequential phenomenon under study and by the associated hierarchy of time scales. 
The recurrence notions developed for semigroup actions (e.g. in \cite{ellis2001topological,lalwani2022attractors}) are therefore not directly applicable in this setting.

\section{Transfinite orbits}
\label{sec4}
Up to now, from a finitely convergent sequence $\{f_n\}_{n\in\mathbb{N}}$ of self-maps of the compact metric space $X$, with continuous pointwise limit $f$, we defined the transfinite dynamical system $\left(X, \{f\}\right)$.
In this Section, we start looking at different types of possible transfinite orbits through some examples.  
All the examples in this section are set in the unit interval $\mathbb{I}=[0,1]$. Later on we will see examples in which the domain is a different compact space.

\begin{exa}\label{exa1}
    We provide a first, very simple example, showing the existence of transfinite iterations and, more specifically, of a transfinite cycle of order $\omega$. Assume $0<c<1$. Let us define the map $f:\mathbb{I}\circlearrowleft$ as 
    \begin{equation}
    \label{ex1_f}
    f(x)=\begin{cases}
        m x &\quad \text{if } x\in [0,c]\\
        1+\frac{1-m c}{1-c}(x-1) &\quad \text{if } x\in (c,1]
    \end{cases}
    \end{equation}
    where $0<m<1$. We point out that the function was defined so that $1$ is a fixed point because this property will be useful later in this section (in the proof of Theorem \ref{th_existence}). Take $z\in \mathbb{I}$ and let $U=(a,b)\subset [0,1]$ be an interval such that $z\in U$ and $m b<a$.  For any $n\in\mathbb{N}$ we set $z_n:=f^n (z) $. Let $\{a_n\}_{n\in\mathbb{N}}$ and $\{b_n\}_{n\in\mathbb{N}}$ be two sequences of real numbers that converge to $z$ such that 

\begin{itemize}
\item $\{a_n\}_{n\in\mathbb{N}}$ is increasing and $a_n\in (a,z)$ for all $n\in\mathbb{N}$;
\item $\{b_n\}_{n\in\mathbb{N}}$ is decreasing and $b_n\in (z,b)$ for all $n\in\mathbb{N}$.
\end{itemize}
For every $n\in\mathbb{N}$, we set $u_n=f^n (a_n)$, $v_n=f^n (b_n)$ and we indicate by $U_n=[u_n,v_n]$. Notice that $z_n\in U_n$ for all $n\in\mathbb{N}$. Pick $h\in (0,1)$ and we define the sequence of functions $f_n :\mathbb{I}\rightarrow \mathbb{I}$ as follows (see Fig.\ref{fig_ex_D}):
\begin{equation}
\label{ex1_fn}
f_n(x)=
\begin{cases}
f(x) &\quad \text{if } x\in \mathbb{I}\setminus U_n\\
 l_n(x-z_n)+h &\quad \text{if } x\in [u_n,z_n)\\
r_n(x-z_n)+h &\quad \text{if } x\in [z_n,v_n]\\
\end{cases}
\end{equation}
where
$$l_n=\frac{h- m u_n}{z_n-u_n},\qquad r_n=\frac{h- m v_n}{z_n-v_n}.$$
The maps $f_n$ are continuous for every $n\in \mathbb{N}$ and we have $f_n \dot{\longrightarrow} f$ as $n$ goes to $\infty$, with the limit map $f$ continuous as well. If $h=z$, for every $n\in\mathbb{N}$,
\begin{equation}
\label{ex_On}
    \mathcal{O}_n(z)=\{f(z),\ldots,f^n(z),z\}.
\end{equation}

Thus $z$ is an $(n+1)$-periodic point for every $f_n$. From \eqref{ex_On},
$\mathcal{O}_\infty(z)=\mathcal{O}(z)\cup \{z\} \ \text{and}\ \{f\}^\omega(z)=z.$

By Definition \ref{cycle} (with $x,y=z$ and $\beta,\eta=\omega$),  we have $ \{ f \}^{\omega \cdot k} (z) = z 
$ for all $k\in\mathbb{N}_0$. 
Therefore, $\{ \{f\}^\eta(z) : 1 \le \eta<\omega \}$ is a transfinite cycle of $(\{f\}^k(z))$-order $\omega+k$ for every $k\in\mathbb{N}_0$, and thus of order $\omega$. 
In general, for $h\in(0,1)$, we have $\{f\}^\omega(z)=h$ and $\{f\}^\omega(x)=\emptyset$ for any $x\in \mathbb{I}\setminus\mathcal{O}^\mathbb{Z}(z)$. Indeed, if we set $I^k_n=[f^k(a_n),f^k(b_n)]$, we have $\bigcap_{n=1}^\infty I^0_n=\{z\}$. 
Therefore $\bigcap_{n=1}^\infty \bigcup_{k=0}^\infty I^k_n=\{z\}\cup \mathcal{O}(z).$

If $\{f\}^\omega(x)\neq \emptyset$, then there exists $N\in\mathbb{N}$ such that $\mathcal{O}(x)\cap I_n^n\neq \emptyset$ for all $n>N$. Therefore $x\in\mathcal{O}^\mathbb{Z}(z)$. We remark that, if $h=f^k(z)$ for some $k\in\mathbb{N}$, then $[f]^{k+\omega}(z)=[f]^\omega(z)$ is not defined as a consequence of Theorem \ref{lem_group}, applied with $x=z$, $y=s=h$, $\beta=k$ and $\eta=\omega$, because we have $[f]^k(z)=h$ and $$[f]^\omega(h)=h\in\mathcal{O}_{[k+\omega]}(z)=\mathcal{O}_{\omega}(z)=\mathcal{O}(z).$$

Finally, for $h\in\mathcal{O}^{\mathbb{Z}}(z)$, we have  $D(x)=\omega^2$ if $x\in\mathcal{O}^{\mathbb{Z}}(z)$, while $D(x)=\omega$ for all other points. If instead $h\notin \mathcal{O}^{\mathbb{Z}}(z)$, we have $D(x)=\omega\cdot 2$ for $x\in \mathcal{O}^{\mathbb{Z}}(z)$ and $D(x)=\omega$ for all other points.

\begin{figure}[H]
\centering
\fbox{\includegraphics[width=10cm]{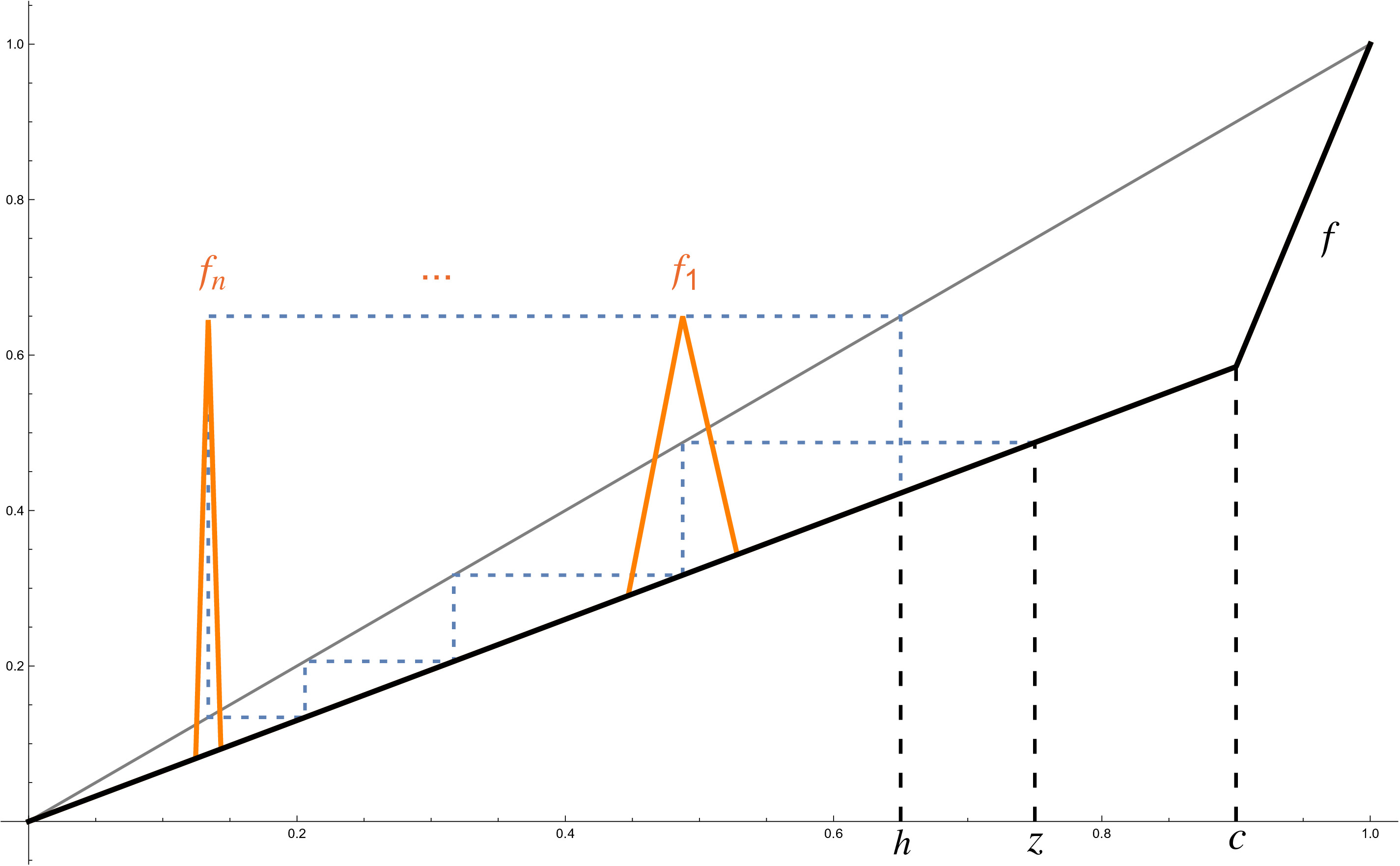}}
\caption{The system defined in Example \ref{exa1}.}\label{fig_ex_D}
\end{figure}
\end{exa}

\begin{exa}\label{exa2}
    We provide an example (see Fig. \ref{fig_cyc_w2}) where a transfinite cycle of order $\omega^2$ emerges.
    Let $\{z_n\}_{n\in\mathbb{N}_0}$ be a sequence of points in $\mathbb{I}$ such that $z_i\in(\frac{1}{2^{i+1}}, \frac{1}{2^{i}})$ for all $i\in\mathbb{N}_0$. 
    Set $0<m<1$. Let $h:\mathbb{I}\circlearrowleft$ be defined as $h(x)=m x$. We define a sequence $\{h_i\}_{i\in\mathbb{N}_0}$ of functions $h_i:[\frac{1}{2^{i+1}},\frac{1}{2^{i}})\rightarrow \mathbb{I}$, given by $h_i(x)= m (x-\frac{1}{2^{i+1}})+\frac{1}{2^{i+1}}$.
    Let $\{\epsilon_n\}_{n\in\mathbb{N}},\{\delta_n\}_{n\in\mathbb{N}}$ be two strictly decreasing sequences converging to $0$ as $n\to\infty$. 
    For every $n\in\mathbb{N}$ and $i=0,\ldots,n-1$ we set
\begin{align*}
&a^i_n=b^i_n-\delta_n \quad,\quad b^i_n=h_i^n (z_i)-\epsilon_n \quad,\quad  c^i_n=h_i^n (z_i)+\epsilon_n \quad,\quad d^i_n=c^i_n+\delta_n\\
&a^n_n=b^n_n-\delta_n \quad,\quad b^n_n=h^n(z_n)-\epsilon_n \ \ \ ,\quad c^n_n=h^n(z_n)+\epsilon_n \ \ ,\quad  d^n_n=c^n_n+\delta_n.\\
\end{align*}

We can choose the sequences $\{\epsilon_n\}_{n\in\mathbb{N}}$ and $\{\delta_n\}_{n\in\mathbb{N}}$ so that, for all $n\in\mathbb{N}$,
$$
\mathcal{O}_{h_i}(z_i)\cap [a^i_n,d^i_n]=h_i^n (z_i)\quad \text{for } i=0,\ldots,n-1$$
and
$$ \mathcal{O}^\mathbb{Z}(z_n)\cap [a^n_n,d^n_n]=h^n(z_n).$$
We can now define a sequence $\{f_n\}_{n \in \mathbb{N}}$ of functions $f_n:\mathbb{I}\rightarrow \mathbb{I}$  as follows:
\begin{equation*}
    f_n(x)=
\begin{cases}
\frac{z_0-h(a^n_n)}{b^n_n-a^n_n}(x-b^n_n)+z_0 &\quad \text{if } x\in [a^n_n,b^n_n)\\
z_0 &\quad \text{if } x\in [b^n_n,c^n_n]\\
\frac{z_0-h(d^n_n)}{c^n_n-d^n_n}(x-c^n_n)+z_0 &\quad \text{if } x\in (c^n_n,d^n_n]\\
h(x) &\quad \text{if } x\in [0,z_n] \setminus [a^n_n,d^n_n] \\
\frac{1}{2^n}+\frac{h(z_n)-\frac{1}{2^n}}{z_n-\frac{1}{2^n}}(x-\frac{1}{2^n}) &\quad \text{if } x\in (z_n,\frac{1}{2^n}]\\
g_i(x) &\quad \text{if } x\in (\frac{1}{2^{i+1}},\frac{1}{2^i}]\end{cases}
\end{equation*}
where, for every $i=0,\ldots,n-1$
\begin{equation*}
   \quad \ \ g_i(x)=
\begin{cases}
l^i_n(x-b^i_n)+z_{i+1} &\quad \text{if } x\in [a^i_n,b^i_n)\\
z_{i+1} &\quad \text{if } x\in [b^i_n,c^i_n]\\
r^i_n(x-c^i_n)+z_{i+1} &\quad \text{if } x\in (c^i_n,d^i_n]\\
h_{i}(x) &\quad \text{if } x\in (\frac{1}{2^{i+1}},z_i] \setminus [a^i_n,d^i_n] \\
\frac{1}{2^i}+\frac{h_{i}(z_i)-\frac{1}{2^i}}{z_i-\frac{1}{2^i}}(x-\frac{1}{2^i}) &\quad \text{if } x\in (z_i,\frac{1}{2^i}]
\end{cases}
\end{equation*}
\begin{equation*}
l^i_n=\frac{z_{i+1}-h_{i}(a^i_n)}{b^i_n-a^i_n},\qquad r^i_n=\frac{z_{i+1}-h_{i}( d^i_n)}{c^i_n-d^i_n}.
\end{equation*}
We define the map $f:\mathbb{I}\circlearrowleft$ as follows:
\begin{equation*}
    f(x)=
\begin{cases}
g_i(x) &\quad \text{if } x\in (\frac{1}{2^{i+1}},\frac{1}{2^i}]\text{ for some $i\in\mathbb{N}_0$}\\
0 &\quad \text{if } x=0
\end{cases}
\end{equation*}

By construction, we have $f_n \dot{\longrightarrow} f$ as $n$ goes to $\infty$. For every $n\in\mathbb{N}$, 
$$\mathcal{O}_n(z_0)=\{f(z_0),\ldots,f^n(z_0),z_1,f(z_1),\ldots,f^n(z_1),z_2,\ldots,z_n,f(z_n),\ldots,f^n(z_n),z_0\},$$ 
so that
$$\mathcal{O}_\infty(z_0)=\{f(z_0),f^2(z_0),\ldots,z_1,f(z_1),f^2(z_1),\ldots,z_2,\ldots,z_n,f(z_n),f^2(z_n),\ldots,z_{n+1},\ldots,z_0\}.$$ 
Hence we have $\{f\}^{\omega\cdot k}(z_0)=z_k$ for all $k\in\mathbb{N}$ and $\{f\}^{\omega^2}(z_0)=z_0$. Therefore, $$\mathcal{C}=\{\{f\}^\eta(z_0): 1\le \eta \le \omega^2\}$$ is a transfinite cycle of $\{f\}^k(z_j)$-order equal to $\omega^2+\omega\cdot j+k$ for every $k\in\mathbb{N}_0$, and therefore of order $\omega^2$. Finally, $D(z_0)=\omega^3$ and, for every $k,j\in\mathbb{N}$, $\mathcal{C}$ is $\omega^3$-saturated but not $(\omega^3+1)$-saturated. 

We remark that, in the examples above, the transfinite cycles are not only the set-theoretic $\liminf$ of finite cycles of increasing order for the maps $f_n$, as Theorem \ref{C_n} says, but they are actually the set-theoretic limit of finite cycles.

\begin{figure}[H] 
\centering
\subfigure[$f_1$]{\fbox{\includegraphics[width=6.9cm]{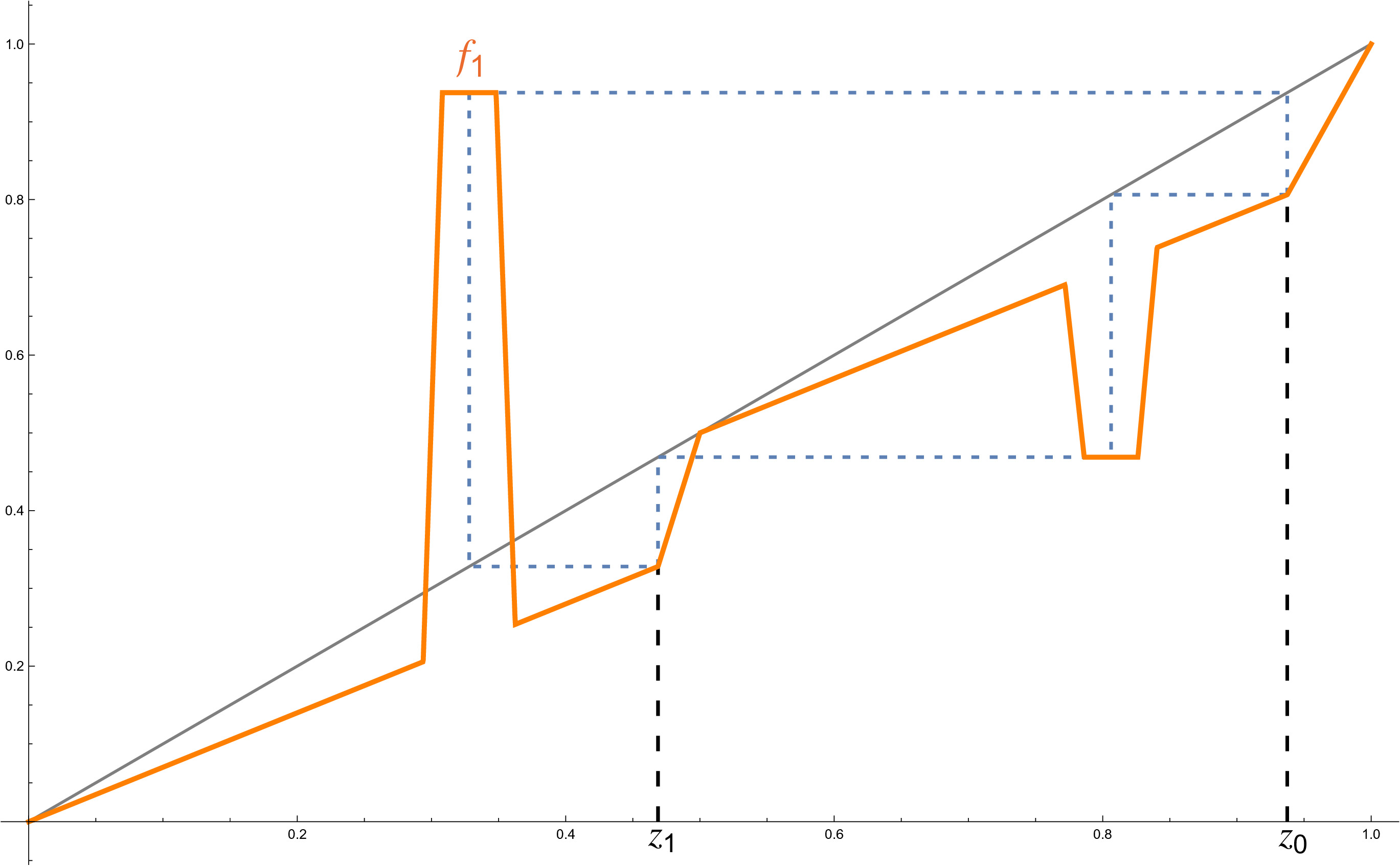}}}
\hspace{0.5mm}
\subfigure[$f_3$]
{\fbox{\includegraphics[width=6.9cm]{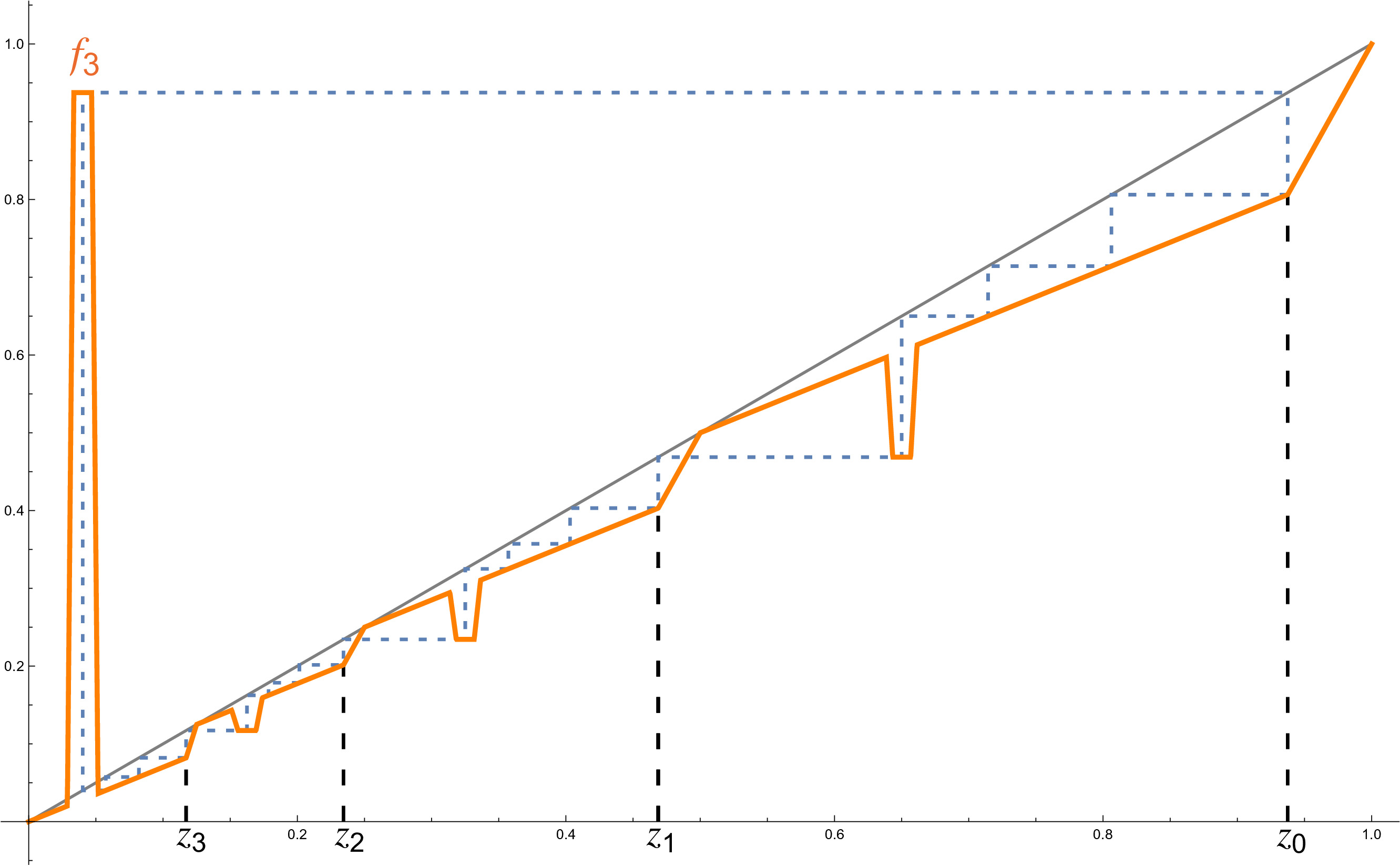}}}
\caption{The maps $f_1$ and $f_3$ from the sequence $\{f_n\}_{n\in\mathbb{N}}$ defined in Example \ref{exa2}.}
\label{fig_cyc_w2}
\end{figure}
\end{exa}

\begin{exa}\label{exa3}
For the kind of transfinite orbit we want to show in this example, we need, as the limit map $f$, a transitive map having also periodic points, so we pick the logistic map with parameter 4. 
Let thus $f:\mathbb{I}\circlearrowleft$ be defined as $f(x)=4x(1-x)$. 
Let $z\in (0,1)$ be a transitive point (that is we have $\overline{\mathcal{O}(z)}=\mathbb{I}$) and $a\in\mathbb{I}$ be a periodic point of order $t\in\mathbb{N}$. 
Pick $\epsilon_1>0$ and set $k_1=\min\{k\in\mathbb{N}: d(f^k(z),a)<\epsilon_1\}$. 
Define $\delta_1$ as $\delta_1=d(f^{k_1}(z),a)$ and set 

\[
a_1=a-\epsilon_1 \quad,\quad
b_1=a-\delta_1 \quad,\quad
c_1=a+\delta_1 \quad,\quad
d_1=a+\epsilon_1.
\]

Pick $\epsilon_2<\delta_1$. With the same procedure we can define $k_2,\delta_2,a_2,b_2,c_2$ and $d_2$. Proceeding inductively, for $\epsilon_m<\delta_{m-1}$ we set $k_m=\min\{k\in\mathbb{N}: d(f^k(z),a)<\epsilon_m\}$ and $\delta_m=d(f^{k_m}(z),a)$. Finally, we set $$a_m=a-\epsilon_m\quad,\quad b_m=a-\delta_m\quad,\quad c_m=a+\delta_m\quad,\quad d_m=a+\epsilon_m.$$
We can now define a sequence $\{f_n\}_{n\in\mathbb{N}}$ of functions $f_n:\mathbb{I}\circlearrowleft$ as follows (see Fig. \ref{fig_logic}):

\begin{equation*}
    f_n(x)=
    \begin{cases}
        \frac{f(a_n)-f(a)}{a_n-b_n}(x-b_n)+f(a) &\quad\text{if }x\in[a_n,b_n)\\
        f(a) &\quad\text{if }x\in [b_n,c_n]\\
        \frac{f(d_n)-f(a)}{d_n-c_n}(x-c_n)+f(a) &\quad\text{if }x\in(c_n,d_n]\\
        f(x) &\quad \text{if }x\in\mathbb{I}\setminus [a_n,d_n]   
    \end{cases}
\end{equation*}

The maps $f_n$ are continuous for every $n\in \mathbb{N}$ and we have $f_n \dot{\longrightarrow} f$ as $n$ goes to $\infty$, with the limit map $f$ continuous as well, so that we have a TDS $(X,\{f\})$. For every $n\in\mathbb{N}$, we have

\begin{equation}
    \mathcal{O}_n=\{f(z),\ldots,f^{k_n}(z),f(a),\ldots,f^{t-1}(a),a\}.
\end{equation}

so that $\mathcal{O}_\infty(z)=\{f(z),f^2(z),\ldots,f(a),\ldots,f^{t-1}(a),a\}$. Hence we have $\{f\}^\omega(z)=f(a)$. The point $z$ and every point in $\mathcal{O}(z)$ has a transfinite orbit consisting of a dense portion of length $\omega$ followed by a finite periodic cycle, so that their ordinal degree is $\omega\cdot 2$. 
\begin{figure}[H] 
\centering
\subfigure[$f_1$]{\fbox{\includegraphics[width=6.9cm]{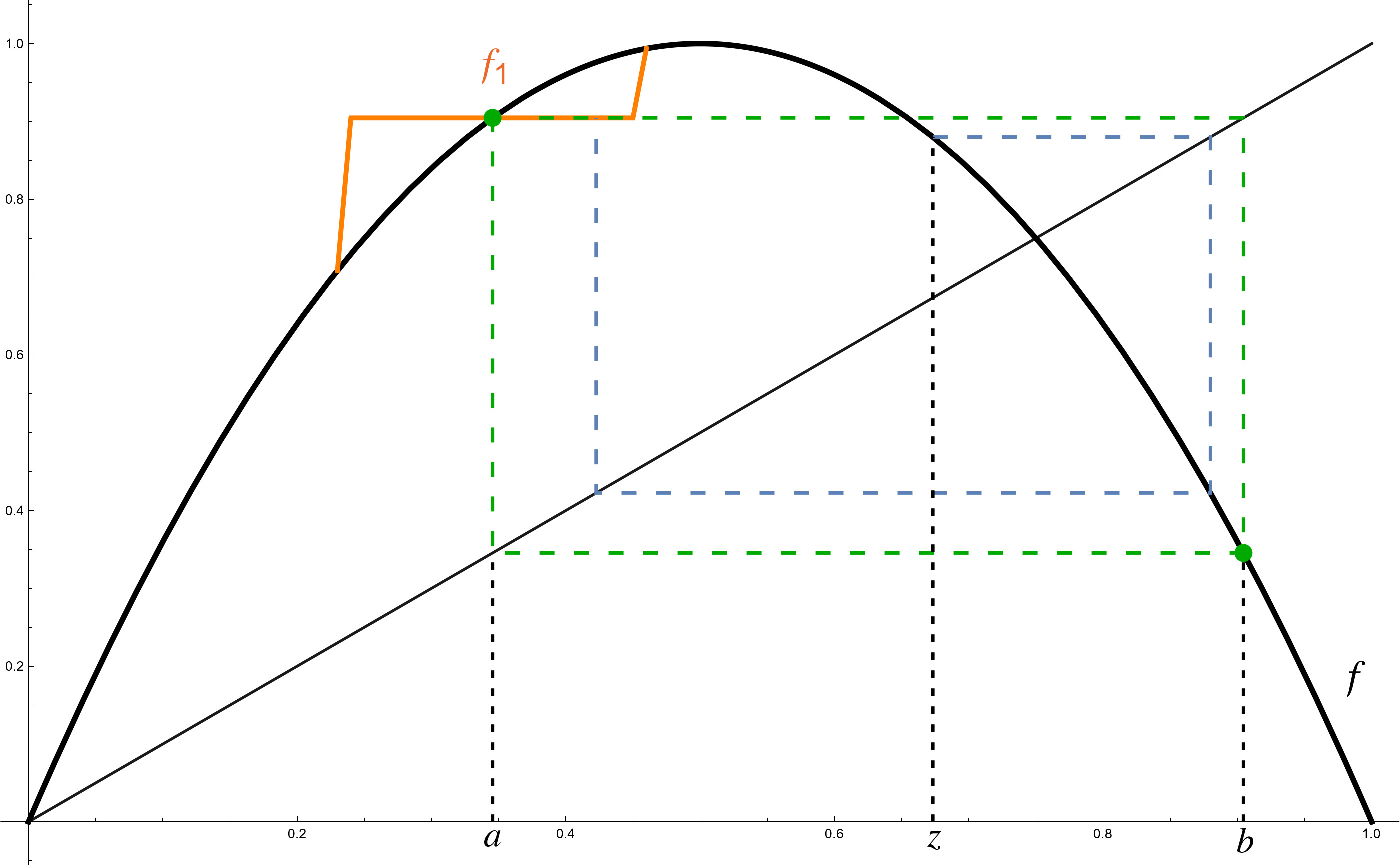}}}
\hspace{0.5mm}
\subfigure[$f_4$]
{\fbox{\includegraphics[width=6.9cm]{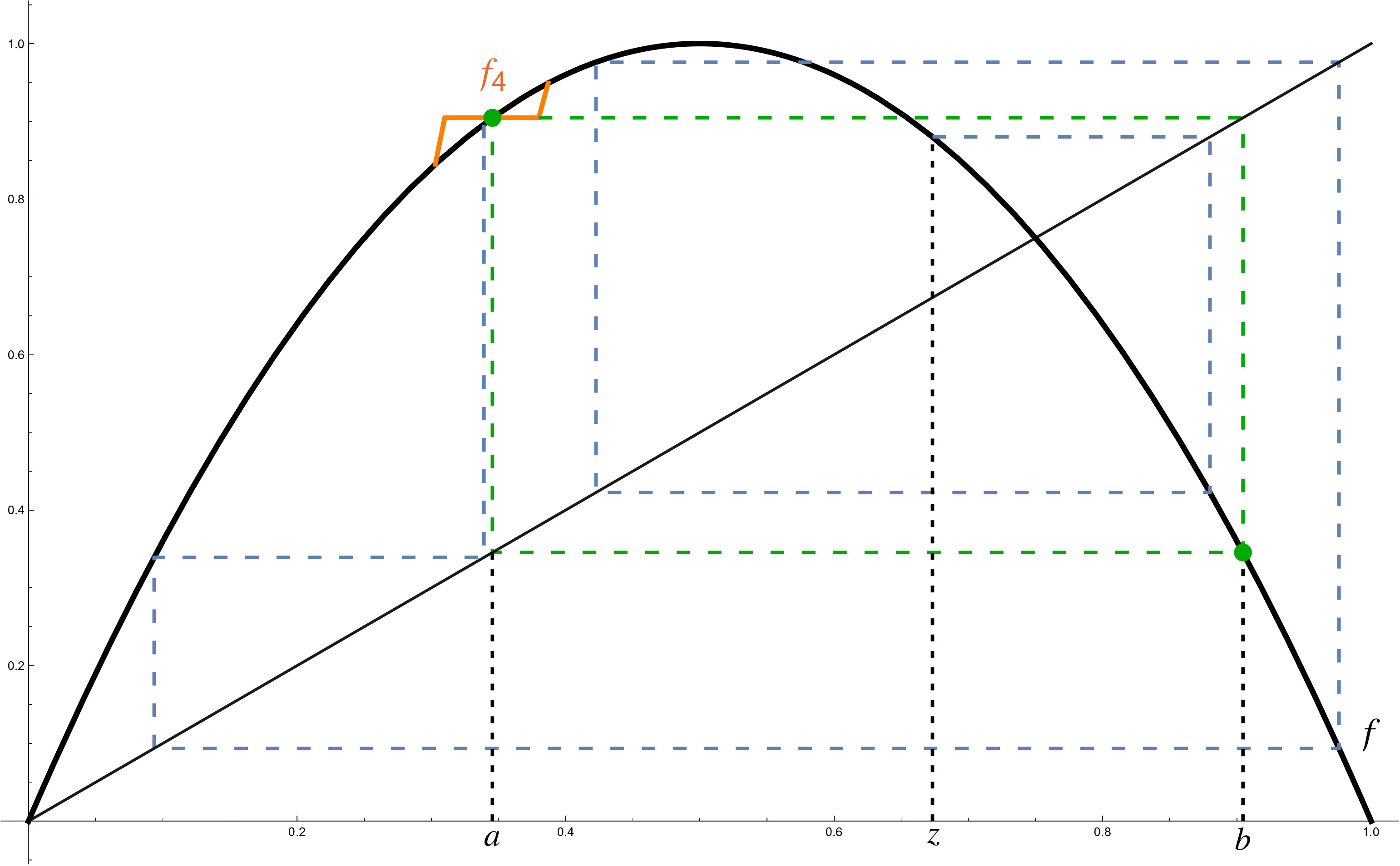}}}
\caption{The maps $f_1$ and $f_4$ from the sequence $\{f_n\}_n$ defined in Example \ref{exa3}}
\label{fig_logic}
\end{figure}
\end{exa}

The following two examples concern the ways in which the transfinite iterations of a point $x$ can cease to exist at a certain countable ordinal level. The most trivial case occurs when there exists $N$ such that $f^k_n(x)=f^k(x)$ for every $k\in\mathbb{N}$ and every $n > N$. This implies that the hierarchy of sets in Eq. \eqref{hierarchy} collapses, and in particular $\mathcal{O}_\infty(x)=\mathcal{O}(x)=\{\mathcal{O}\}(x)=\mathcal{O}_{\{\omega\}}(x)$, so that the finite iterations of $x$ exhaust all the points of $\mathcal{O}_\infty(x)$. Let us see two more interesting cases.

\begin{exa}\label{exa4}
In this example (see Fig. \ref{fig_A2A3}, left), there exists a point $x$ such that the relation $<_\infty$ is not a total order on $\mathcal{O}_\infty(x)$.  Let $z, y_1, y_2\in (0,1)$ be such that $z<y_2<y_1$. Set $0<m<1$ and define $f:\mathbb{I}\rightarrow \mathbb{I}$ as 

\begin{equation}
\label{ex_f}
f(x)=
\begin{cases}
\frac{y_1-mz}{y_2-z}(x-y_2)+y_1 \quad &\text{if }x\in (z,y_2]\\
\frac{y_1-y_2}{y_2-y_1}(x-y_2)+y_1 \quad& \text{if }x\in (y_2,y_1]\\
\frac{m - y_2}{1 - y_1}(x - 1)+m \quad& \text{if }x\in (y_1,1]\\
m x \quad &\text{otherwise} 
\end{cases}
\end{equation}

Let $\{\epsilon_n\}_{n\in\mathbb{N}},\{\delta_n\}_{n\in\mathbb{N}}$ be two strictly decreasing, vanishing sequences. For every $n\in\mathbb{N}$, we set:

\begin{equation*}
a_n=b_n-\delta_n \quad ,  \quad b_n=m^n z-\epsilon_n \quad,\quad c_n=m^n z+\epsilon_n \quad, \quad d_n=c_n+\delta_n.
\end{equation*}

We can choose the sequences $\{\epsilon_n\}_{n\in\mathbb{N}}$ and $\{\delta_n\}_{n\in\mathbb{N}}$ so that, for all $n\in\mathbb{N}$,

\begin{equation*}
\mathcal{O}_n(z)\cap [a_n,d_n]=f^n (z).
\end{equation*}

We can now define a sequence $\{f_n\}_{n \in \mathbb{N}}$ of functions $f_n:\mathbb{I}\rightarrow \mathbb{I}$  as follows:

\begin{align*}
\text{For odd $n$:\ \ }&f_n(x)=
\begin{cases}
\frac{y_1-m a_n}{b_n-a_n}(x-b_n)+y_1 &\quad \text{if } x\in [a_n,b_n)\\
y_1 &\quad \text{if } x\in [b_n,c_n]\\
\frac{y_1-m d_n}{c_n-d_n}(x-c_n)+y_1 &\quad \text{if } x\in (c_n,d_n]\\
f(x) &\quad \text{otherwise} 
\end{cases}\\
\text{For even $n$:\ \ }&f_n(x)=
\begin{cases}
\frac{y_2-m a_n}{b_n-a_n}(x-b_n)+y_2 &\quad \text{if } x\in [a_n,b_n)\\
y_2 &\quad \text{if } x\in [b_n,c_n]\\
\frac{y_2-m d_n}{c_n-d_n}(x-c_n)+y_2 &\quad \text{if } x\in (c_n,d_n]\\
f(x) &\quad \text{otherwise} 
\end{cases}
\end{align*}

We have $f_n \dot{\longrightarrow} f$ as $n$ goes to $\infty$. For every $n\in\mathbb{N}$, $\mathcal{O}_n(z)=\{f(z),\ldots,f^n(z),y_1,y_2\}$. Thus $\mathcal{O}_\infty(z)=\{f(z),f^2(z)\ldots,y_1,y_2\}$. We have $y_1<_n y_2$ for odd $n$ and $y_2<_n y_1$ for even $n$ and then $y_1$ and $y_2$ are not comparable with respect to $<_\infty$, so that $<_\infty$ is not a total order on $\mathcal{O}_\infty(z)$. This means that $\{f\}^{\omega}(z)$ does not exist, since $\mathcal{O}_\infty(z)\setminus \mathcal{O}_{[\omega]}(z)$ does not have a least element.
\end{exa}

\begin{exa}\label{ex_order}
    We provide now an example (see Fig. \ref{fig_A2A3}, right) where a certain transfinite iteration of a point $z$ is not defined because $<_{z,\infty}$, while being a total order, it is not a well-ordering. Pick $z\in (0,1)$ and take $0<m<1$. Let $f:\mathbb{I}\rightarrow \mathbb{I}$ be defined as
    
\begin{equation*}
f(x)=
\begin{cases}
m x \quad &\text{if } x\in [0,z]\\
\frac{1-m z}{1-z}(x-1)+1 \quad &\text{if } x\in (z,1]
\end{cases}
\end{equation*}

Let $\{y_n\}_{n\in\mathbb{N}}$ be a sequence of points such that $y_1=f^{-1}(z)$ and $y_{n+1}=f^{-1}(y_n)$ for every $n\in\mathbb{N}$. Let $\{\epsilon_n\}_{n\in\mathbb{N}},\{\delta_n\}_{n\in\mathbb{N}}$ be two strictly decreasing sequences converging to $0$ as $n\to\infty$. For every $n\in\mathbb{N}$, we set

\begin{equation*}
a_n=b_n-\delta_n \quad ,  \quad b_n=m^n z-\epsilon_n \quad,\quad c_n=m^n z+\epsilon_n \quad, \quad d_n=c_n+\delta_n.
\end{equation*}

We can choose the sequences $\{\epsilon_n\}_{n\in\mathbb{N}}$ and $\{\delta_n\}_{n\in\mathbb{N}}$ so that, for all $n\in\mathbb{N}$

\begin{equation*}
\mathcal{O}_n(z)\cap [a_n,d_n]=f^n (z).
\end{equation*}

Let us now define a sequence $\{f_n\}_{n \in \mathbb{N}}$ of functions $f_n:\mathbb{I}\rightarrow \mathbb{I}$  as follows:

\begin{equation*}
f_n(x)=
\begin{cases}
\frac{y_n-m a_n}{b_n-a_n}(x-b_n)+y_n &\quad \text{if } x\in [a_n,b_n)\\
y_n &\quad \text{if } x\in [b_n,c_n]\\
\frac{y_n-m d_n}{c_n-d_n}(x-c_n)+y_n &\quad \text{if } x\in (c_n,d_n]\\
f(x) &\quad \text{otherwise} 
\end{cases}
\end{equation*}

\vspace{0.2cm}

We have, for every $n\in\mathbb{N}$, $$\mathcal{O}_n(z)=\{ f(z),\ldots,f^n(z),y_n,y_{n-1},\ldots y_1,z\}$$ so that 
$$\mathcal{O}_\infty(z)=\{f(z),f^2(z),\ldots,y_n,y_{n-1},\ldots,y_1,z\}.$$
Here $\{y_n\}_{n\in\mathbb{N}}$ is an infinite decreasing sequence with respect to $<_\infty$, so that $\mathcal{O}_\infty(z)\setminus \mathcal{O}_{[\omega]}(z)$ does not have a least element, and therefore $\{f\}^\omega(z)$ does not exist. 

\begin{figure}[H] 
\centering
\subfigure[]
{\fbox{\includegraphics[width=6.9cm]{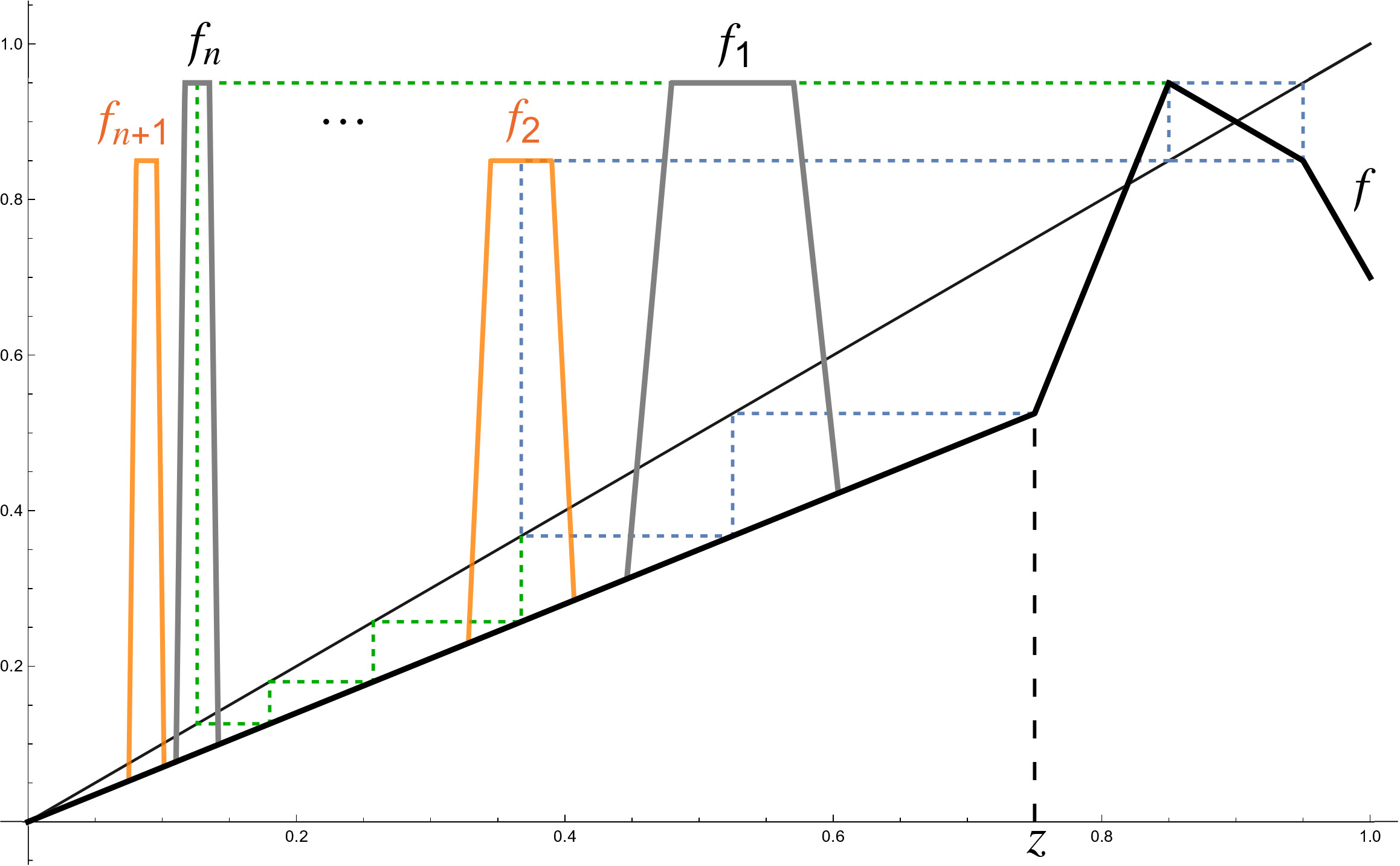}}}
\hspace{0.5mm}
\subfigure[]
{\fbox{\includegraphics[width=6.9cm]{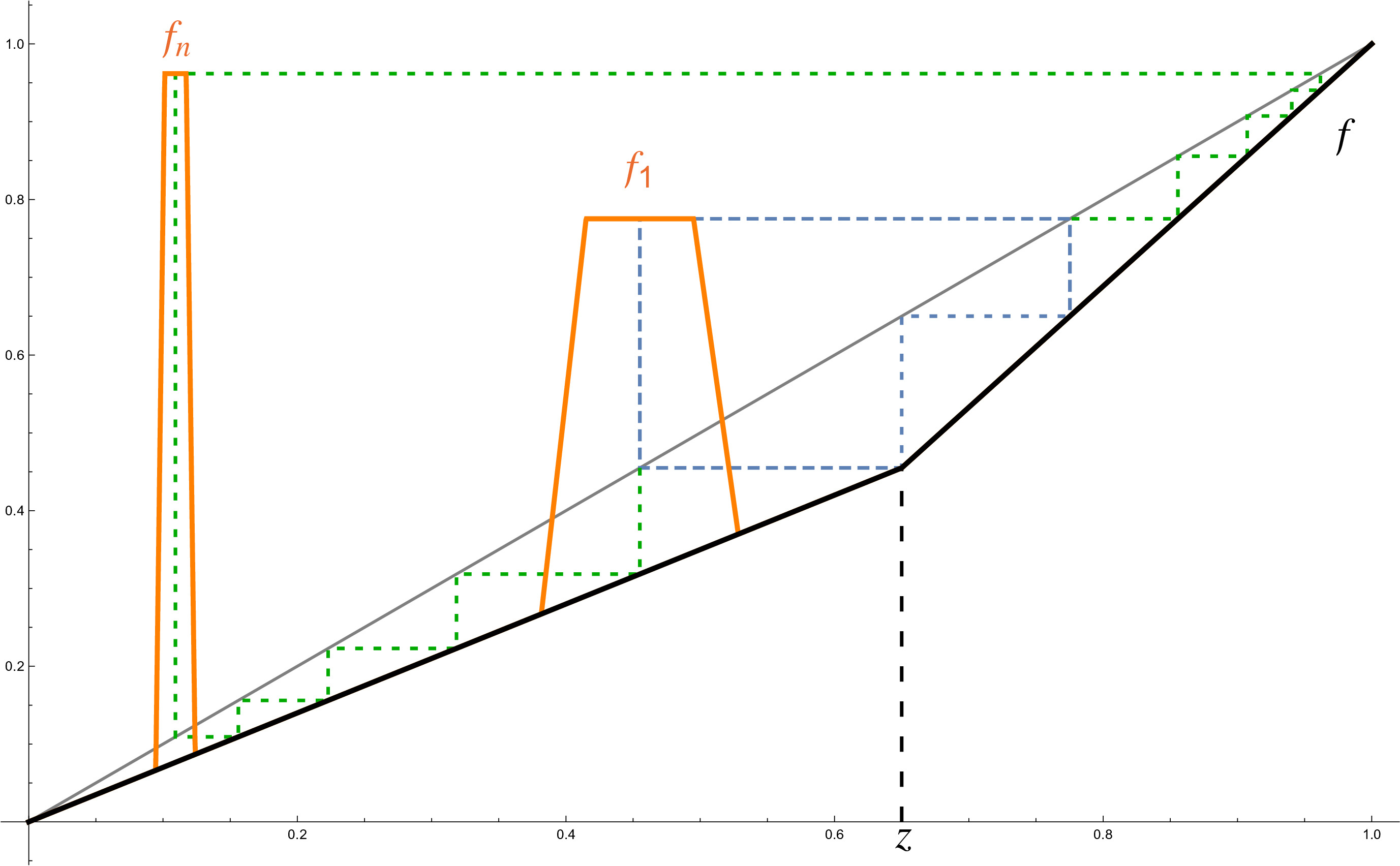}}}
\caption{(a): The system defined in Example \ref{exa4}: the order induced by $<_{_\infty}$ is not a total order on $\mathcal{O}_\infty$. (b): The system defined in Example \ref{ex_order}: the order induced by $<_{_\infty}$ on $\mathcal{O}_\infty$ is not a well-ordering.}
\label{fig_A2A3}
\end{figure}
\end{exa}

A natural question about transfinite orbits concerns in which cases the iteration of order $\omega$ of $x$ has something to do with the $\omega$-limit of $x$ with respect to $f$:
\begin{equation}
\label{def_omegalimit}
\omega_f(x)=\bigcap_{n=0}^\infty\overline{\bigcup_{k>n}f^k(x)}.
\end{equation}
The two objects can be totally unrelated, of course. For instance, in Example \ref{exa1}, we have $\omega_f(z)=\{0\}$ while $\{f\}^\omega(z)>0$. However, under some (strong) regularity assumptions, there is a result covering cases that can be seen as ``trivial" since the transfinite dynamics does not provide any further richness than the $\omega$-limit with respect to the limit map $f$. 

\begin{defi}\label{equicontinuity_}
We say that a sequence of maps $\{f_n\}_{n\in\mathbb{N}}$ is \emph{equicontinuous} at $x\in X$ if, for every $\epsilon>0$, there exist $N=N(\epsilon,x)\in\mathbb{N}$ and $\delta=\delta(\epsilon,x)>0$ such that, for every $n>N$,
\begin{equation*}
   d(x,y)<\delta \implies d(f_n(x),f_n(y))<\epsilon. 
\end{equation*}
$\{f_n\}_{n\in\mathbb{N}}$ is said \emph{strongly equicontinuous} at $x$ if $\{f^k_n\}_{n\in\mathbb{N}}$ is equicontinuous at $x$ for all $k\in\mathbb{N}$.
\end{defi}

\begin{prop}
Let $(X,\{f\})$ be a TDS and assume that $\{f_n\}_{n\in\mathbb{N}}$ is strongly equicontinuous at $x\in X$. If $\omega_f(x)=\{z\}$ and $\{f\}^\omega(x)=y$, then $y=z$.
\end{prop}

\begin{proof}
Suppose that $\omega_f(x)=z\in X$. Assume that $\{f\}^\omega(x)=y\notin \mathcal{O}(x)$ and that $y\neq z$. Therefore $d(y,z)>\epsilon>0$. Since $\{f_n\}_{n\in\mathbb{N}}$ is strongly uniformly equicontinuous in $x$, there exist $N_{\epsilon,x}\in\mathbb{N}$ and $\delta_{\epsilon,x}>0$ such that
\[
   d(x,u)<\delta_{\epsilon,x} \implies d(f_n^k(x),f_n^k(u))<\epsilon \quad \forall k\in\mathbb{N},\ \forall n>N_{\epsilon,x}. 
\]
Since $z=\omega_f(x)$, there exists $K\in\mathbb{N}$ such that 
\begin{equation}\label{3}
d(f^K(x),z)<\delta_{\epsilon,x}.
\end{equation}
By Lemma \ref{nbound_seq}, the sequence $\{k^{\omega,x}_n\}_{n>N^\omega(x)}$ is unbounded, and in particular 
\begin{equation}\label{2}
f_n^{k^\omega_n}(x)=y \qquad \forall n> N^\omega.
\end{equation}
Thus, there exists $N_1>N^\omega$ such that $k^\omega_n>K$ for all $n >N_1$. Since $f_n\dot{\longrightarrow}f$, there exists $N_2\in\mathbb{N}$ such that
\begin{equation}\label{4}
f_n^k(x)=f^k(x)\qquad \forall k\le K,\ \forall n> N_2.
\end{equation}
 Since $\omega_f(x)$ is $f$-invariant, there exists $N_3\in\mathbb{N}$ such that
\begin{equation}\label{5}
f_n(z)=f(z)=z \qquad \forall n> N_3.
\end{equation}
Set $N:=\max\{N_{\epsilon,x},N_1,N_2,N_3\}$. By \eqref{3} and \eqref{4}, we have
$d(f_n^K(x),z)=d(f^K(x),z)<\delta_{\epsilon,x}$ for all $n>N$.
In particular, since $y\notin \mathcal{O}(x)$, we have  $f_n^k(x)\neq y$ for all $k\le K$ and $n>N$. By \eqref{2} and \eqref{5} and since $\{f_n\}_{n\in\mathbb{N}}$ is strongly uniformly equicontinuous at $x$, for all $n>N$ we have:
$$
d\left(f_n^{K+(k^\omega_n-K)}(x),f_n^{k^\omega_n-K}(z)\right)=d(y,z)<\epsilon,
$$
which is a contradiction. Suppose that $\{f\}^\omega(x)\in \mathcal{O}(x)$, then there exists $k\in\mathbb{N}$ such that $f^k(x)=\{f\}^\omega(x)$. Notice that $k$ is unique, indeed, otherwise, we have that $x$ is a pre-periodic point and, by Proposition \ref{periodic_finite}, it follows that $\{f\}^\omega(x)=\emptyset$. Therefore, we can apply the previous argument replacing the point $x$ by the point $f^k(x)$.
\end{proof}
We conclude this section by showing that there exist TDSs, set on the interval and sequentially continuous, in which arbitrarily large countable ordinal iterations exist. 

\begin{thm}\label{th_existence}
    For every countable ordinal $\lambda$ there is a sequentially continuous TDS $(\mathbb{I},\{f\})$ 
    such that $\mathfrak{D}(\mathbb{I},\{f\})\ge \lambda$.
\end{thm}
\begin{proof}

    Let $\lambda$ be a countable ordinal. We will build, by transfinite induction up to $\lambda$, a countable collection 
    \[
    \Big \{(\mathbb{I},\{f_\beta\})\Big \}_{\omega\le\beta\le\lambda}
    \]
    of TDSs set in the unit interval such that, for every $\beta$, there exists a point $x_\beta\in\mathbb{I}$ such that $\{f_\beta\}^\beta(x_\beta)\ne\emptyset$.
    
    We start by defining the system $(\mathbb{I},\{f_{\omega,n}\}_{n\in\mathbb{N}})$ as equal to the TDS of Example \ref{exa1}, that is $f_\omega\equiv f$ as defined in Eq. \eqref{ex1_f} and, for every $n\in\mathbb{N}$, $f_{\omega,n}\equiv f_n$ as defined in Eq. \eqref{ex1_fn}.
    
    Using again the notation of Example \ref{exa1}, if we take $x_\omega=z$ and $h\in \mathbb{I}\setminus \mathcal{O}^\mathbb{Z}(x_\omega)$, we have 
    $$\{f_\omega\}^\omega(x_\omega)=[f_\omega]^\omega(x_\omega)=h.$$ 
    The starting point of the induction procedure is therefore a system $(\mathbb{I},\{f_\omega\})$ in which the following set of properties, that we will call $\boldsymbol{P}_\omega$, hold: 
    \begin{itemize}
        \item[$\boldsymbol{P}_\omega$-1.] $0$ and $1$ are fixed points for $f_\omega$;
        \item [$\boldsymbol{P}_\omega$-2.] there exists $x_\omega\in\mathbb{I}$ such that $[f_\omega]^\omega(x_\omega)=\{f_\omega\}^\omega(x_\omega)\neq \emptyset$;
        \item [$\boldsymbol{P}_\omega$-3.] the transfinite orbit of $x_\omega$ is $(\omega\cdot 2)$-open.
    \end{itemize}
    
Let us now assume that, for all $\alpha$ such that $\omega\le \alpha<\lambda$, there exists $(\mathbb{I},\{f_\alpha\})$ such that properties $\boldsymbol{P}_\alpha$ hold, that is: 
    \begin{enumerate}
    \item[$\boldsymbol{P}_\alpha$-1.] $0$ and $1$ are fixed points for $f_\alpha$; 
    \item[$\boldsymbol{P}_\alpha$-2.] there exists $x_\alpha\in\mathbb{I}$ such that $[f_\alpha]^\alpha(x_\alpha)=\{f_\alpha\}^\alpha(x_\alpha)\neq \emptyset$;
    \item[$\boldsymbol{P}_\alpha$-3.] the transfinite orbit of $x_\alpha$ is $(\alpha+\omega)$-open.        
    \end{enumerate} 
    We want to prove the existence of a system $(\mathbb{I},\{f_{\lambda}\})$ having properties $\boldsymbol{P}_\lambda$.
    If $\lambda$ is a successor ordinal, that is $\lambda=\alpha+1$, then there exist a sequentially continuous TDS $(\mathbb{I},\{f_\alpha\})$ and $x_\alpha\in \mathbb{I}$ having properties $\boldsymbol{P}_\alpha$. 
    By Lemma \ref{lem_exist_iter}, we have $\{f_\alpha\}^{\alpha+1}(x_\alpha)=[f_\alpha]^{\alpha+1}(x_\alpha)=f([f_\alpha]^\alpha(x_\alpha)),$ and since $\alpha+\omega=\alpha+1+\omega$, the transfinite orbit of $x_\alpha$ is $(\alpha+1+\omega)$-open. 
    Hence, we can take $\{f_\lambda\}$ coinciding with $\{f_\alpha\}$ and $x_\lambda=x_\alpha$.
    
    Suppose now that $\lambda$ is a countable limit ordinal. Then $\lambda=\sup_{j\in\mathbb{N}} \beta_j$ where $\beta_j<\omega_1$ for all $j\in\mathbb{N}$. We consider two cases, depending on $\lambda$ being of type $\lambda=\beta+\omega$ or not. 
    \begin{enumerate}
        \item Suppose $\lambda=\beta+\omega$ for some $\omega\le \beta <\omega_1$. We can assume  $\beta_j=\beta+j$ for every $j\in\mathbb{N}$. By the inductive hypothesis, there exist $(\mathbb{I},\{f_\beta\})$ and a point $x_\beta\in \mathbb{I}$ with the properties $\boldsymbol{P}_\beta$.  
    
    For every $n\in\mathbb{N}$, set $\Tilde{f}_{\beta,n}(x):=\frac 1 2 f_{\beta,n}(2x)$ and $\Tilde{f}_\beta(x):=\frac 1 2 f_\beta(2x)$. Let $z\in (\frac 1 2,1)$, and let $g:[\frac 1 2,1]\circlearrowleft$ be a continuous function with $\frac 1 2$ and $1$ fixed points and such that the orbit of $z$ is $\omega$-open, that is $z$ is not a periodic or ultimately periodic point. 
    Consider the system $([0,\frac 1 2],\{\Tilde{f}_\beta\})$, set $\Tilde{x}_\beta:=x_\beta/2$, $y_\beta:=\{\Tilde{f}_\beta\}^{\beta}(\Tilde{x}_\beta)=[\Tilde{f}_\beta]^{\beta}(\Tilde{x}_\beta)$ and, for every $j\in\mathbb{N}$, $$y_j:=\{\Tilde{f}_\beta\}^{\beta_j}(\Tilde{x}_\beta)=\{\Tilde{f}_\beta\}^j\left(\{\Tilde{f}_\beta\}^{\beta}(\Tilde{x}_\beta)\right) =\Tilde{f}_\beta^j(y_\beta).$$ 
    Since $\beta\ge \omega$, by Lemma \ref{nbound_seq} the sequence $\{k^\beta_n\}_{n>N^\beta}$ is unbounded and $\Tilde{f}_{\beta,n}^{k^\beta_n}(\Tilde{x}_\beta)=y_\beta$ for all $n>N^\beta$. Since the transfinite orbit of $\tilde{x}_\beta$ is $(\beta+\omega)$-open, we have that the points $\{y_j\}_{j\in\mathbb{N}}$ are all distinct, so that we can extract from them a a strictly monotone subsequence $\{y_{j_k}\}_{k\in\mathbb{N}}$. By the finite convergence of $\Tilde{f}_{\beta,n}$ to $\Tilde{f}_\beta$, for every $k\in\mathbb{N}$ there exists $N_k>N^\beta$ such that, for every $i\le k$, we have 
    \[
    \Tilde{f}_{\beta,n}^{k^\beta_n+j_i}(\Tilde{x}_\beta)=\Tilde{f}_\beta^{j_i}(y_\beta)=y_{j_i} \qquad \ \forall n\ge N_k.
    \]
    For all $k\in\mathbb{N}$ and for all $n\in\mathbb{N}$, we take $\epsilon_{k,n}>0$ so small to have, for all $h<k^\beta_n+j_k$, 
    \[
    \Tilde{f}_{\beta,n}^h(\Tilde{x}_\beta)\cap B_{\epsilon_{k,n}}(y_{j_k})=\emptyset,
    \]
    where $B_{\epsilon_{k,n}}(y_{j_k})=(y_{j_k}-\epsilon_{k,n}, y_{j_k}+\epsilon_{k,n})$.
    Let us define a sequence of continuous functions $f_{\lambda,n}:\mathbb{I} \rightarrow \mathbb{I}$. Set $f_{\lambda,n}(x):=g(x)$ for all $x\in [\frac 1 2,1]$ and for all $n\in\mathbb{N}$. If $x\in [0,\frac 1 2)$, we set $f_{\lambda,n}(x):=\Tilde{f}_{\beta,n}(x)$ for all $n< N_1$. Let $c_n=\max\{k\le n : N_k\le n\}$. 
    Set $a_n:=y_{j_{c_n}}-\epsilon_{c_n,n}$ and $b_n:=y_{j_{c_n}}+\epsilon_{c_n,n}$. Then, for $n\ge N_1$, we set

\begin{equation*}
    f_{\lambda,n}(x)=
\begin{cases}
    \Tilde{f}_{\beta,n}(x) &\quad \text{if }x\in [0,\frac 1 2)\setminus (a_n,b_n)\\
    \frac{z-\Tilde{f}_{\beta,n}(a_n)}{y_{j_{c_n}}-a_n}(x-y_{j_{c_n}})+z  &\quad \text{if }x\in (a_n,y_{j_{c_n}}]\\
    \frac{z-\Tilde{f}_{\beta,n}(b_n)}{y_{j_{c_n}}-b_n}(x-y_{j_{c_n}})+z &\quad \text{if }x\in (y_{j_{c_n}},b_n).
\end{cases}
\end{equation*}

 Notice that $\{c_n\}_{n\ge N_1}$ is increasing and not ultimately constant. Let us define $f_\lambda:\mathbb{I}\rightarrow \mathbb{I}$ as  
 \begin{equation*}
     f_\lambda(x)=
     \begin{cases}
         \Tilde{f}_\beta(x) &\quad\text{if }x\in[0,\frac 1 2 )\\
         g(x) &\quad\text{if }x\in[\frac 1 2 ,1]
     \end{cases}
 \end{equation*}
 
By construction $f_\lambda$ is a continuous function. Since $\{y_{j_k}\}_{k\in\mathbb{N}}$ is a strictly monotonic sequence and $\epsilon_{k,n}\to 0$ as $n\to \infty$, it follows that $f_{\lambda,n}\dot{\longrightarrow} f_\lambda$. 
Since the maps $f_{\lambda,n}$ were also continuous for every $n\in\mathbb{N}$, we have that $(\mathbb{I},\{f_\lambda\})$ is a sequentially continuous TDS. 
We want to show that the system verifies the properties $\boldsymbol{P}_\lambda$. 
By construction, $0$ and $1$ are fixed points for $f_\lambda$. For all $n\ge N_1$ we have that $z\in \mathcal{O}_n(\Tilde{x}_\beta)$, which implies $z\in\mathcal{O}_\infty(\Tilde{x}_\beta)$. 
Moreover, for all $n\ge N_1$ we have that $z>_n y_i$ for all $i< j_{c_n}$. Since the sequence $\{c_n\}_{n\ge N^1}$ is increasing and not ultimately constant, it follows that $z>_\infty u$ for all $u\in \mathcal{O}_{[\beta+\omega]}(\Tilde{x}_\beta)$. Since $\mathcal{O}_{[\beta+\omega]}(\Tilde{x}_\beta)=\mathcal{O}_{[\beta+1]}(\Tilde{x}_\beta)\cup(\cup_{j=1}^\infty y_j)$, there is no $v\in\mathcal{O}_\infty(\Tilde{x}_\beta)$ such that, for all $u\in \mathcal{O}_{[\beta+\omega]}(\Tilde{x}_\beta)$, $u <_\infty v <_\infty z$. Therefore, $z=[f_\lambda]^\lambda(\Tilde{x}_\beta)=\{f_\lambda\}^\lambda(\Tilde{x}_\beta)$. Finally, since the orbit of $z$ is $\omega$-open, it follows that the transfinite orbit of $\Tilde{x}_\beta$ is $(\lambda+\omega)$-open.

\item Suppose that $\lambda$ is a countable limit ordinal which is not of the form $\beta+\omega$. We can thus assume that $\{\beta_j\}_{j\in\mathbb{N}}$ is a sequence of limit ordinals such that $\lambda=\sup_{j\in\mathbb{N}}\beta_j$. By the inductive hypothesis, for every $j\in\mathbb{N}$, there exist a sequentially continuous TDS $( \mathbb{I},\{f_{\beta_j,n}\}_{n\in\mathbb{N}})$ and a point $x_j\in\mathbb{I}$ with the properties $\boldsymbol{P}_{\beta_j}$.
 
 For every $j\in\mathbb{N}$, we define the maps
 \[
 \Tilde{f}_{\beta_j,n}(x):=\frac{1}{2^{j+1}}+\frac{1}{2^{j+1}} f_{\beta_j,n}\left(2^{j+1}\left(x-\frac{1}{2^j+1}\right)\right),
 \]
 and its pointwise limit (notice that the convergence is finite): 
 \[
 \Tilde{f}_{\beta_j}(x):=\frac{1}{2^{j+1}}+\frac{1}{2^{j+1}} f_{\beta_j}\left(2^{j+1} \left(x-\frac{1}{2^{j+1}}\right)\right).
 \]
  For every $j\in\mathbb{N}$, consider the sequentially continuous TDS $([\frac{1}{2^{j+1}},\frac{1}{2^j}],\{\Tilde{f}_{\beta_j,n}\}_{n\in\mathbb{N}})$ and set $\Tilde{x}_j:= \frac{x_j}{2^{j+1}}+\frac{1}{2^{j+1}}$, $y_j:=\{\Tilde{f}_{\beta_j}\}^{\beta_j}(\Tilde{x}_j)=[\Tilde{f}_{\beta_j}]^{\beta_j}(\Tilde{x}_j)$.  By Lemma \ref{nbound_seq}, for every $j\in\mathbb{N}$, the sequence $\{k_n^{\beta_j}\}_{n>N^{\beta_j}}$ is unbounded. We set $N^j:=N^{\beta_j}$ and $k_n^j:=k_n^{\beta_j}$, so that  $$\Tilde{f}_{\beta_j,n}^{k^j_n}(\Tilde{x}_j)=y_j \quad \forall n>N^j.$$ 
  We can assume that the sequence $\{\Tilde{f}_{\beta_j,n}\}_{n\in\mathbb{N}}$ is such that $\Tilde{f}_{\beta_j,n}^{k_n^j}(\Tilde{x}_j)\neq \Tilde{f}_{\beta_j}(\Tilde{f}_{\beta_j,n}^{k_n^j-1}(\Tilde{x}_j))$ for all $n>N^j$. Set $z_{j,n}:=\Tilde{f}_{\beta_j,n}^{k_n^j-1}(\Tilde{x}_j)$. Then there exists $\epsilon_{j,n}>0$ such that 
 \begin{equation}\label{cond_fc}
 \Tilde{f}_{\beta_j,n}(x)\neq \Tilde{f}_{\beta_j}(x) \quad\text{ for all }x\in B_{\epsilon_{j,n}}(z_{j,n}),
 \end{equation}
 where $B_{\epsilon_{j,n}}(z_{j,n})=(a_{j,n},b_{j,n}):=(z_{j,n}-\epsilon_{j,n}, z_{j,n}+\epsilon_{j,n})$. Take $z\in (\frac 1 2,1)$ and let $$g:\left[\frac 1 2,1\right]\rightarrow \left[\frac 1 2,1\right]$$ be a continuous function such that i) $\frac 1 2$ and $1$ fixed points for $g$; ii) $z$ is not a periodic or ultimately periodic point for $g$. For any $j\in\mathbb{N}$, we define a sequence of functions $\overline{f}_{\beta_j,n}:[\frac{1}{2^{j+1}},\frac{1}{2^j}]\circlearrowleft$ as follows. For all $n<N^j$, we set $\overline{f}_{\beta_j,n}(x):= \Tilde{f}_{\beta_j,n}(x)$ for all $x\in[\frac{1}{2^{j+1}},\frac{1}{2^j}]$. For $n>N^j$, we set 
 
  \vspace{0.2cm}
  
 \begin{equation*}\label{eq_pert_lim}
    \overline{f}_{\beta_j,n}(x)=
\begin{cases}
    \Tilde{f}_{\beta_j,n}(x) &\quad \text{if }x\in [0,\frac 1 2)\setminus (a_{j,n},b_{j,n})\\
    \frac{\Tilde{x}_{j+1}- \Tilde{f}_{\beta_j,n}(a_{j,n})}{z_{j,n}-a_n}(x-z_{j,n})+\Tilde{x}_{j+1}  &\quad \text{if }x\in (a_{j,n},z_{j,n}]\\
    \frac{\Tilde{x}_{j+1}-\Tilde{f}_{\beta_j,n}(b_{j,n})}{z_{j,n}-b_n}(x-z_{j,n})+\Tilde{x}_{j+1} &\quad \text{if }x\in (z_{j,n},b_{j,n}).
\end{cases}
\end{equation*}

 \vspace{0.2cm}
 
 By construction $\overline{f}_{\beta_j,n}(z_{j,n})=\Tilde{x}_{j+1}\in [\frac{1}{2^{j+2}},\frac{1}{2^{j+1}}]$ and $\overline{f}_{\beta_j,n}$ is  continuous for all $j\in\mathbb{N}$. Moreover, by \eqref{cond_fc} and since $\Tilde{f}_{\beta_j,n} \dot{\longrightarrow} \Tilde{f}_{\beta_j}$, we have $\overline{f}_{\beta_j,n} \dot{\longrightarrow} \Tilde{f}_{\beta_j}$ for all $j\in\mathbb{N}$. 
 
 For every $n\in\mathbb{N}$, we define a continuous functions $\overline{f}_{\lambda,n}:\mathbb{I} \rightarrow \mathbb{I}$ as follows. We set $\overline{f}_{\lambda,n}(0)=0$, $\overline{f}_{\lambda,n}(x)=g(x)$ for all $x\in [\frac 1 2,1]$, and $\overline{f}_{\lambda,n}(x)=\overline{f}_{\beta_j,n}(x)$ whenever $x\in [\frac{1}{2^{j+1}},\frac{1}{2^j})$ for some $j\in\mathbb{N}$.
 
 We can now define a sequence of functions $f_{\lambda,n}:\mathbb{I} \rightarrow \mathbb{I}$ as follows.
 Set $f_{\lambda,n}(x)=\overline{f}_{\lambda,n}(x)$ for all $x\in\mathbb{I}$ and for all $n<N^1$. Let $c_n:=\max\{k\le n : N^k\le n\}$, then, for $n\ge N^1$, we set

 \vspace{0.2cm}
 
\begin{equation*}
    f_{\lambda,n}(x)=
\begin{cases}
    \overline{f}_{\lambda,n}(x) &\quad \text{if }x\in [0,\frac 1 2)\setminus (a_{c_n,n},b_{c_n,n})\\
    \frac{z-\overline{f}_{\lambda,n}(a_{c_n,n})}{z_{c_n,n}-a_{c_n,n}}(x-z_{c_n,n})+z  &\quad \text{if }x\in (a_{c_n,n},z_{c_n,n}]\\
    \frac{z-\overline{f}_{\lambda,n}(b_{c_n,n})}{z_{c_n,n}-b_{c_n,n}}(x-z_{c_n,n})+z &\quad \text{if }x\in (z_{c_n,n},b_{c_n,n}).
\end{cases}
\end{equation*}

 \vspace{0.2cm}
 
 Therefore, we have $f_{\lambda,n}(z_{c_n,n})=z$ for all $n\ge N^1$ and $f_{\lambda,n}$ is continuous for all $n\in\mathbb{N}$. Let us define $f_\lambda:\mathbb{I}\rightarrow \mathbb{I}$ as $f_\lambda |_{[1/2,1]}=g$, $f_\lambda(0)=0$ and  $f_\lambda |_{[\frac{1}{2^{j+1}},\frac{1}{2^j})}=\overline{f}_{\beta_j}$ for all $j\in\mathbb{N}$. Notice that $f_\lambda$ is a continuous function. Since the sequence $\{c_n\}_{n\ge N^1}$ is increasing and not ultimately constant, it follows that $f_{\lambda,n}\dot{\longrightarrow} f_\lambda$. This means that $(\mathbb{I},\{f_\lambda\})$ is a sequentially continuous TDS with the property $\boldsymbol{P}_\lambda$-1. We want to show that the point $\Tilde{x}_1\in\mathbb{I}$ verifies $\boldsymbol{P}_\lambda$-2 and $\boldsymbol{P}_\lambda$-3. By construction, we have that $$\Tilde{x}_j=\{f_\lambda\}^{\sum_{i=1}^{j-1}\beta_i}(\Tilde{x}_1)=[f_\lambda]^{\sum_{i=1}^{j-1}\beta_i}(\Tilde{x}_1)$$ for all $j\ge 2$. 
 Moreover, $z\in \mathcal{O}_n(\Tilde{x}_1)$ for all $n\ge N^1$, which implies $z\in\mathcal{O}_\infty(\Tilde{x}_1)$. 
 By the properties of the sequence $\{c_n\}_{n\in\mathbb{N}}$, for every $j\in\mathbb{N}$, we have that $z>_\infty u$, for all $u\in  \mathcal{O}_{\left[\sum_{i=1}^j\beta_i\right]}(\Tilde{x}_1)$ and there is no $y\in \mathcal{O}_\infty(\Tilde{x}_1)$ such that $z>_\infty y>_\infty u$ for all $u\in \mathcal{O}_{\left[\sum_{i=1}^j\beta_i\right]}(\Tilde{x}_1)$. It follows that 
 $$
 z=[f_\lambda]^\alpha(\Tilde{x}_1)=\{f_\lambda\}^\alpha(\Tilde{x}_1)\neq\emptyset
 $$ 
 with 
 \[
 \alpha=\sup_{j\in\mathbb{N}}\sum_{i=1}^j\beta_i\ge \sup_{j\in\mathbb{N}}\beta_j=\lambda.
 \]
  \end{enumerate}
\end{proof}

\begin{rem}
In the previous proof, one cannot replace $\lambda$ with $\omega_1$, because no countable sequence of ordinals converges to it. To obtain a sequentially continuous TDS having ordinal degree $\omega_1$ one should have a set of points $\{x_\iota\}_{\iota\in S}$, with $|S|\ge \aleph_1$, such that $D(x_\iota)=\lambda_\iota$ and $\lambda_\iota\to\omega_1$.
\end{rem}

\section{Transfinite dynamical relations} \label{sec5}
The topological dynamical relations introduced by E. Akin \cite[Chapter 1]{akin1993general} represent one of the cleanest paths towards the topological theory of attractors. Our main aim is to lay the foundation for an analogous theory for transfinite attractors, so we start by adapting some of Akin's concepts to the transfinite case. While certain ideas transfer formally,
the transfinite setting raises structural issues that do not occur in ordinary dynamics. We give first some generalizations to transfinite systems of some standard regularity properties for finite systems.

\begin{defi}[\textbf{Transfinite continuity}] \label{cont___}
Assume that $\lambda\le\omega_1$ is a limit ordinal and consider a TDS $(X,\{f\})$ and a point $x\in X$ such that $D(x)\ge\lambda$.
We say that $\{f\}$ is \emph{$\lambda$-continuous} at $x$ if, for every $\beta<\lambda$, the map \eqref{continuity__}
is continuous at $x$. 
\end{defi}

\begin{defi}
Assume that $\lambda<\omega_1$ is a limit ordinal and consider a TDS $(X,\{f\})$.  We say that $S\subseteq X$ is \emph{$\lambda$-closed} if $\{f\}^\beta(\overline{S})$ is closed for every $\beta<\lambda$.
\end{defi}
Clearly, in a finite topological dynamical system, by compactness every set is $\omega$-closed.

\begin{defi}\label{regularity_}
For $\lambda<\omega_1$, a TDS $(X,\{f\})$ such that $\mathfrak{D}(X,\{f\})\ge\lambda$ is called \emph{$\lambda$-regular} (we will write in short $\lambda$-TDS) if \eqref{continuity__} is a continuous map for every $\beta<\lambda$. 
The system $(X,\{f\})$ is called \emph{$\lambda$-closed} if \eqref{continuity__} is a closed map for every $\beta<\lambda$. 
The system $(X,\{f\})$ is called \emph{$\lambda$-normal} if it is both $\lambda$-regular and $\lambda$-closed. 
\end{defi}

    Notice that $\{f\}$ is $\lambda$-closed if and only if every set $S\subseteq X$ is $\lambda$-closed.
    We recall that a Baire space is a topological space such that every countable intersection of open, dense sets is dense, and that every complete metric space is Baire (see for instance \cite{munkres2000topology}, p. 295).
\begin{defi}\label{lambda_star}
If a $\lambda$-regular ($\lambda$-closed, $\lambda$-normal) system $(X,\{f\})$ is such that $X^\lambda$ is a Baire space, we say that it is $\lambda^*$-regular ($\lambda^*$-closed, $\lambda^*$-normal). 
\end{defi}

In general, $\lambda$-continuity does not imply that $\{f\}^\beta$ is uniformly continuous for $\beta<\lambda$, because the topological subspace $X^{\beta+1}$ can fail to be compact. A case of continuous but not uniformly continuous transfinite iteration is  shown in Example \ref{ex_attr_cyc}.
Of course every continuous finite topological dynamical system $(X,f)$ is $\omega$-regular. Moreover, assuming $X$ compact, as we do throughout, it is also $\omega^*$-normal and, thus, $\omega$-normal.
It is easy to see that the Example \ref{exa1} is $(\omega\cdot 2)$-normal.

It is important to appreciate that being $\lambda$-continuous and $\lambda$-closed are rather strong requirements. In the following result we see a set of (quite restrictive) assumptions under which $\lambda$-continuity and $\lambda$-regularity always fail alredy at level $(\omega+1)$.
\begin{prop} \label{fail_cont}
    Let $(X,\{f\})$ be a TDS. Suppose that $X^{\omega+1}=X$ and that $(X,f)$ is a minimal system. Then:
    \begin{enumerate} 
    \item if $|\{f\}^\omega(X)|\ge 2$, then $\{f\}^\omega$ is not continuous at any point;
    \item if $|\{f\}^\omega(X)|\ge\aleph_0$, then $\{f\}^\omega$ is not closed.
    \end{enumerate}
\end{prop}
\begin{proof}
Pick $x\in X$ and set $\{f\}^{\omega}(x)=y$. 
\begin{enumerate}
  
    \item Take $z\ne y$ such that $\{y,z\}\subseteq \{f\}^{\omega}(X)$. Let $w\in X$ be such that $\{f\}^{\omega}(w)=z$. Since $(X,f)$ is minimal, in every open neighborhood $U(x)$ of $x$ there is a point $t=f^k(w)$ for a certain $k\in\mathbb{N}$. We have: $$\{f\}^{\omega}(t)=\{f\}^{\omega}(\{f\}^k(w))=\{f\}^{k+\omega}(w)=\{f\}^\omega(w)=z,$$ and since $d(y,z)>0$ is independent of $U(x)$, the map $\{f\}^\omega$ is discontinuous at $x$.
      \item 
      Take a countable set of points $\{z_n\}_{n\in\mathbb{N}}$ such that, setting $w_n:=\{f\}^\omega(z_n)$, the sequence $\{w_n\}_{n\in\mathbb{N}}$ converges to a limit point $w$ distinct from every $w_n$ and also from $y$. 
      Let $\{U_n\}_{n\in\mathbb{N}}$ be a set of nested open sets such that $\cap_n U_n=x$. By minimality of $(X,f)$, for every $n\in\mathbb{N}$ we can find a point $t_n\in U_n$ such that $t_n=f^{k_n}(z_n)$ for a certain $k_n\in\mathbb{N}$, so that $$\{f\}^{\omega}(t_n)=\{f\}^{\omega}(\{f\}^{k_n}(z_n))=\{f\}^{k_n+\omega}(z_n)=\{f\}^\omega(z_n)=w_n.$$ Therefore, we have $$\{f\}^\omega(\{x\}\cup\{t_n\}_{n\in\mathbb{N}})=\{y\}\cup\{w_n\}_{n\in\mathbb{N}},$$
      where the set $\{x\}\cup\{t_n\}_{n\in\mathbb{N}}$ is closed since $t_n\longrightarrow x$ as $n$ goes to $\infty$, while the set on the right hand side is not closed since it does not contain the limit point $w$.
\end{enumerate}
\end{proof}
In the following example we can see a case to which Proposition \ref{fail_cont} applies.

\begin{exa}\label{irrational__}
    
    Let $(\mathbb{T}=[0,1),R_a)$ be the irrational rotation on the circle. Set $U_n:=B_{\epsilon_n}(0)$ where $\epsilon_n \to 0$ s $n$ goes to $\infty$.
    Let $g: A:=\{ \mathcal{O}^\mathbb{Z}(x) : x\in \mathbb{T} \}\rightarrow \mathbb{T}$ be a surjective map such that $g(\mathcal{O}^\mathbb{Z}(0))=R_a(0)=a$.
    Notice that the latter exists since $|A|=|\mathbb{T}|$. 
    For every $x\in\mathbb{T}$, we set $B_n(x):=\mathcal{O}^\mathbb{Z}(x)\cap U_n$. 
    Let us define a sequence $\{R_{a,n}\}_{n\in\mathbb{N}}$ where $R_{a,n}:\mathbb{T}\circlearrowleft$ is given by:
    $$
    R_{a,n}(x)=R_a(x) \text{ if $x\notin B_n(x)$}\quad,\quad R_{a.n}(x)=g(\mathcal{O}^\mathbb{Z}(x)) \text{ if $x\in B_n(x)$}.
    $$
    Since $\epsilon_n\to 0$ as $n\to\infty$ and $R_{a,n}(0)=R_a(0)$ for all $n\in\mathbb{N}$, we have that $R_{a,n}\dot{\longrightarrow} R_a$, so that $(\mathbb{T},\{R_a\})$ is a TDS. 
    Since all the orbits are disjoint and dense in $\mathbb{T}$, it follows that $\{f\}^\omega(x)=g(\mathcal{O}^\mathbb{Z}(x))$ for all $x\in\mathbb{T}$. Since the map $g$ is surjective, we have that $\{f\}^\omega(\mathbb{T})=\mathbb{T}$.

    In this example, the iteration of order $\omega$ is defined on every point of the space, while no point has iteration of order $\ge\omega^2$. Therefore, for every $x\in \mathbb{T}$, we have $$D(x)=\mathfrak{D}(\mathbb{T},\{R_\alpha\})=\omega^2.$$ Moreover, the system is minimal and, consistently with Proposition \ref{fail_cont}, the map $\{f\}^\omega$ is everywhere discontinuous.
\end{exa}

\begin{defi}[\textbf{Transfinite dynamical relations}]
\label{def_tdr}
Let $(X,\{f\})$ be a TDS and  $\lambda\le\omega_1$ a limit ordinal. We define the following relations:

\begin{enumerate}
\item \emph{Transfinite orbit relation}: 

$(x,y)\in \lambda\mathcal{\{H\}}\ \text{($x$ ``hits" $y$) } \iff \exists\beta<\lambda:\{f\}^\beta(x)=y$

\item \emph{Transfinite recurrence relation}:

$(x,y)\in \lambda\mathcal{\{R\}} \iff$ for every open neighborhood $ U(y)$,   there exists $\beta<\lambda : \{f\}^\beta (x)\in U(y)$.

\item \emph{Transfinite nonwandering relation}:

$(x,y)\in \lambda\mathcal{\{N\}}\iff$ for every pair of open neighborhoods $U(y)$ and $V(x)$,  there exist $z\in V(x)$ and $\beta<\lambda$ such that $\{f\}^\beta (z)\in U(y)$.

\item \emph{Transfinite chain-recurrence relation}:

For $n\in\mathbb{N}$ we say that the pair 
$$(\{x_0,\dots,x_n\},\{\lambda_0,\dots,\lambda_{n-1}\})$$
is an \emph{$(\epsilon,\lambda)$-chain} from $x$ to $y$ (we may say as well ``connecting $x$ and $y$" or ``between $x$ and $y$") if we have:
\begin{enumerate}
    \item $x_i\in X$ for $i=0,\dots, n$, with $x_0=x, x_n=y$;
    \item the ordinals $\lambda_0,\ldots \lambda_{n-1}$ are such that $0<\lambda_i<\lambda$ and $\{f\}^{\lambda_i}(x_i)\neq \emptyset$ for $i=0,\ldots,n-1$;
    \item $d(\{f\}^{\lambda_i}(x_i),x_{i+1})<\epsilon$ for $i=0,\dots, n-1$.
\end{enumerate}

Set now:

$(x,y)\in \lambda\mathcal{\{C\}}\iff\forall \epsilon>0$ there exists an $(\epsilon,\lambda)$-chain  connecting $x$ and $y$.
\end{enumerate}
For $\mathcal{A}=\mathcal{H},\mathcal{R},\mathcal{N}$ or $\mathcal{C}$, we may write $x\,\lambda\{\mathcal{A}\}\,y$ to mean $(x,y)\in\lambda\{\mathcal{A}\}$. For $S\subset X$, we set 
$$\lambda\{\mathcal{A}\}(S):=\{y\in X: \exists x\in S: x\,\lambda\{\mathcal{A}\}\, y\}.$$
\end{defi}

The relation $\lambda\{\mathcal{C}\}$, just like the usual relation $\mathcal{C}$ in finite systems, is independent of the metric, in the sense that it stays unchanged if $d$ is replaced by a topologically equivalent metric $d'$. In fact, $x\,\lambda\{\mathcal{C}\}\,y$ is equivalent to the claim that, for every neighborhood $U\subseteq X\times X$ of the set $\{(x,x):x\in X\}$, there is a $\lambda$-chain  $$C_U:=\{(\{f\}^{\lambda_i}(x_i),x_{i+1})\}_{0\le i<n}$$ such that $C_U\subseteq U$. We prefer the definition using the metric because, just as it happens with $\mathcal{C}$, it is easier to work with.

\begin{defi} For $\mathcal{A}=\mathcal{H},\mathcal{R},\mathcal{N}$ or $\mathcal{C}$ we set

$$(x,y)\in \mathcal{\{A\}}\iff (x,y)\in \lambda\mathcal{\{A\}}\text{ for some  limit ordinal $\lambda$}.$$
\end{defi}

When $\lambda=\omega$, the previous relations become respectively the usual orbit, recurrence, nonwandering and chain recurrence relations (which we indicate here respectively by $\mathcal{H}$, $\mathcal{R}$, $\mathcal{N}$ and $\mathcal{C}$; often the same symbol $\mathcal{O}$ is used for both the orbit of a point and the orbit relation; to avoid potential confusion here we introduced the symbol $\mathcal{H}$), as defined for instance in \cite[Chapter 1]{akin1993general}. In particular, 
\begin{equation}
\label{chain_}
    \centering
    \begin{tabular}{lc}
       $x\,\mathcal{C}\,y\iff$  &  $\forall \epsilon>0,\  \exists \,n>0,\  \exists\, x_0,x_1,\ldots,x_n \text{, $x_0=x,\ x_n=y$, and }$\\
        & \\
         & $d(f(x_i),x_{i+1})<\epsilon, \quad\forall i\le n-1$.
    \end{tabular}
\end{equation}

However, notice that, although $(x,y)\in \omega\{\mathcal{C}\}\iff (x,y)\in\mathcal{C}$, an $(\epsilon,\omega)$-chain does not generally correspond to an (ordinary) $\epsilon$-chain with the same number of points.

\begin{prop}
\label{incl_top_rel}
Let $(X,\{f\})$ be a TDS. Then, for every countable ordinal $\lambda$, we have
\begin{equation*}
    \lambda\{\mathcal{H}\}\subseteq\lambda\{\mathcal{R}\}\subseteq\lambda\{\mathcal{N}\}\subseteq\lambda\{\mathcal{C}\}.
\end{equation*}
\end{prop}
\begin{proof} 
Let $(x,y)\in \lambda\{\mathcal{H}\}$, then there exists $\beta<\lambda$ such that $\{f\}^\beta(x)=y$. Hence, for any $U(y)$ open neighborhood of $y$ we have that $\{f\}^\beta(x)=y\in U(y)$, which implies that $(x,y)\in\lambda\{\mathcal{R}\}$.  Let $(x,y)\in\lambda\{\mathcal{R}\}$. Since $x\in V(x)$ for any $V(x)$ open neighborhood of $x$, it follows that $(x,y)\in \lambda\{\mathcal{N}\} $. For the last inclusion, assume that $(x,y)\in\lambda\{\mathcal{N}\}$ and pick $\epsilon>0$. Since $f$ is continuous, we can take $0<\delta<\epsilon$ sufficiently small that
$
f(B_\delta(x))\subseteq B_\epsilon(f(x)).
$
Then there exist $z\in B_\delta(x)$ and $\beta<\lambda$ such that $\{f\}^\beta(z)\in B_\delta(y)$. It follows that
\begin{equation}\label{eq_iterata1}
d(f(x),f(z))<\epsilon,
\end{equation}
 and we can consider two cases:
 
\begin{enumerate}
\item Let us assume that $\beta=k\in\mathbb{N}$. Then the result is just the inclusion  between the usual non-wandering and chain recurrence relations, $\mathcal{N}\subseteq \mathcal{C}$: the pair $(\{x,f(z),y\},\{1,k-1\})$ is an $(\epsilon,\lambda)$-chain from $x$ to $y$. Indeed, recalling \eqref{eq_iterata1}, it is sufficient to observe that
\[
f^{k-1}(f(z))=f^k (z)=\{f\}^\beta(z)\in B_\delta(y)\subset B_\epsilon(y).
\]
\item Let us assume that $\beta\ge \omega$. Consider the pair $(\{x,f(z),y\},\{1,\beta\})$: this is an $(\epsilon,\lambda)$-chain from $x$ to $y$. Indeed, since $\beta\ge \omega$, it follows that 
\[
\{f\}^{\beta}(f(z))=\{f\}^{1+\beta}(z)=\{f\}^\beta(z),
\] 
and by the definition of $\lambda\{\mathcal{N}\}$ we have $\{f\}^\beta(z)\in B_\delta(y)\subsetneq B_\epsilon(y).$
\end{enumerate} 
\end{proof}
\begin{rem}\label{etabeta}
For later purposes, let us point out that, in the previous proof:
\begin{itemize}
    \item in point 1. we deduced from $f^k(z)\in B_\delta(y)$ the existence of an $(\epsilon,\lambda)$-chain in which the last iteration is $k-1$, if $k$ is the iteration assumed to map $z\in B_\delta(x)$ to a point in $B_\delta(y)$;
    \item in point 2. we deduced from $\{f\}^\beta(z)\in B_\delta(y)$ the existence of an $(\epsilon,\lambda)$-chain in which the last iteration is $\beta$, if $\beta$ is the iteration assumed to map $z\in B_\delta(x)$ to a point in $B_\delta(y)$.
\end{itemize} 
\end{rem}

\begin{defi}
 For a relation $\mathcal{A}\subseteq X^2$, we indicate by $|\mathcal{A}|$ the \emph{diagonal} of $\mathcal{A}$, that is the set $|\mathcal{A}|:=\{x\in X:(x,x)\in \mathcal{A}\}$.  
\end{defi}

\begin{cor}
Let $(X,\{f\})$ be a TDS. Then, for every countable ordinal $\lambda$, we have $$|\lambda\{\mathcal{H}\}|\subseteq|\lambda\{\mathcal{R}\}|\subseteq|\lambda\{\mathcal{N}\}|\subseteq|\lambda\{\mathcal{C}\}|.$$
\end{cor}
\begin{proof}
  It is sufficient to take the initial pair of the form $(x,x)$ in the proof of Proposition \ref{incl_top_rel}. 
\end{proof}
\begin{defi}
Let us indicate by $\mathcal{H}_n,\ \mathcal{R}_n,\ \mathcal{N}_n$ and $\mathcal{C}_n$ respectively the usual orbit, recurrence, nonwandering and chain recurrence relations for $(X,f_n)$.
   For $\mathcal{A}=\mathcal{H},\mathcal{R},\mathcal{N}$ or $\mathcal{C}$ we set
   \[ \mathcal{A}_\infty:= \liminf_{n\to\infty} \mathcal{A}_n. \]
\end{defi}

The transfinite orbit of a point is always a subset of the limit inferior of the set of finite orbits $\{\mathcal{O}_n(x)\}_{n\in\mathbb{N}}$, that is $\{\mathcal{O}\}(x)\subseteq \mathcal{O}_\infty(x)$. A similar fact is true for the transfinite orbit relation, that is $\{\mathcal{H}\}\subseteq \mathcal{H}_\infty$. Indeed, if $(x,y)\in \{\mathcal{H}\}$, then $(x,y)\in \lambda\{\mathcal{H}\}$ for some limit ordinal $\lambda$, which means that $y\in\mathcal{O}_\infty(x)$. 

 An analogous inclusion, however, does not hold for the other transfinite topological relations. To see that, consider Example \ref{exa1} with $h=z$. We have that $(z,0)\in \mathcal{R}\subseteq \{\mathcal{R}\}$, but $(z,0)\notin \mathcal{R}_\infty$. By Proposition \ref{incl_top_rel}, this also implies that $\{\mathcal{N}\}$ and $\{\mathcal{C}\}$ are in general not subsets of, respectively, $\mathcal{N}_\infty$ and $\mathcal{C}_\infty$, because $(z,0)\notin\mathcal{C}_{\infty}\supseteq\mathcal{N}_\infty\supseteq\mathcal{R}_\infty$.

The converse inclusions do not hold either, in general, in the sense that it can be: \begin{equation}\label{notsubset_}
\mathcal{H}_\infty\not\subseteq \{\mathcal{H}\},\ \mathcal{R}_\infty\not\subseteq \{\mathcal{R}\},\ \mathcal{N}_\infty\not\subseteq \{\mathcal{N}\},\   \mathcal{C}_\infty\not\subseteq \{\mathcal{C}\}.
\end{equation}
Indeed, consider again the example in Fig. \ref{fig_A2A3}b. Since $z\in \mathcal{O}_n(z)$ for all $n\in\mathbb{N}$, we have that $(z,z)\in \mathcal{H}_\infty$. On the other hand, since $\{f\}^\omega(z)=\emptyset$ and $z\notin \mathcal{O}(z)$, it follows that $(z,z)\notin \{\mathcal{H}\}$. By Proposition \ref{incl_top_rel}, this provides an example also for the other inclusion negations in \eqref{notsubset_}.

The following definition refines the concept of recurrence by adding some further requirement to the mere $\lambda$-recurrence. In particular, we want to specify whether i) a point $x$ recurs with $y$ ``not before" a certain limit ordinal $\lambda$; ii) the same applies to every point in the $\lambda$-orbit of $x$.
\begin{defi}\label{proper_strict}
    Let $(x,y) \in \lambda\{\mathcal{R}\}$ for some limit ordinal $\lambda$. We say that 
    \begin{itemize}
        \item $x$ is \emph{properly $\lambda$-recurrent} with $y$ if, for all limit ordinals $\beta<\lambda$, we have $(x,y)\notin \beta\{\mathcal{R}\}$;
        \item $x$ is \emph{strictly $\lambda$-recurrent} with $y$ if there is no $z\in\mathcal{O}_{\{\lambda\}}(x)$ such that   $(z,y)\in \beta\{\mathcal{R}\}$ for some limit ordinal $\beta<\lambda$.
    \end{itemize}
\end{defi}
 Notice that every recurrent pair in the ordinary sense, which in our notation can be indicated by $(x,y)\in\omega\{\mathcal{R}\}$, is always properly and strictly recurrent as there are no limit ordinals below $\omega$.
\begin{lem}\label{strictprop_}
Let $\lambda$ be a countable limit ordinal.
  \begin{enumerate}
      \item Strict $\lambda$-recurrence implies proper $\lambda$-recurrence;
      \item if $\lambda$ is an indecomposable ordinal, then strict $\lambda$-recurrence is equivalent to proper $\lambda$-recurrence.
  \end{enumerate}  
\end{lem}
\begin{proof}
\begin{enumerate}
    \item Suppose that $x$ is strictly $\lambda$-recurrent with $y$. Assume by absurd that $x\,\beta\{\mathcal{R}\}\,y$ for some limit ordinal $\beta<\lambda$.
If $\beta=\omega$, then $\{f\}(x)$ is also $\omega$-recurrent with $y$, contradicting strict $\lambda$-recurrence. Assume thus $\beta>\omega$. For every $\eta\ge\omega$ it is $\{f\}^\eta(\{f\}(x))=\{f\}^\eta(x)$, so we have again $\{f\}(x)\,\beta\{\mathcal{R}\}\,y$, contradicting strict $\lambda$-recurrence. 
\item Assume that $\lambda<\omega_1$ is an additively indecomposable ordinal. By item 1., we just have to prove that proper $\lambda$-recurrence implies strict $\lambda$-recurrence. Let $x$ be properly $\lambda$-recurrent with $y$.  Consider a limit ordinal $\eta<\lambda$. If, by absurd, we have $(\{f\}^\eta(x),y)\in \beta\{\mathcal{R}\}$ for some limit ordinal $\beta<\lambda$, then $x$ would be $(\eta+\beta)$-recurrent with $y$. Since $\lambda$ is additively indecomposable, we have $\eta+\beta<\lambda$, which contradicts proper $\lambda$-recurrence. 
\end{enumerate}
\end{proof}

\begin{exa}\label{circle__}
Let us consider the unit disk $\mathbb{D}:=\{z\in\mathbb{C}:|z|\le 1\}$ in the complex plane. We write complex points $z$ in  polar form, and in particular we write $z$ as $(\rho,\theta)$ if $z=\rho e^{i\theta}$, with $\rho\in \mathbb{I}$ and $\theta\in[0,2\pi]$; we indicate the origin $(0,0)$ by $O$. Let $g:\mathbb{I}\circlearrowleft$ be a strictly monotonic continuous function such that: 
    
\begin{itemize}
\item The point 0 is a repulsive fixed point;
\item The point 1 is an attractive fixed point;
\item the $\omega$-limit set $\omega_g(x)=1$ for every $x\in (0,1]$.
\end{itemize}

Set $\theta_0=0$ and let $\{\theta_k\}_{k\in\mathbb{N}}$ be a strictly increasing sequence of points in $(0,2\pi]$ that converges to $2\pi$ as $k\to\infty$.  Let us define: 
$$R_k=[0,1]\times \theta_k, \quad \forall k\in\mathbb{N}_0$$
and further: 
$$\Theta= \bigcup_{k\in\mathbb{N}_0}\ \theta_k ,\qquad R=\bigcup_{k\in\mathbb{N}_0} R_k=[0,1]\times \Theta.$$
Notice that $R$ is a compact subspace of $\mathbb{D}$. 
Let $f:R\rightarrow R$ be the function given by $$f(\rho,\theta)=(g(\rho),\theta)\ \text{for all }  (\rho,\theta)\in R.$$ 
For every $n\in\mathbb{N}$ we consider the point $x_n=g^n(1/2)$ and a sequence or positive reals $\{\epsilon_n\}_{n\in\mathbb{N}}$ converging to 0 and such that $\epsilon_n$ is so small that $g^{n-1}(1/2)$ and $g^{n+1}(1/2)$ are not contained in $[x_n-\epsilon_n,x_n+\epsilon_n]$. Let $\{I_n\}_{n\in\mathbb{N}}$ be a sequence of intervals in $\mathbb{I}$, where $I_n=[x_n-\epsilon_n,x_n+\epsilon_n]$. We can define a sequence $\{f_n\}_{n\in\mathbb{N}}$ of functions $f_n:R\rightarrow R$ as follows:
\begin{equation}
f_n(\rho,\theta)=
\begin{cases}
(g^{-(k+1)}(\frac 1 2),\theta_{k+1}) &\quad \text{if }(\rho,\theta)\in I_n\times \theta_k \text{ for some } k\in\mathbb{N} \\
(g(\rho),\theta) &\quad \text{otherwise}
\end{cases}
\end{equation}
Since $\epsilon_n\to 0$, we have $f_n \dot{\longrightarrow}  f$. Moreover, the limit map $f$ is continuous. Therefore, we have defined the TDS $\left(R,\{f\}\right)$.
In Fig.\ref{fig_ex_circ} we plot the space $R$ and we indicate by dots the iterations of the point $z=\left(\frac 12, \theta_1\right)$ under the map $f_1$ (left) and $f_n$ (right). The dotted circle represents the level $x_n-\epsilon_n$ at which the point ``jumps" from the interval at angle $\theta_n$ to the interval at angle $\theta_{n+1}$. 
Looking at some transfinite orbits, we have $\{f\}^\omega(1/2,\theta_1)=(g^{-2}(1/2),\theta_2)$, and more generally, for $k\ge 1$, 
\begin{equation}\label{transorb}\{f\}^\omega(1/2,\theta_k)=\{f\}^{\omega\cdot k}{(1/2,\theta_1)}=(g^{-(k+1)}(1/2),\theta_{k+1}).
\end{equation}
Since the sequence $\{(g^{-(k+1)}(1/2),\theta_{k+1})\}_{k\in\mathbb{N}}$ converges to $O$, Eq. \eqref{transorb} implies: $(1/2,\theta_1)\,\omega^2\{\mathcal{R}\}\,O$, where the $\omega^2$-recurrence is proper (in the sense of Def. \ref{proper_strict}). In fact, since $\omega^2$ is additively indecomposable, for $i\in\mathbb{N}$, by Lemma \ref{strictprop_} each $t_i:=\left(\frac 1 2, \theta_i\right)$  is strictly $\omega^2$-recurrent with $O$. 

We will come back to this example in Section \ref{sec6}.

\begin{figure}[H] 
\centering
\subfigure[$f_1$]{\fbox{\includegraphics[width=6.9cm]{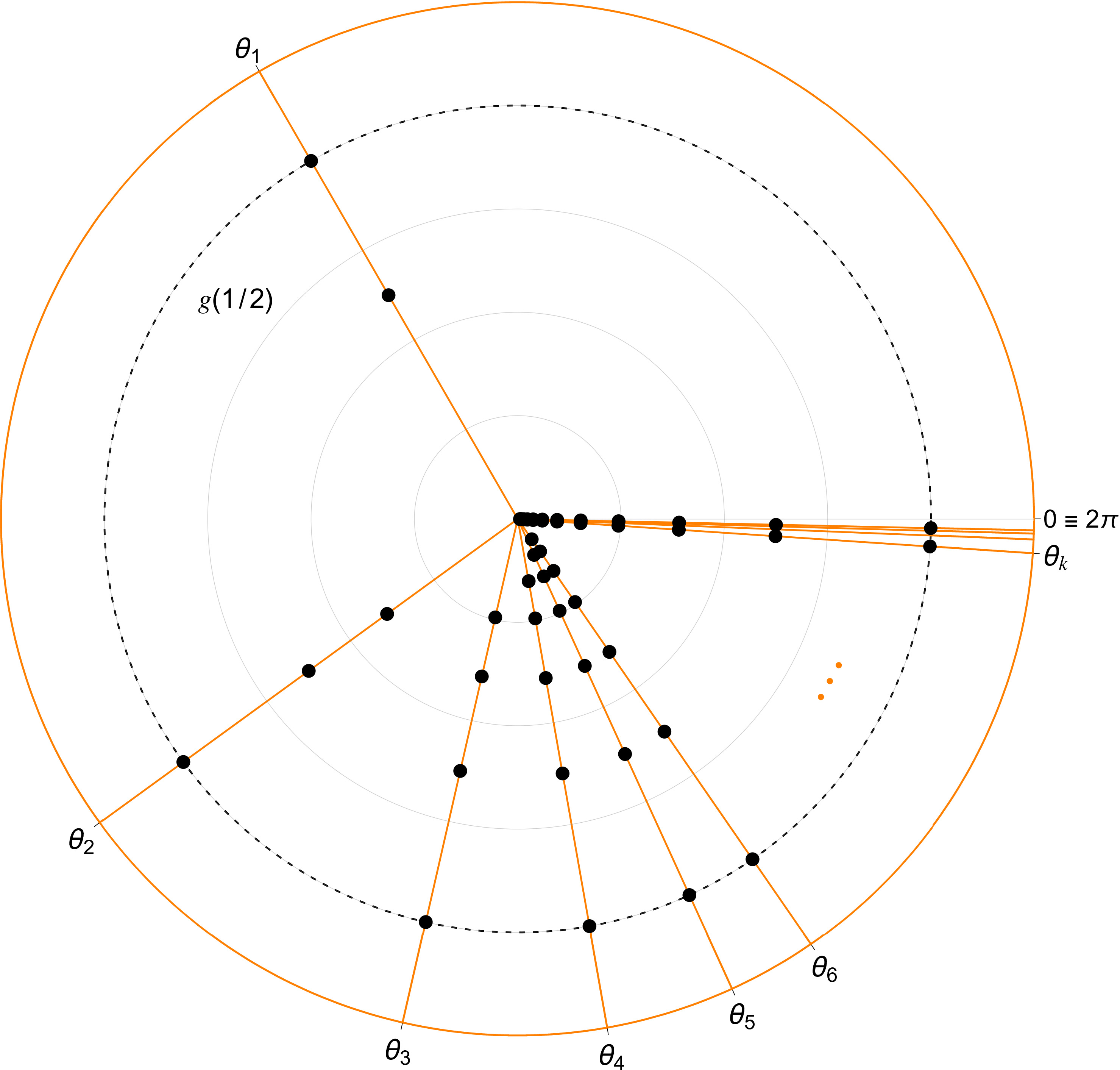}}}
\hspace{0.5mm}
\subfigure[$f_n$]
{\fbox{\includegraphics[width=6.9cm]{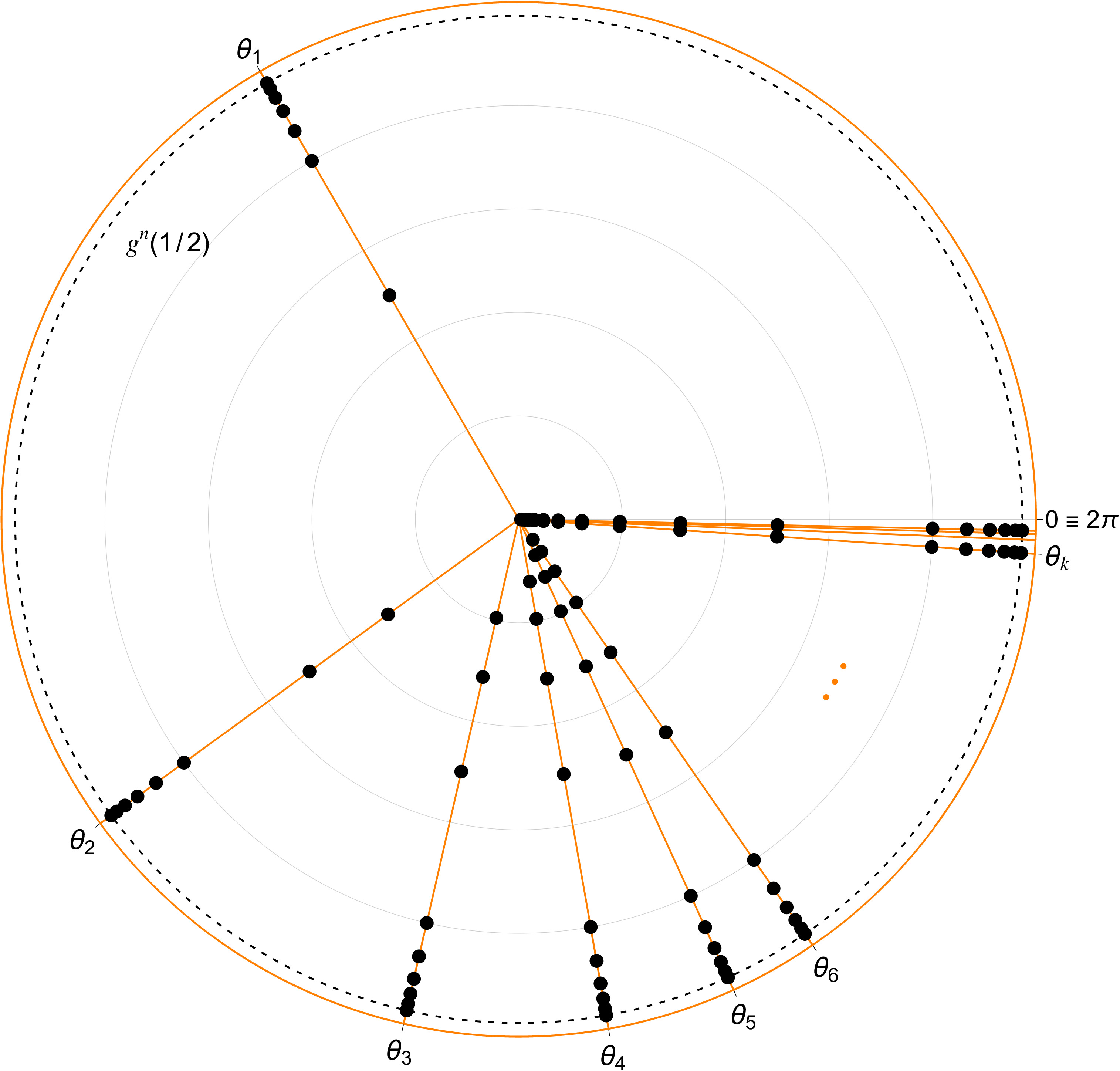}}}
\caption{The maps $f_1$ and $f_n$ from the sequence $\{f_n\}_{n\in\mathbb{N}}$ defining the system in Example \ref{circle__}.}
\label{fig_ex_circ}
\end{figure}
\end{exa}

\begin{prop}\label{prop_N_C}
Let $(X,\{f\})$ be a TDS and $\lambda<\omega_1$. Then the following hold:
\begin{itemize}
 \item $\lambda\{\mathcal{H}\}$ is transitive.
 \item $\lambda\{\mathcal{N}\}=\overline{\lambda\{\mathcal{H}\}}$ in the product topology of $X\times X$.
    \item $\lambda\{\mathcal{C}\}$ is transitive and closed in $X\times X$.
\end{itemize}
\end{prop}

\begin{proof}
The transitivity of $\lambda\{\mathcal{H}\}$ is a consequence of the composition laws \eqref{composition1}-\eqref{composition2}.

By Def. \ref{def_tdr}, $(x,y)\in\lambda\{\mathcal{N}\}$ if and only if for every pair of neighborhoods $U(y),V(x)$ there exist $z\in V(x)$ and $w\in U(y)$ such that $(z,w)\in \lambda\{\mathcal{H}\}$, which means that $(x,y)\in \overline{\lambda\{\mathcal{H}\}}$.

The transitivity of $\lambda\{\mathcal{C}\}$ is trivial. 
Let $(x,y)\in \overline{\lambda\{\mathcal{C}\}}$, we want to show that $(x,y)\in \lambda\{\mathcal{C}\}$. Pick $\epsilon>0$. We want to construct an $(\epsilon,\lambda)$-chain from $x$ to $y$. Since $f$ is continuous, there exists $0<\delta<\epsilon/2$ such that 
$$d(x,z)<\delta\implies d(f(x),f(z))<\frac \epsilon 2.$$
By hypothesis, there exists an $(\epsilon/2,\lambda)$-chain 
\[
(\{x_0,x_1,\ldots,x_n\},\{\lambda_0,\ldots,\lambda_{n-1}\})
\]
from some $x_0\in B_{\delta}(x)$ to some $x_n\in B_\delta(y)$. 

Let us assume $\lambda_0\ge \omega$. 
Then $\{f\}^{\lambda_0}(f(x))=\{f\}^{\lambda_0}(x)$ for all $x\in X$. 
Thus the pair
\[
(\{x_0,f(x_0),x_1,\ldots,x_n\},\{1,\lambda_0,\ldots,\lambda_{n-1}\})
\]
is an $(\epsilon/2,\lambda)$-chain from $x_0\in B_\delta (x)$ to $x_n\in B_\delta (y)$. Indeed,
$$d(f(x_0),f(x_0))=0,\quad d(\{f\}^{\lambda_0}(f(x_0)),x_1)=d(\{f\}^{\lambda_0}(x_0),x_1)<\frac \epsilon 2.$$
Therefore we have that the pair
\[
(\{x,f(x_0),x_1,\ldots,x_{n-1},y\},\{
1,\lambda_0,\ldots,\lambda_{n-1}\})
\]
is an $(\epsilon,\lambda)$-chain from $x$ to $y$. Indeed, by the continuity of $f$ we have
$d(f(x),f(x_0))<\epsilon/2$,
and since $\delta<\epsilon/2$, it follows that 
$$d(\{f\}^{\lambda_{n-1}}(x_{n-1}),y)\le d(\{f\}^{\lambda_{n-1}}(x_{n-1}),x_n)+d(x_n,y)<\epsilon.$$
If we assume that $\lambda_0<\omega$, it is sufficient to consider the pair 
\[
(\{x,f(x_0),x_1,\ldots,x_{n-1},y\},\{1,\lambda_0-1,\lambda_1,\ldots,\lambda_{n-1}\}),
\] 
which is an $(\epsilon,\lambda)$-chain from $x$ to $y$.
\end{proof}

\begin{defi}\label{minimal__}
 We say that a TDS is \emph{$\lambda$-minimal} if, for every $x\in X$, we have $\overline{\mathcal{O}_{\{\lambda\}}(x)}=X$.
\end{defi}

\begin{defi}\label{minimal_transitive}
    We say that a TDS is \emph{$(\lambda,\beta$)-transitive} if, for every pair $U,V\subseteq X$ of nonempty open sets, there is $x\in U\cap X^{\beta}$ and $\eta<\lambda$ such that $\{f\}^\eta(x)\in V$.
    If $(X,\{f\})$ is $(\lambda,\omega)$-transitive, we simply say that it is $\lambda$-transitive (we recall that we have always $X^\omega=X$).
\end{defi}
Note that $(\lambda,\beta)$-transitive implies $(\lambda,\beta')$-transitive if $\beta\ge\beta'$ and $(\lambda,\beta)$-transitive implies $(\lambda',\beta)$-transitive if $\lambda\le\lambda'$.
Clearly a finite system is $\omega$-transitive if and only if it is transitive in the usual sense.
\begin{defi}
Given a point $x\in X$ and a pair $(\epsilon,\lambda)$, where $\epsilon>0$ and $\lambda<\omega_1$ a limit ordinal, we say that $x$  \emph{$(\epsilon,\lambda)$-shadows} a finite sequence $x_0,\ldots, x_n$ in $X$ if $x\in B_\epsilon(x_0)$ and for every $1\le i\le n$, there exists $\beta_i<\lambda$ such that 
$$d(\{f\}^{\beta_i}(x),x_i)<\epsilon.$$
We say that a system $(X,\{f\})$ has the \emph{$\lambda$-shadowing property} if, for every $\epsilon>0$, there exists $\delta>0$ such that any $(\delta,\lambda)$-chain is $(\epsilon,\lambda)$-shadowed by some point $x\in X$.
\end{defi}

\begin{prop}\label{th_min_tran_shw} 
For a TDS $(X,\{f\})$ and $\lambda<\omega_1$, the following hold:
\begin{enumerate}
\item The system $(X,\{f\})$ is $\lambda$-minimal if and only if $\lambda\{\mathcal{R}\}=X^2$.
\item  The system $(X,\{f\})$ is $\lambda$-transitive if and only if $\lambda\{\mathcal{N}\}=X^2$.
\item If $(X,\{f\})$ has the $\lambda$-shadowing property, then $\lambda\{\mathcal{N}\}=\lambda\{\mathcal{C}\}$.
\end{enumerate}
\end{prop}
\begin{proof}
\begin{enumerate} 
\item Suppose that the system is $\lambda$-minimal and pick $x,y\in X$. By definition, $\overline{\mathcal{O}_{\lambda}(x)}=X$. It follows that for every $U(y)$ open neighborhood of $y$ there exists $\beta<\lambda$ such that $\{f\}^\beta(x)\in U(y)$, which implies that $x\, \lambda\{\mathcal{R}\}\,y$. Conversely, suppose that $\lambda\{\mathcal{R}\}=X^2$ and pick $x\in X$. For every fixed $y\in X$, since $x\, \lambda\{\mathcal{R}\}\, y$ we have that for every $U(y)$ open neighborhood of $y$ there exists $\beta<\lambda$ such that $\{f\}^\beta(x)\in U(y)$. By the arbitrariness of $y\in X$, it follows that $\overline{\mathcal{O}_{\{\lambda\}}(x)}=X$.
    
\item Suppose that the system is $\lambda$-transitive and pick $x,y\in X$. 
Therefore, for every pair of open neighborhoods $U(x)$ and $V(y)$, there exist $z\in U(x)$ and $\eta<\lambda$ such that $\{f\}^\eta(x)\in V(y)$, which means that $x\,\lambda\{\mathcal{N}\}\,y$. 
On the other hand, suppose that $\lambda\{\mathcal{N}\}=X^2$ and consider a pair $U,V\subset X$ of nonempty open sets. 
Take $x\in U$ and $y\in V$. Since $x\,\lambda\{\mathcal{N}\}\,y$, there exists $z\in U$ and $\eta<\lambda$ such that $\{f\}^\eta (x)\in V$.  

\item It is sufficient to show that $\lambda\{\mathcal{C}\}\subseteq \lambda\{\mathcal{N}\}$. 
Let $(x,y)\in \lambda\{\mathcal{C}\}$ and take a pair of open neighborhood $U(y)$ and $V(x)$. 
Then, there exists $\epsilon>0$ such that $B_\epsilon(y)\subseteq U(y)$ and $B_\epsilon(x)\subseteq V(x)$. 
Since the system has the $\lambda$-shadowing property, there exists $\delta>0$ such that any $(\delta,\lambda)$-chain is $(\epsilon,\lambda)$-shadowed by some point in $X$. 
We know that there exists a $(\delta,\lambda)$-chain from $x$ to $y$ and it is $(\epsilon,\lambda)$-shadowed by a point $z\in X$. Then $z\in B_\epsilon(x)\subseteq V(x)$ and there exists $\beta<\lambda$ such that $d(\{f\}^\beta(z),y)<\epsilon$.
\end{enumerate}
\end{proof}

For the next result, we use $\lambda^*$-regularity. We want to prove that, under certain conditions, transfinite topological transitivity implies transfinite point transitivity. For this, we need to find a suitable relative topology in which the set of preimages of the elements of a base is both open and dense. As we will see, this happens when a simple order condition is verified.

\begin{prop}\label{__trans__}
Let us assume that $(X,\{f\})$ is $\lambda^*$-regular and $(\lambda,\beta)$-transitive for some countable limit ordinals $\lambda,\beta$. If $\beta\ge\lambda$, then there exists $x\in X$ such that $\overline{\mathcal{O}_{\{\lambda\}}(x)}=X$.
\end{prop}
\begin{proof}
Let $B=\{U_n\}_{n\in\mathbb{N}}$ be a countable base for $X$. 
For $\iota<\lambda$, set 
$$
W_{\iota,n}:=\{f\}^{-\iota}(U_n)\quad,\quad V_n:=\bigcup_{\iota<\lambda} W_{\iota,n}.
$$
By $\lambda$-regularity, every set $W_{\iota,n}$ is open in the relative topology of $X^{\iota}$. 
Since for $\iota<\lambda$ it is $X^\lambda\subseteq X^\iota$, we also have $W_{\iota,n}\cap X^\lambda$, and thus $V_n\cap X^\lambda$, is open in the relative topology of $X^\lambda$.

Let $U\subseteq X$ be a nonempty open set. 
By $(\lambda,\beta)$-transitivity, for every $n\in\mathbb{N}$ there exists $\iota<\lambda$ and $x\in U\cap X^\beta$ such that $\{f\}^\iota(x)\in U_n$, and therefore $U\cap W_{\iota,n}\cap X^\beta\ne\emptyset$, so that $$
\overline{\bigcup_{\iota<\lambda}(W_{\iota,n}\cap X^\beta)}=\overline{V_n\cap X^\beta}=X.
$$ 
Since $\beta\ge\lambda$, we have $X^\beta\subseteq X^\lambda\subseteq X$, so that, for every $n\in\mathbb{N}$,
$V_n\cap X^\beta\subseteq V_n\cap X^\lambda$, so that the latter is  dense in $X^\lambda\subseteq X$.
By $\lambda^*$-regularity, we have 
$\bigcap_{n\in\mathbb{N}}(V_n\cap X^\lambda)=:S\ne\emptyset$.
Take $y\in S$. 
For every $U_n\in B$ there is $\iota<\lambda$ such that $\{f\}^\iota(y)\in U_n$, so that $\overline{\mathcal{O}_{\{\lambda\}}(y)}=X$.  
\end{proof}

Proposition \ref{__trans__}, in the particular case $\lambda=\omega$, becomes the well-known result that topological transitivity implies point-transitivity. In fact, in case of finite dynamical systems we can say more (see for instance \cite{kurka2003topological}, Theorem 2.9):
\begin{fact}
In a topological dynamical system $(X,f)$, with $X$ compact, $f$ continuous, there is a point $x\in X$ with a dense orbit if and only if the system is topologically transitive. 
\end{fact} 
Let us remark that, in \cite{kurka2003topological}, $\mathcal{O}(x)$ is defined as the set of \emph{strictly positive} iterations of $x$, and a similar convention is used here for transfinite orbits (in Def. \ref{transfinite}, the inductive procedure starts at 1 and not at 0). Notice that the equivalence stated here fails if one, instead, includes by definition the point $x$ itself in the orbit $\mathcal{O}(x)$. In this case, the additional assumption of $X$ having no isolated points is needed (see \cite{silverman1992maps}, Proposition 1.1).
The ``only if" part of the result, however is not true for $\lambda$-regular TDSs (even assuming sequential continuity), as the following example shows.

\begin{exa}
 Let $X$ and $Y$ be two disjoint compact metric spaces with distances respectively $d_X$ and $d_Y$. Define $d:X\cup Y\to\mathbb{R}_0^+$ as $d|_X:=d_X$, $d|_Y:=d_Y$, $d(x,y)=1$ if $(x,y)\in X\times Y$.

Let us define a TDS  $(X\cup Y,\{f\})$ with the following properties:
\begin{enumerate}
\item $X$ and $Y$ are $f$-invariant;
\item The restrictions $f_{|_X}$ and $f_{|_Y}$ are continuous;
\item there exists $y\in Y$ such that $\overline{\mathcal{O}(y)}=Y$;
\item there exists $x\in X$ such that $\overline{\mathcal{O}(x)}=X$ and $\{f\}^\omega(x)=y$;
\item $D(z)\le\omega$ for all $z\in (X\setminus\{x\})\cup Y$.
\end{enumerate}
Under these assumptions, it follows that $(X\cup Y,\{f\})$ is a (sequentially continuous, if one wants) $(\omega\cdot 2)$-normal TDS.
By items 3. and 4. we have that $$\overline{\mathcal{O}_{\{\omega\cdot 2\}}(x)}=X\cup Y.$$ On the other hand, items 1. and 2. imply that $(u,v)\notin (\omega\cdot 2)\{\mathcal{N}\}$ for every $u\in Y$ and $v\in X$.
Notice finally that, for every $z\in X$, we have $(x,z)\in(\omega\cdot 2)\{\mathcal{R}\}$ but the recurrence is not strict nor proper.
\end{exa}

In case of a finite dynamical system $(X,f)$, an invariant set can be defined (just) as a subset $V\subseteq X$ such that $f(V)\subseteq V$, because then we are sure that every point stays in $V$ ``forever". In a TDS, however, the condition $\{f\}(V)\subseteq V$ is not enough, because then $\{f\}^\beta(V)$ can have nonempty intersection with $X\setminus V$ for some ordinal $\beta\ge \omega$. This is due to the fact that, in the additive monoid $\omega_1$, the submonoid $\omega$ has just one generator (the ordinal 1) so it is sufficient to assume that the iteration corresponding to this generator (i.e., $f^1$) stays in $V$ to be sure that every other iteration of finite order will stay there. However, this fails when $\beta\ge\omega$, since then the smallest submonoid in $\omega_1$ containing $\beta$ will have at least two (up to countably many) generators. This motivates the following definition.   

\begin{defi}[\textbf{Transfinite invariance}]
Let $A$ be a subset of $X$. We say that $A$ is \emph{$\lambda$-invariant} with respect to $\{f\}$ if
\begin{equation}\label{eq_inv}
    \{f\}^\beta(A)\subseteq A   \qquad \forall\beta < \lambda.
\end{equation}
Moreover, $A$ is said \emph{strongly $\lambda$-invariant} with respect to $\{f\}$ if all inclusions in \eqref{eq_inv} are equalities.
\end{defi}
\begin{prop}\label{orbit_invariant_}
    Let $(X,\{f\})$ be a TDS and $\lambda$ be an additively indecomposable countable ordinal. Then, for every $x\in X$ such that $\lambda\le D(x)$: 
    \begin{enumerate}
        \item $\mathcal{O}_{\{\lambda\}}(x)$ is $\lambda$-invariant; 
        \item $\overline{\mathcal{O}_{\{\lambda\}}(x)}$ is $\lambda$-invariant if $\{f\}$ is $\lambda$-continuous at $x$.
    \end{enumerate}

\end{prop} 
\begin{proof}
\begin{enumerate}
    \item If $y\in\mathcal{O}_{\{\lambda\}}(x)$, then there is $\beta<\lambda$ such that $y=\{f\}^\beta(x)$. For every $\eta<\lambda$, we have $\beta+\eta=\iota$ for some $\iota<\lambda\le D(x)$. Therefore, $\{f\}^\eta(\{f\}^\beta(x))=\{f\}^{\beta+\eta}(x)=\{f\}^\iota(x)\ne\emptyset$ and $\{f\}^\iota(x)\in\mathcal{O}_{\{\lambda\}}(x)$;
    \item Assume $y\in \overline{\mathcal{O}_{\{\lambda\}}(x)}$. Then there is a sequence of ordinals $\{\beta_k\}_{k\in\mathbb{N}}$ such that $\sup_{k\in\mathbb{N}}\beta_k=\lambda$, and $\lim_{k\to\infty}\{f\}^{\beta_k}(x)=y$. Take $\eta<\lambda$ and assume $\{f\}^\eta(y)\ne\emptyset$. We have $\sup_{k}(\beta_k+\eta)<\lambda$. Recalling that $\lambda\le D(x)$, by $\lambda$-continuity we have: $$\emptyset\ne \{f\}^{\beta_k+\eta}(x)=\{f\}^\eta(\{f\}^{\beta_k}(x))\to \{f\}^\eta(y)$$ as $k\to\infty$, so that $\{f\}^\eta(y)\in \overline{\mathcal{O}_{\{\lambda\}}(x)}$.
\end{enumerate}
\end{proof}
A case of a $\lambda$-orbit not $\lambda$-invariant is in Example \ref{exa6}, where $\mathcal{O}_{\{\omega\cdot 2\}}(z_0)$ is not $(\omega\cdot 2)$-invariant.

\begin{defi}\label{subsystem_}
A \emph{$\lambda$-subsystem} of the TDS $(X,\{f\})$ is a closed, $\lambda$-invariant subset $Y\subseteq X$. We say that the $\lambda$-subsystem is \emph{proper} if $Y\subsetneq X$. 
\end{defi}
Clearly, being $\omega$-invariant means being invariant in the usual sense. Therefore, an $\omega$-subsystem is a subsystem which, in itself, is a finite dynamical system.
In a finite dynamical system $(X,f)$, the set $\overline{\mathcal{O}(x)}$ is invariant so that 
$$(X,f) \text{ is not minimal} \iff \overline{\mathcal{O}(x)}\subsetneq X \text{ for some }x\in X \implies \text{$X$ has a proper subsystem}$$
In a TDS, the closure $\overline{\mathcal{O}_{\{\lambda\}}(x)}$ of the $\lambda$-orbit of a point $x\in X$ need not to be $\lambda$-invariant, so in general one can have systems without proper $\lambda$-subsystems which are not $\lambda$-minimal.  
On the other hand, if $(X,\{f\})$ is $\lambda$-minimal, then there are no proper $\lambda$-subsystems:
\begin{prop}\label{_minimal__}
If a TDS $(X,\{f\})$ is $\lambda$-minimal then it has no proper $\lambda$-subsystems.
\end{prop}
\begin{proof}
    Assume that $(X,\{f\})$ is $\lambda$-minimal. Suppose that $Y$ is closed and $\lambda$-invariant and pick $x\in Y$. By $\lambda$-invariance we have $\{f\}^\beta(x)\in Y$ for every $\beta<\lambda$, so that $\mathcal{O}_{\{\lambda\}}(x)\subseteq Y$. Since $Y$ is closed, we have $\overline{\mathcal{O}_{\{\lambda\}}(x)}\subseteq Y$ as well, which is a contradiction.
\end{proof}

In a finite dynamical system, the diagonal of the recurrence relation is invariant: $f(|\mathcal{R}|)\subseteq |\mathcal{R}|$ (see for instance \cite{kurka2003topological}, Proposition 2.17). Let us show that this is generally not the case for TDSs.

\begin{exa}
In this example we show that in a sequentially continuous TDS in general $|\lambda\{\mathcal{R}\}|$ is not $\lambda$-invariant. Let $z,y\in\mathbb{I}$ and $f:\mathbb{I}\circlearrowleft$ a continuous map such that
\begin{itemize}
\item $\overline{\mathcal{O}(z)}=\mathbb{I}$;
\item $f(y)$ is a pre-periodic point.
\end{itemize}
For every $n\in\mathbb{N}$ consider the closed interval $I_n:=[y-1/n, y+1/n]$. Then there exist
$$
k_n=\min \{k\in\mathbb{N} : f^k(z)\in I_n\}
\qquad\text{and}\qquad 
\delta_n=\min_{j<k_n} d\left( f^j(z),y\right).
$$
Set $A_n:=(y-\delta_n, y+\delta_n)$. 
Since $\delta_n>1/n$, we set $U_n:=A_n\setminus I_n$. Let us recall now the well-known Tietze Extension Theorem (see for instance \cite{willard2012general}, p. 99-108):
\begin{fact}[Tietze Extension Theorem]\label{th_tieze}
A topological space $X$ is normal if and only if for any closed set $C\subset X$ and given any continuous function $f :C\rightarrow \mathbb{I}$, there exists a continuous function $\hat{f}:X\rightarrow\mathbb{I}$ which extends the map $f$.
\qed
\end{fact}

We can now define, for any $n\in\mathbb{N}$, the function $f_n:X\rightarrow X$ as
\begin{equation}
f_n(x)=
\begin{cases}
f(x) \quad & \text{if } x\in \mathbb{I}\setminus A_n\\
\hat{f}_n(x) & \text{if } x\in U_n\\
f(y) & \text{if } x\in I_n
\end{cases}
\end{equation}
where $\hat{f}_n: \mathbb{I}\circlearrowleft$ is a continuous extension of $f_n|_{\mathbb{I}\setminus U_n}$, which exists by Tietze Extension Theorem (of course $\mathbb{I}$ is a normal space). Thus, we have defined the sequentially continuous TDS $(\mathbb{I},\{f\})$. 
Take $\lambda=\omega\cdot 2$.
We observe that $z\in |\mathcal{R}|\subseteq |\lambda\{\mathcal{R}\}|$. Since $\{f\}^\omega(z)=f(y)$ and $f(y)$ is pre-periodic, we have that $\{f\}^\omega(z)\notin |\lambda\{\mathcal{R}\}|$. Thus, $|\lambda\{\mathcal{R}\}|$ is not $\lambda$-invariant.
\end{exa}

\section{Transfinite limit sets and attractors}\label{sec6}

\begin{defi}[\textbf{Transfinite limit sets}]
\label{def_lambda_limits}
Let $(X,\{f\})$ be a TDS, $\lambda$ a countable limit ordinal, and consider $Y\subseteq X$. 
\begin{enumerate}

\item The \emph{proper $\lambda$-limit} of $Y$ is defined as
\begin{equation}\label{prop_limit}
p\lambda(Y):=\bigcap_{\eta<\lambda}\,\,\overline{\bigcup_{\eta<\beta<\lambda}\{f\}^\beta(Y)}.
\end{equation}

\item  The \emph{extended $\lambda$-limit} of $Y$ is defined as
\begin{equation}\label{ext_limit}
e\lambda(Y):=\bigcup_{\beta\le\lambda} \left( \bigcap_{\eta<\beta}\,\,\overline{\bigcup_{\eta<\iota<\beta}\{f\}^\iota(Y)}\right).
\end{equation}
\end{enumerate}
\end{defi}
If $\lambda=\omega$ we obtain that the classical definition of $\omega$-limit set given by Eq. \eqref{def_omegalimit},  that is for any $Y\subseteq X$,
\begin{equation*}
    p\omega(Y)=e\omega(Y)=\omega_f(Y).
\end{equation*}

In general, for a nonempty $Y\subseteq X$ we have that $p\lambda(Y)$ can be empty for a limit ordinal $\lambda>\omega$, while $e\lambda(Y)$ is always nonempty, since it contains $\omega(Y)$. 
The following result generalizes to the transfinite case a well-known characterization holding for finite attractors.

\begin{prop}\label{th_plim_seq}
Let $\lambda$ be a limit ordinal and $Y\subseteq X$ be a nonempty set. A point $y\in X$ belongs to $p\lambda(Y)$ if and only if there exists a strictly increasing sequence of ordinals $\{\beta_j\}_{j\in\mathbb{N}}$ and a sequence of points $\{y_j\}_{j\in \mathbb{N}}\subseteq Y$ such that
\begin{itemize}
\item $\sup_{j\in\mathbb{N}}\beta_j= \lambda$, and
\item the sequence $\{f\}^{\beta_j}(y_j)$ converges to $y$ as $j\to\infty$.
\end{itemize}
\end{prop}
\begin{proof}
    Take $y\in p\lambda(Y)$.  We construct the sequences $\{y_j\}_{j\in\mathbb{N}}$ and $\{\beta_j\}_{j\in\mathbb{N}}$ by induction.  Since $Y\neq \emptyset$, we take $y_0\in Y$ and set $\beta_0:=0$.  Suppose that $y_{j-1}$ and $\beta_{j-1}$ have been constructed. Since $y\in p\lambda(Y)$, we have that 
    \[
    y\in \overline{\bigcup_{\beta_{j-1}<\beta<\lambda}\{f\}^\beta(Y)}.
    \]
    Therefore, setting $\epsilon_j=1/j$, we have:
    \[
    B_{\epsilon_j}(y)\cap \big(\bigcup_{\beta_{j-1}<\beta<\lambda} \{f\}^\beta(Y)\big) \neq \emptyset \quad \forall j\in\mathbb{N}.
    \]
    Recalling that $\lambda$ is a limit ordinal, there exist $y_j\in Y$ and $\beta_{j-1}<\beta_j<\lambda$ such that 
    $$
    d(\{f\}^{\beta_j}(y_j),y)<\epsilon_j.
    $$
    Since $\epsilon_j\to 0$ as $j\to\infty$, we have $\lim_{j\to\infty}\{f\}^{\beta_j}(y_j)=y$. 
    
    Conversely,  let us assume that there exists a sequence of ordinals $\{\beta_j\}_{j\in\mathbb{N}}$ and a sequence $\{y_j\}_{j\in\mathbb{N}}\subseteq Y$ that verify our assumptions. Thus, for every neighborhood $U(y)$  of $y$ and for any given $\beta<\lambda$, there exist $y_j\in Y$ and $\beta<\beta_j<\lambda$ such that $\{f\}^{\beta_j}(y_j)\in U(y)$.  
    Hence $y\in\overline{\bigcup_{\beta<\alpha<\lambda} \{f\}^\alpha(Y)}$,
    and since the latter holds for every $\beta<\lambda$, it follows that: 
    \begin{equation*}   y\in\bigcap_{\beta<\lambda}\overline{\bigcup_{\beta<\alpha<\lambda} \{f\}^\alpha(Y)}.
    \end{equation*}
\end{proof}
The following Corollary is an immediate consequence of the previous result, generalizing it to extended transfinite limit sets.
\begin{cor}
Let $\lambda$ be a limit ordinal and $Y\subseteq X$ be a nonempty set. A point $y\in X$ belongs to $e\lambda(Y)$ if and only if there exists an increasing sequence of ordinals $\{\beta_j\}_{j\in\mathbb{N}}$ and a sequence of points $\{y_j\}_{j\in \mathbb{N}}\subseteq Y$ such that
\begin{itemize}
\item $\sup_{j\in\mathbb{N}}\beta_j = \beta\le \lambda$, and
\item the sequence $\{f\}^{\beta_j}(y_j)$ converges to $y$ as $j\to\infty$.
\end{itemize}
\end{cor}
\begin{proof}
    Take $y\in e\lambda(Y)$. Then there exists a limit ordinal $\beta\le \lambda$ such that $y\in p\beta(Y)$.  By Proposition \ref{th_plim_seq}, there exists an increasing sequence $\{\beta_j\}_{j\in\mathbb{N}}$ of ordinals and a sequence of points $\{y_j\}_{j\in\mathbb{N}}$ in $Y$ such that $\sup_{j\in\mathbb{N}}\beta_j= \beta$ and $\{f\}^{\beta_j}(y_j)\to y$ as $j\to\infty$. 

    Conversely, assume that there is a sequence  $\{\beta_j\}_{j\in\mathbb{N}}$ and a sequence $\{y_j\}_{j\in\mathbb{N}}\subseteq Y$ that verify our assumptions for some $\beta\le \lambda$.   By Proposition \ref{th_plim_seq}, it follows that $y\in p\beta(Y)\subseteq e\lambda(Y)$.
\end{proof}

\begin{prop}\label{prop_inv}
Let $(X,\{f\})$ be a TDS. Let $Y\subseteq X$ be a nonempty set and $\lambda$ be a limit ordinal. Then $p\lambda(Y)$ is closed and 
$$f(p\lambda(Y))\subseteq p\lambda(Y).$$
\end{prop}
\begin{proof}
The closure of the set $p\lambda(Y)$ follows from the fact that it is an intersection of closed sets. 
Assume now that $p\lambda(Y)\ne\emptyset$ (otherwise invariance is vacuously true). Take $y\in p\lambda(Y)$.
By Proposition \ref{th_plim_seq}, there exists an increasing sequence of ordinals $\{\beta_j\}_{j\in\mathbb{N}}$ and a sequence of points $\{y_j\}_{\mathbb{N}}\subseteq Y$ such that $\sup_{j\in\mathbb{N}}\beta_j=\lambda$ and $\{f\}^{\beta_j}(y_j)\to y$ as $j\to \infty$.  Since $f$ is continuous we have that 
$$\{f\}^{\beta_j+1}(y_j)\longrightarrow f(y)\quad \text{as }j\to \infty.$$
Since $\sup_{j\in\mathbb{N}}\{\beta_j+1\}=\lambda$, by Proposition \ref{th_plim_seq} it follows that $f(y)\in p\lambda(Y)$.
\end{proof}

\begin{exa}\label{exa6}
    This example shows that, in general, a proper $\lambda$-limit is not $\lambda$-invariant, that is there exists $\beta<\lambda$ such that $ \{f\}^\beta(P)\not\subseteq P$.
    
    Let $x_0=1/3$, $x_1=2/3$,  $p\in (0,x_0)$ and $q\in (x_0,x_1)$. Set $0<m<1$ and set $g_0(x)=m x$. Let us define the functions $g:[0,x_0)\circlearrowleft$, $h:[x_0,x_1)\circlearrowleft$ and $l:[x_1,1]\circlearrowleft$ as follows
    \begin{align*}
        g(x)&=
    \begin{cases} 
    g_0(x) &\quad \text{if }x\in [0,p)\\ 
    \frac{g_0(p) -x_0}{p-x_0}(x-x_0)+x_0  &\quad \text{if }x\in[p,x_0)
    \end{cases} \\   
    h(x)&=
    \begin{cases} 
    m (x-x_0)+x_0 &\quad \text{if }x\in [x_0,q)\\ 
    \frac{m(q-x_0) +x_0-x_1}{q-x_1}(x-x_1)+x_1  &\quad \text{if }x\in[q,x_1)
    \end{cases}\\    
     l(x)&=4 \frac{x-x_1}{1-x_1}(1-x)+x_1    
    \end{align*}
    Let $f:\mathbb{I}\circlearrowleft$ be the function given by
    \begin{equation*}
        f(x)=
    \begin{cases} 
    g(x) &\quad \text{if }x\in [0,x_0)\\ 
    h(x)  &\quad \text{if }x\in[x_0,x_1)\\
    l(x)  &\quad \text{if }x\in[x_1,1]
    \end{cases} 
    \end{equation*}
    Let $z_0\in (0,p)$ and $U=(a,b)\subseteq(0,p)$ be such that $z_0\in U$ and $f(b) <a$. For every $n\in\mathbb{N}$ we set $z_n:=f^n(z_0)$ and let $\{a_n\}_{n\in\mathbb{N}}$ and $\{b_n\}_{n\in\mathbb{N}}$ be two sequences of real numbers that converge to $z_0$ and such that 
    \begin{itemize}
        \item $\{a_n\}_{n\in\mathbb{N}}$ is strictly increasing and $a_n\in (a,z_0)$ for all $n\in\mathbb{N}$;
        \item $\{b_n\}_{n\in\mathbb{N}}$ is strictly decreasing and $b_n\in (z_0,b)$ for all $n\in\mathbb{N}$.
    \end{itemize}
    For every $n\in\mathbb{N}$, we set $u_n=g^n(a_n)$, $v_n=g^n(b_n)$ and $U_n=[u_n,v_n]$. Let $y_0\in (x_1,1)$ be a transitive point for the function $l(x)$, that is $\overline{\mathcal{O}(y_0)}=[x_1,1]$. Let us define the sequence of functions $g_n:[0,x_0)\rightarrow \mathbb{I}$ as follows 
    \begin{equation}
\label{gn}
g_n(x)=
\begin{cases}
g(x) &\quad \text{if } x\in [0,x_0)\setminus U_n\\
 p_n(x-z_n)+y_0 &\quad \text{if } x\in [u_n,z_n)\\
r_n(x-z_n)+y_0 &\quad \text{if } x\in [z_n,v_n]\\
\end{cases}
\end{equation}
where
\[
p_n=\frac{y_0- g(u_n)}{z_n-u_n}>0,\qquad r_n=\frac{y_0- g(v_n)}{z_n-v_n}<0.
\]
Set $I_n:=[c_n,d_n]=[l^n(y_0)-\epsilon_n,l^n(y_0)+\epsilon_n]$. For every $n\in\mathbb{N}$ we take $\epsilon_n>0$ sufficiently small that $I_n\cap l^k(y_0)=\emptyset$ for all $k<n$. Take $y_1\in (x_0,q)$ and let us define the sequence of functions $l_n:[x_1,1]\rightarrow\mathbb{I}$ as follows 
\begin{equation*}
\label{ln}
l_n(x)=
\begin{cases}
l(x) &\quad \text{if } x\in [x_1,1]\setminus I_n\\
\frac{y_1-l(c_n)}{l^n(y_0)-c_n}(x-c_n)+l(c_n) &\quad \text{if } x\in [c_n,l^n(y_0))\\
\frac{y_1-l(d_n)}{l^n(y_0)-d_n}(x-d_n)+l(d_n) &\quad \text{if } x\in [l^n(y_0),d_n]
\end{cases}
\end{equation*}
Let us now define the sequence of functions $f_n:\mathbb{I}\circlearrowleft$ as follows (see Fig. \ref{fig_comp})
\begin{equation*}
        f_n(x)=
    \begin{cases} 
    g_n(x) &\quad \text{if }x\in [0,x_0)\\ 
    h(x)  &\quad \text{if }x\in[x_0,x_1)\\
    l_n(x)  &\quad \text{if }x\in[x_1,1]
    \end{cases} 
    \end{equation*}
We have: $f_n\dot{\longrightarrow}  f$ as $n\to\infty$. Therefore, we have defined the TDS $(\mathbb{I},\{f\})$. Moreover, we have that $\{f\}^\omega(z_0)=y_0$, $\{f\}^\omega(y_0)=y_1$ (and thus $\{f\}^{\omega\cdot 2}(z_0)=y_1$). Observe further that $\{f\}^\omega(x)=\emptyset$ for all $x\in [x_0,x_1]$.
Set $\lambda:= \omega\cdot 2$. Since $\omega_f(y_1)=x_0$, we have $P:=p\lambda(\mathbb{I})=[x_1,1]\cup \{x_0\}$.
However, since $\{f\}^\omega(y_0)=y_1\notin P$, we have that $\{f\}^\omega(P)\not\subseteq P$.

\begin{figure}[H]
\centering
\fbox{\includegraphics[width=10cm]{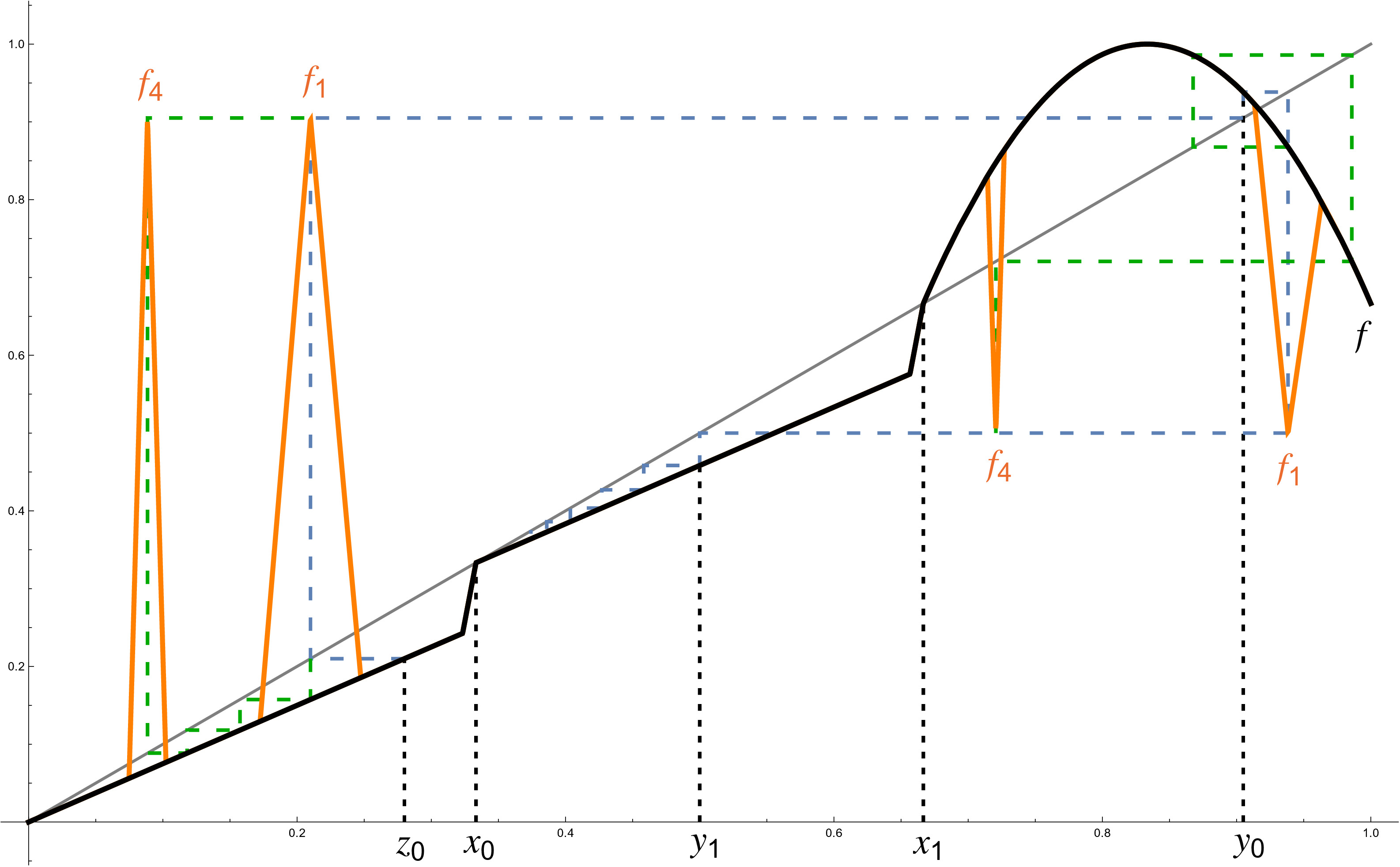}}
\caption{the system defined in Example  \ref{exa6}.}
\label{fig_comp}
\end{figure}

\end{exa}

\begin{defi}[\textbf{Transfinitely inward sets}]\label{inward__}
Let $\lambda$ be a countable limit ordinal. We say that $V\subseteq X$ is a \emph{$\lambda$-inward set} if 
\[
\{f\}^\beta (\overline{V})\subseteq V^\circ,\qquad \forall \beta<\lambda.
\]
Moreover, if there exists $\delta>0$ such that
\begin{equation*}
    \inf_{\beta<\lambda} d\left(\{f\}^\beta(\overline{V}),\partial V\right)>\delta,
\end{equation*}
then $V$ is called a \emph{uniformly $\lambda$-inward set}.
\end{defi}
Notice that, if $X$ has dimension zero, we have $\partial V=\emptyset$ and therefore $d(A,\partial V)=+\infty$ for every nonempty $A\subseteq X$. Notice also that the whole space $X$ is always uniformly $\lambda$-inward in $(X,\{f\})$. 

Below level $\omega^\omega$, transfinitely inward sets are uniformly inward provided a basic regularity condition, as we prove in the following proposition.

\begin{prop}\label{inward_}
Let $(X,\{f\})$ be a TDS and $V\subseteq X$ a $\lambda$-inward set for some limit ordinal $\lambda<\omega^\omega$. If $V$ is  $\lambda$-closed, then $V$ is uniformly $\lambda$-inward.
\end{prop}
\begin{proof}
Set $\lambda=\omega^{k_1}\cdot h_1+\ldots+ \omega^{k_n}\cdot h_n$ where $n\in\mathbb{N}$, $h_i,k_i$ are natural numbers for all $i=1,\ldots,n$ and $k_1>\ldots>k_n>0$. Observe first that, recalling the composition law \eqref{composition1}-\eqref{composition2}, if $\beta<\lambda$ has the form $\beta=\omega^{a_1}\cdot c_1+\ldots + \omega^{a_m}\cdot c_m$ with $m\le n$, it follows that:
$$
\{f\}^\beta(\overline{V})\subseteq \{f\}^{\omega^{a_m}\cdot c_m}(\overline{V})\subseteq \{f\}^{\omega^{a_m}}(\overline{V}),
$$ 
where $0\le a_m \le k_1$. Therefore, we have:
\begin{equation}
\label{distance_le}
d(\{f\}^\beta(\overline{V}),\partial V)\ge d(\{f\}^{\omega^{a_m}\cdot c_m}(\overline{V}),\partial V)\ge d(\{f\}^{\omega^{a_m}}(\overline{V}),\partial V).
\end{equation}
Moreover, since $V$ is $\lambda$-closed, 
$$
S_j:=\{f\}^{\omega^j}(\overline{V})=\{f\}^{\omega^j}(X^{\omega^j+1}\cap \overline{V})
$$ 
is a closed set included in $V^{\circ}$ for every integer $j$ such that $0\le j\le k_1$, so that $d(S_j,\partial V)=:d_j>0$ for every $j$. Therefore, by \eqref{distance_le}, we have:
$$\inf_{\beta<\lambda} d\left(\{f\}^\beta(\overline{V}),\partial V\right)= \min \{d_0,\dots, d_{k_1}\}>0.$$
\end{proof}

Let us recall the definition of a stable set in finite dynamical systems.
\begin{defi}
    Let $(X,f)$ be a topological dynamical system. A set $S\subseteq X$ is \emph{stable} if for all $\epsilon>0$, there exists $\delta>0$ such that for all $x\in B_\delta(S)$,
    \[
    d(f^n(x),S)<\epsilon \quad  \forall n\in\mathbb{N}.
    \]
    In the following, when we consider a TDS $(X,\{f\})$, we will say that a set $S\subseteq X$ is stable if it is such in the above sense for the finite dynamical system $(X,f)$. 
\end{defi}
In finite topological dynamical systems, attractors are usually defined as nonempty, closed, strongly invariant sets which are stable and attractive (see for instance \cite[Def. 2.60]{kurka2003topological}). Moreover, one can prove that every attractor is the $\omega$-limit set of an inward set, and vice-versa (\cite[Propositions 2.64 and 2.65]{kurka2003topological}). In our context, it is convenient to take the latter as the \textit{definition} of transfinite attractors, and then refine the concept to let further dynamical properties arise. 

\begin{defi}[\textbf{Transfinite attractors}] \label{defi_attractors}
Let $V\subseteq X$ be a $\lambda$-inward set. 
\\
We say that $A\subseteq X$ is a \emph{proper $\lambda$-attractor} if $A=p\lambda(V)$ and there is no $\beta<\lambda$ such that $A=p\beta(V)$. We say that $A\subseteq X$ is an \emph{extended $\lambda$-attractor} if $A=e\lambda(V)$. We say that the $\lambda$-attractor (proper or extended) is \emph{uniform} if $V$ is uniformly $\lambda$-inward.
\end{defi}
Note that every attractor in a finite dynamical system is a at the same time a proper $\omega$-attractor and an extended $\omega$-attractor. Note also that every global $\lambda$-attractor, that is an attractor which is the $\lambda$-limit of the whole space $X$, is uniform.

\begin{lem}
Let $V\subseteq X$ be a $\lambda$-inward set. Set $\Gamma_\lambda:=\{\beta\le \lambda : p\beta(V) \text{ is a proper $\beta$-attractor}\}$. Then $e\lambda(V)=\bigcup_{\beta\in \Gamma_\lambda} p\beta(V)$.
\end{lem}
\begin{proof}
    By Def. \ref{def_lambda_limits}, we have that $e\lambda(V)=\bigcup_{\beta\le\lambda}p\beta(V)\supseteq \bigcup_{\beta\in \Gamma_\lambda} p\beta(V)$.
    It remains to prove the converse inclusion. 
    Let $y\in e\lambda(V)$ and suppose, by absurd, that $y\notin \bigcup_{\beta\in \Gamma_\lambda} p\beta(V)$. Then, there exists $$\gamma=\min \{ \beta\le \lambda : \beta\notin \Gamma_\lambda \text{ and } y\in p\beta(V) \}.$$ 
    Thus, $y\in p\gamma(V)$ and since $\gamma\notin \Gamma_\lambda$, there exists $\eta<\gamma$ such that $p\gamma(V)=p\eta(V)$. Hence $y\in p\eta(V)$ and since $y\notin \bigcup_{\beta\in \Gamma_\lambda} p\beta(V)$ we have that $\eta\notin \Gamma_\lambda$, which is a contradiction.
\end{proof}

\begin{defi}
  Let $Y$ be a proper $\lambda$-attractor. We set
\begin{equation*}
    \mathcal{B}(Y):=\{x\in X : d(\{f\}^\beta(x),Y)\longrightarrow 0 \text{ as } \beta\to \lambda \},
\end{equation*}
where the convergence is meant in the sense that, for every $\epsilon>0$, there is $\eta<\lambda$ such that $d(\{f\}^\beta(x),Y)<\epsilon$ for all $\beta$ such that $\eta<\beta<\lambda$.
  We call $\mathcal{B}(Y)$ the \emph{basin} of the attractor $Y$.
  
  Let $Z$ be an extended $\lambda$-attractor. Let $\{\beta_i\}_{i\in\mathbb{N}}$ be the set of all limit ordinals smaller than or equal to $\lambda$, and let $Y_i$ indicate the proper $\beta_i$-attractor.
  We set:
  \begin{equation*}
      \mathcal{B}(Z):=\bigcup_{i\in\mathbb{N}} \mathcal{B}(Y_i).
      \end{equation*}
\end{defi}
    For finite topological systems, the basin of an attractor is always an open set. In case of TDSs, even assuming sequential continuity and $\lambda$-regularity, $\mathcal{B}(Y)$ may not be an open set in $X$ neither when $Y$ is a proper nor an extended attractor. When $Y$ is a proper $\lambda$-attractor, its basin is in general only open in the relative topology induced on $X^\lambda$. Indeed, as we can see in Example \ref{exa1}, for $\lambda=\omega\cdot 2$ and $h=1$, the $\lambda$-attractor $Y=p\lambda(\mathbb{I})=\{1\}$ and $\mathcal{B}(Y)=\mathcal{O}^{\mathbb{Z}}(z)$. Hence, $\mathcal{B}(Y)$ is not an open set in $\mathbb{I}$, while it is open in $\mathbb{I}^{\omega\cdot 2}$ (in fact it coincides with this set).

\begin{defi}
Let $\lambda$ be a limit ordinal and $S\subseteq X$.  We say that $S$ is \emph{$\lambda$-attracting} if there exists $\delta>0$ such that for all $x\in B_\delta(S)$ we have that 
$
\lim_{\beta\to \lambda} d\left(\{f\}^\beta(x),S\right)=0
$
whenever $\{f\}^\beta(x)\ne\emptyset$ for all $\beta<\lambda$.
\end{defi}

\begin{exa}
\label{ex_attr_cyc}
Fix $0<m<1$. Take $x_0\in (0,m)$ and $z\in (m,1)$. Set $U_0:=(a,b)=(z-\nu,z+\nu)\subset (x_0,1)$. Let us define $f:\mathbb{I}\circlearrowleft$ as $f(x)=m (x-x_0)+x_0$. Set $U_n:=(f^n(a),f^n(b))$ for every $n\in\mathbb{N}$,  and denote by $z_n:=f^n(z)\in U_n$. Let $\{c_n\}_{n\in\mathbb{N}},\{d_n\}_{n\in\mathbb{N}}$ be two sequences in $\mathbb{I}$ which are respectively strictly increasing and strictly decreasing, and such that 
\begin{itemize}
    \item $c_n<a$ and $d_n>b$ for all $n\in\mathbb{N}$;
    \item $c_n\longrightarrow a$ and $d_n\longrightarrow b$ as $n\to\infty$;
    \item for every $n\in\mathbb{N}$, we have $[f^h(c_n),f^h(d_n)]\cap [f^k(c_n),f^k(d_n)]=\emptyset$ whenever $h\neq k$.
\end{itemize}
For every $n\in\mathbb{N}$, we set $A_n:=[f^n(c_n),f^n(a)]$ and $B_n:=[f^n(b),f^n(d_n)]$.

Let us define the sequence $\{f_n\}_{n\in\mathbb{N}}$ of functions $f_n:\mathbb{I}\circlearrowleft$ as follows:
\begin{equation}
f_n(x)=
\begin{cases}
f(x) &\quad \text{if } x\in \mathbb{I}\setminus \left( U_n \cup A_n\cup B_n\right)\\
z-m(z-f^{-n}(x)) &\quad \text{if } x\in (f^n(a),f^n(z)]\\
z-m(f^{-n}(x)-z) &\quad \text{if } x\in (f^n(z),f^n(b))\\
l_n(x-f^n(c_n)) +f^{n+1}(c_n) &\quad \text{if } x\in A_n\\
r_n(x-f^n(d_n))+f^{n+1}(d_n) &\quad \text{if } x\in B_n
\end{cases}
\end{equation}
where 
\[
l_n=\frac{z-m(z-a)-f^{n+1}(c_n)}{f^n(a)-f^n(c_n)},\quad r_n=\frac{z-m(b-z)-f^{n+1}(d_n)}{f^n(b)-f^n(d_n)}.
\]
The functions $\{f_n\}_{n\in\mathbb{N}}$ are continuous for every $n\in\mathbb{N}$, and $f_n \dot{\longrightarrow} f$, where the limit map $f$ is continuous as well. So, we have defined the sequentially continuous TDS $(\mathbb{I},\{f\})$.

Take $y\in \cup_{n\ge 0}U_n$, and for $k\in\mathbb{N}$, set $a^y_k:=z-m^k|z-y|$.  Then, for any $n\in\mathbb{N}$ we have:
\[
\mathcal{O}_n(y)=\{f(y),\ldots,f^n(y),a^y_1, f(a^y_1),\ldots, f^n(a^y_1),a^y_2,\ldots\}.
\]
Note that $\lim_{k\to\infty}a^y_k=z$ and $a^y_k\in (a,z]$ for all $k\in\mathbb{N}$. Moreover, for every $k,n\in\mathbb{N}$, we have that $a^y_k=f_n^{k(n+1)}(y)$, from which it follows that $a^y_k=\{f\}^{\omega\cdot k}(y)$ for every $k\in\mathbb{N}$.
Therefore, for every $h\in\mathbb{N}$, we have:
\[
\lim_{k\to\infty}\{f\}^{\omega\cdot k}(y)=z,\qquad \lim_{k\to\infty}\{f\}^{\omega\cdot k+h}(y)=f^h(z).
\]
On the other hand, if $y\in \mathbb{I}\setminus \cup_{n\ge 0}U_n$, there is $N>0$ such that $y\notin (U_n\cup A_n \cup B_n)$ for all $n>N$. Hence $\{f\}^\omega(y)=\emptyset$.
It follows that the system $(\mathbb{I},\{f\})$ is a sequentially continuous $\omega^2$-TDS.

The set $\mathbb{I}$, coinciding with the whole space, is of course $\omega^2$-inward. In fact, we have the stronger properties $f^k(\mathbb{I})\subsetneq (0,1)$ for all $k\in\mathbb{N}$ and $\{f\}^\beta(\mathbb{I})\subsetneq (0,1)$ for any $\beta<\omega^2$. 
For the latter, it is sufficient to show that $f_n(U_n)\subsetneq(0,1)$ for all $n\in\mathbb{N}$. 
We observe that 
\begin{align*}
&f_n(f^n(a))=z-m(z-a)=z-m\nu \\
&f_n(f^n(z))=z\\
&f_n(f^n(b))=z-m(b-z)=z-m\nu,
\end{align*}
from which it follows that $f_n(U_n)=(z-m\nu ,z]\subseteq (0,1)$. 
Thus, the set $P:=p\omega^2(\mathbb{I})=\{x_0,z,f(z),f^2(z),\ldots\}$ is a proper $\omega^2$-attractor. Note that $x_0$ belongs to $P$ as, for instance,
$\lim_{k\to\infty}\{f\}^{\omega\cdot k+k}(z)=x_0$.
In this case, $P$ is stable. Indeed, since $f$ is a contraction, for any $\epsilon>0$ if we take $0<\delta<\epsilon$ we have that $x\in B_\delta(P)$ implies $d(f(x),f(P))<\delta<\epsilon$.

This example shows that, in general, strong invariance does not hold when $\lambda\ge \omega^2$, as there is no $x\in P$ such that $f(x)=z\in P$.  Moreover observe that, for every $k\in\mathbb{N}$, the set $X^{(\omega\cdot k)+1}$ is not compact. 
Indeed, $X^{\omega\cdot k+1}=\cup_{n\ge 0} U_n$ which is an $F_\sigma$ set but not a closed set, since $x_0\notin X^{\omega\cdot k+1}$. Therefore, the map $\{f\}^{\omega\cdot k}$ need not to be uniformly continuous, and in fact it is not. 
Take indeed $0<\epsilon<m^k(z-a)$. 
Then, for every fixed $\delta>0$, there exists $h\in\mathbb{N}$ such that $f^h(z)-f^h(a)<\delta$. However,  $d(\{f\}^{\omega\cdot k}(f^h(z)), \{f\}^{\omega\cdot k}(f^h(a)))=m^k(z-a)>\epsilon$.
Finally observe that, for $k,h\in\mathbb{N}$, we have $\{f\}^{\omega\cdot k}(\mathbb{I})=[z- m^k\nu,z]$ and thus $$\{f\}^{\omega\cdot k+h}(\mathbb{I})=\{f\}^h(\{f\}^{\omega\cdot k})(\mathbb{I})=f^h([z- m^k\nu,z]),$$ where the latter is clearly a closed set. Therefore, $\mathbb{I}$ is an $\omega^2$-closed set.
 
\begin{figure}[H] 
\centering
{\fbox{\includegraphics[width=10cm]{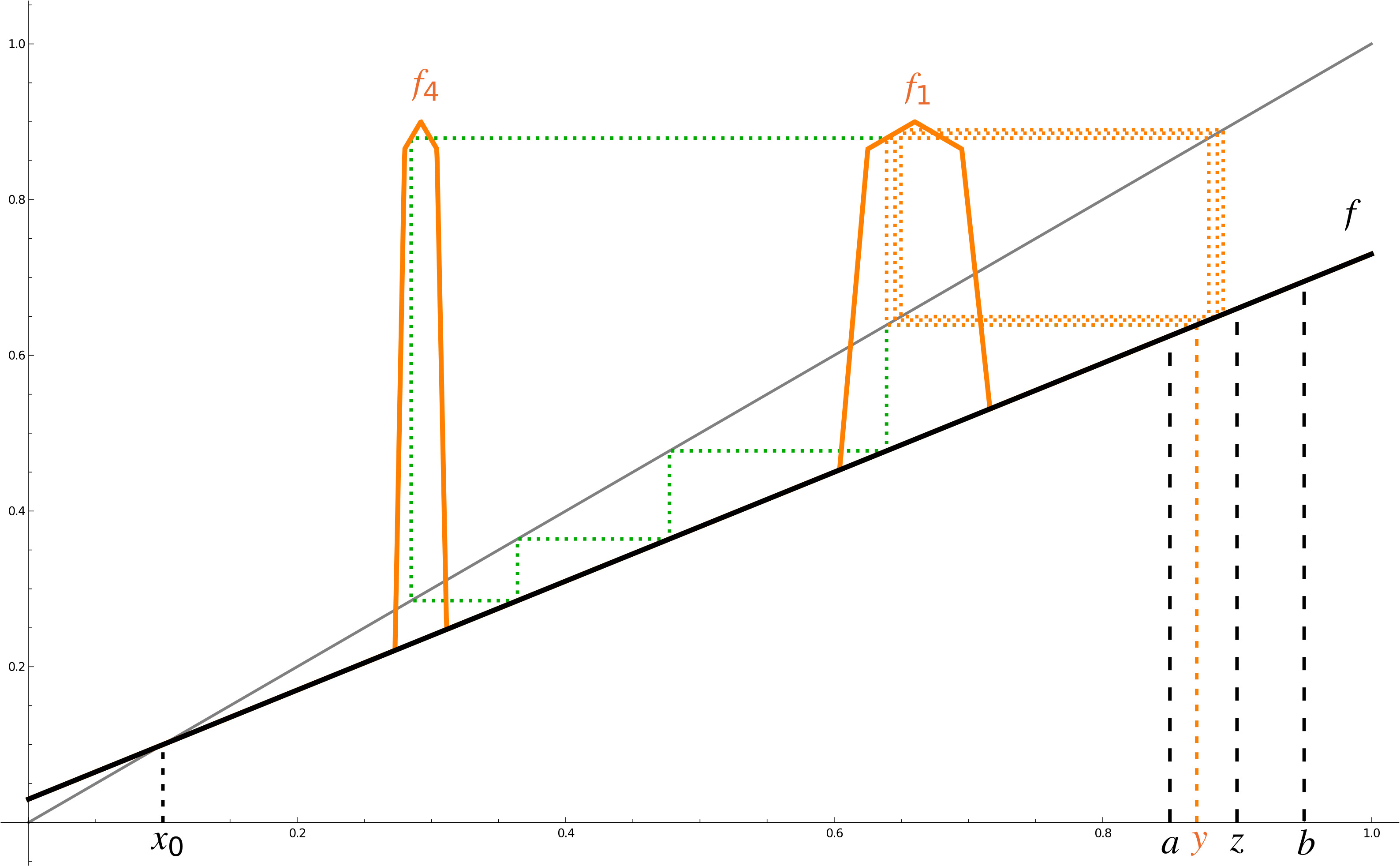}}}
\caption{The system defined in Example \ref{ex_attr_cyc}.}
\label{fig_A2A3_}
\end{figure}

\end{exa}

Proper $\lambda$-limits are clearly closed, being intersection of closed sets. The next result shows that extended $\lambda$-limits, which are clearly $F_\sigma$ sets by Eq. \eqref{ext_limit}, are in fact closed as well.

\begin{prop}\label{closed_attr}
Let $\lambda$ be a limit ordinal and $V$ be a subset of $X$. Then $e\lambda(V)$ is a closed set.
\end{prop}
\begin{proof}
Set $A:=e\lambda(V)$ and take $z\in \overline{A}$. 
Then, there exists a sequence $\{z_n\}_{n\in\mathbb{N}}\subseteq A$ such that $z_n\to z$ as $n\to\infty$. 
Thus, for every $n\in\mathbb{N}$, we have that $z_n\in p\beta_n(V)$ for some $\beta_n\le \lambda$. 
We consider two cases.

\begin{enumerate}
    \item If there is $N\in\mathbb{N}$ and $\beta\le\lambda$ such that $z_n\in p\beta(V)$ for all $n>N$, then $z\in p\beta(V)\subseteq e\lambda(V)$. 
    \item Otherwise, note that the sequence $\{\beta_n\}_{n\in\mathbb{N}}$ cannot have a subsequence $\{\beta_{n_k}\}_{k\in\mathbb{N}}$, which is strictly decreasing. 
It follows that there exist $\beta\le\lambda$ and a strictly increasing subsequence $\{\beta_{n_k}\}_{k\in\mathbb{N}}$ such that $\sup_k\beta_{n_k}=: \beta\le \lambda$. 
We want to show that $z\in p\beta(V)\subseteq e\lambda(V)$. 

For any $k\in\mathbb{N}$, since $z_{n_k}\in p\beta_{n_k}(V)$, by Proposition \ref{th_plim_seq} there exist a sequence of points $\{y^k_{j}\}_{j\in\mathbb{N}}\subseteq V$ and a strictly increasing sequence of ordinals $\{\eta^k_j\}_{j\in\mathbb{N}}$ such that
\[
\{f\}^{\eta^k_j}(y^k_j)\xrightarrow{j\to\infty} z_{n_k} \quad\text{and}\quad \sup_{j\in\mathbb{N}}\eta_j^k=\beta_{n_k}.
\]
Set $j_0:=1$, $\gamma_1:=\eta^1_1$ and $y_1:=y^1_1$. For every $i\ge 2$, we set $\gamma_i:=\eta^i_{j_i}$ and $y_i:=y^i_{j_i}$ where 
\[
j_i=\min\{j\in\mathbb{N} : \eta^i_j>\gamma_{i-1} \text{ and } j>j_{i-1}\}.
\]
By construction, we have:
\begin{itemize}
    \item $\{f\}^{\gamma_i}(y_i)\xrightarrow{i\to\infty} z$;
    \item $y_i\in V$ for all $i\in\mathbb{N}$ and the sequence $\{\gamma_i\}_{i\in\mathbb{N}}$ is strictly increasing;
    \item $\sup_{i\in\mathbb{N}}\gamma_i=\beta$.
\end{itemize}
By Proposition \ref{th_plim_seq}, it follows that $z\in p\beta(V)\subseteq e\lambda(V)$.
\end{enumerate}

\end{proof}
The following two results show some further instances of the fact that, as seen in Remark \ref{inward_}, level $\omega^2$ is critical for the emergence of phenomena which are not possible in finite dynamical systems. Indeed, before $\omega^2$, proper $\lambda$-limits are still strongly invariant, just like  $\omega$-limits are (Proposition \ref{prop_strong_inv}) and moreover, proper $\lambda$-attractors do not contain new points with respect to finite attractors, although they can be proper subsets of them (Proposition \ref{prop9}). 
\begin{prop}
\label{prop_strong_inv}
Let $\lambda<\omega^2$ be a limit ordinal and $S\subseteq X$. Then $p\lambda(S)$ is strongly invariant, i.e. $\{f\}(p\lambda(S))=f(p\lambda(S))= p\lambda(S)$.
\end{prop}
\begin{proof}
If $p\lambda(S)=\emptyset$ the claim is trivial, so assume $p\lambda(S)\ne\emptyset$. By Proposition \ref{prop_inv},  we just have to prove that $f(p\lambda(S)) \supseteq p\lambda(S)$.  Assume that $y\in p\lambda(S)$. Then, there exists a sequence $\{y_j\}_{j\in\mathbb{N}}$ of points of $S$ and a strictly increasing sequence $\{\beta_j\}_{j\in\mathbb{N}}$ of ordinals such that 
\[
\{f\}^{\beta_j}(y_j)\xrightarrow{j\to\infty} y\quad\text{ and }\quad \sup_{j\in\mathbb{N}}\ \beta_j=\lambda.
\]
Since $\lambda<\omega^2$, we have that $\lambda=\omega\cdot k$ for some $k\in\mathbb{N}$. Therefore, there exists $N\in\mathbb{N}$ so large that $\beta_j=\omega\cdot (k-1)+h_j$ for all $j>N$, where $\{h_j\}_{j\in\mathbb{N}}$ is a strictly increasing sequence of positive integers. Let us consider the sequence of points 
$$\{\{f\}^{\omega\cdot(k-1)+h_j-1}(y_j)\}_{j>N}.$$
Since $X$ is compact, this sequence admits a convergent subsequence
\[
\{f\}^{\omega\cdot(k-1)+h_{j_i}-1}(y_{j_i})\xrightarrow{i\to\infty} z,
\]
where the subsequence $\{h_{j_i}\}_{i\in\mathbb{N}}$ is strictly increasing as well.
Since 
$$\sup_{i\in\mathbb{N}}\{\omega\cdot(k-1)+h_{j_i}-1\}=\lambda,$$ by Proposition \ref{th_plim_seq}, it follows that $z\in p\lambda(S)$ and $f(z)=y$.
\end{proof}

\begin{prop}
\label{prop9}
Let $V\subseteq X$ be a $\lambda$-inward set with $\lambda<w^2$ a limit ordinal. Then $p\lambda(V)\subseteq \omega(V)$.
\end{prop}
\begin{proof}
Since $\lambda<\omega^2$, it follows that $\lambda=\omega\cdot k$ for some $k\in\mathbb{N}$. Let us assume that $y\in p\lambda(V)$. By Proposition \ref{th_plim_seq}, there exist an increasing sequence of ordinals $\{\beta_j\}_{j\in\mathbb{N}}$ and a sequence of points $\{y_j\}_{j\in\mathbb{N}}\subseteq V$ such that $\sup_{j\in\mathbb{N}}\beta_j=\lambda$ and $\{f\}^{\beta_j}(y_j)\longrightarrow y$ as $j\to\infty$.

Since $\lambda=\omega\cdot k$, there exists $N\in\mathbb{N}$ sufficiently large that $\beta_j=\omega\cdot (k-1)+h_j$ for all $j>N$, where $\{h_j\}_{j\in\mathbb{N}}$ is an increasing sequence of positive integers.
For every $j>N$, we set 
$$
x_j:=\{f\}^{\omega\cdot (k-1)}(y_j).
$$
Since $V$ is a $\lambda$-inward set, we have that $x_j\in V^\circ$ for all $j>N$. 
From the composition rules \eqref{composition1}-\eqref{composition2}, for every $j>N$, we have: 
$$
\{f\}^{\beta_j}(y_j)=\{f\}^{\omega\cdot (k-1)+h_j}(y_j)=f^{h_j}(x_j)\xrightarrow{j\to\infty}y.
$$ 
Since $x_j\in V^\circ$ for every $j>N$ and $\{h_j\}_{j>N}$ is an increasing sequence of positive integers, it follows that $y\in\omega(V)$.
\end{proof}
The following example shows that, starting from level $\omega^2$, the limit sets contain in general ``new" points with respect to the $\omega$-limit set.

\begin{exa}\label{ex_quad}
    Let $\{x_i\}_{i\in\mathbb{N}}\subseteq \mathbb{I}$ be a strictly increasing sequence such that $x_1=0$ and $x_i\rightarrow 1$ as $i\to\infty$. Set 
    $$I_i:=\{x_i\}\times \mathbb{I} \quad \forall\, i\in\mathbb{N},\qquad I_\infty:=\{1\}\times\mathbb{I},$$
    and
    $$
    X:=I_\infty \cup \bigcup_{ i\in\mathbb{N}}I_i    \subseteq \mathbb{I}^2.
    $$
    Assume on $X$ the Euclidean metric inherited from $\mathbb{R}^2$, so that $X$ is a compact metric space. Let us define the map $f:X\circlearrowleft$ as $$f(x,y)=(x,y/2).$$ 

    For every $n\in\mathbb{N}$, we consider the points $z_{i,n}=f^n(x_i,1/2)=(x_i,1/2^{n+1})$. Let $\{\epsilon_n\}_{n\in\mathbb{N}}$ be a sequence of positive numbers such that, for every $n\in\mathbb{N}$,
    $$\{x_i\}\times U_n \subset \{x_i\}\times \left[\frac{1}{2^n},\frac{1}{2^{n+2}}\right],$$
    where $U_n=\left[\frac{1}{2^{n+1}}-\epsilon_n,\frac{1}{2^{n+1}}+\epsilon_n \right]$.  
    Let us now define a sequence $\{f_n\}_{n\in\mathbb{N}}$ of functions $f_n:X\circlearrowleft$ as follows (see Fig. \ref{fig_quad}):
    \begin{equation*}
        f_n(x,y)=
        \begin{cases}
        (x_{i+1},1) \quad&\text{if }(x,y)\in \{x_i\}\times U_n\text{ for some $i\in\mathbb{N}$}  \\
        f(x,y) \quad& \text{otherwise} 
        \end{cases}
    \end{equation*}
Since $\epsilon_n\to 0$ as $n\to\infty$, we have $f_n \dot{\longrightarrow}  f$, where the limit map $f$ is continuous. Therefore we have defined a TDS $(X,\{f\})$. Notice that $$\omega_f(X)=\{(1,0),(x_1,0),\ldots, (x_k,0),\ldots\}.$$ In particular, $\omega_f(X)\subseteq [0,1]\times 0.$
Let us now look at the proper $\omega^2$-limit of $X$. By construction, for $x=(x_i,1)$ with $i\in\mathbb{N}$, we have:
\[
\mathcal{O}_n(x)=\{(x_i,1/2),\ldots, (x_i,1/2^n),(x_{i+1},1),(x_{i+1},1/2),\ldots,(x_{i+1},1/2^n),(x_{i+2},1),\ldots\}.
\]
Then, for every $k\in\mathbb{N}$ and $h\in \mathbb{N}_0$, we have that $\{f\}^{\omega\cdot k +h}(x,y)=(x_{i+k},1/2^h)$ whenever $(x,y)\in \mathcal{O}(x_i,1)\cup \{(x_i,1)\}$ for some $i\in\mathbb{N}$. 
Since $x_i\to 1$ as $i\to\infty$, it follows that
$$p\omega^2(X)=\{(1,0),(1,1),(1,1/2),\ldots,( 1,1/2^k),\ldots\}=\mathcal{O}_f((1,1))\cup \{(1,1)\} \cup \{(1,0)\}\not\subseteq \omega_f(X).$$

    \begin{figure}[H] 
\centering
\subfigure[]
{\fbox{\includegraphics[width=6.9cm]{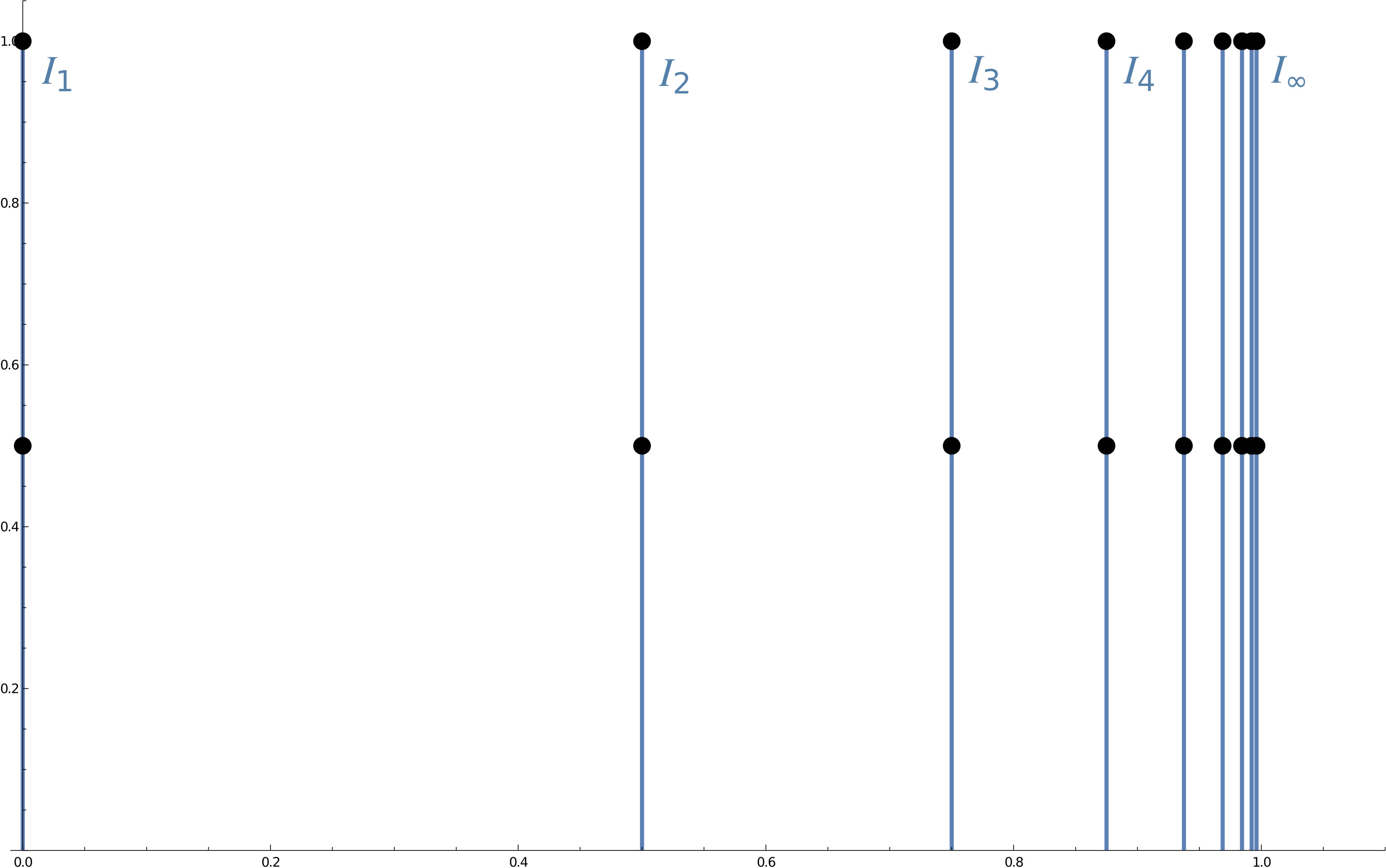}}}
\hspace{0.5mm}
\subfigure[]
{\fbox{\includegraphics[width=6.9cm]{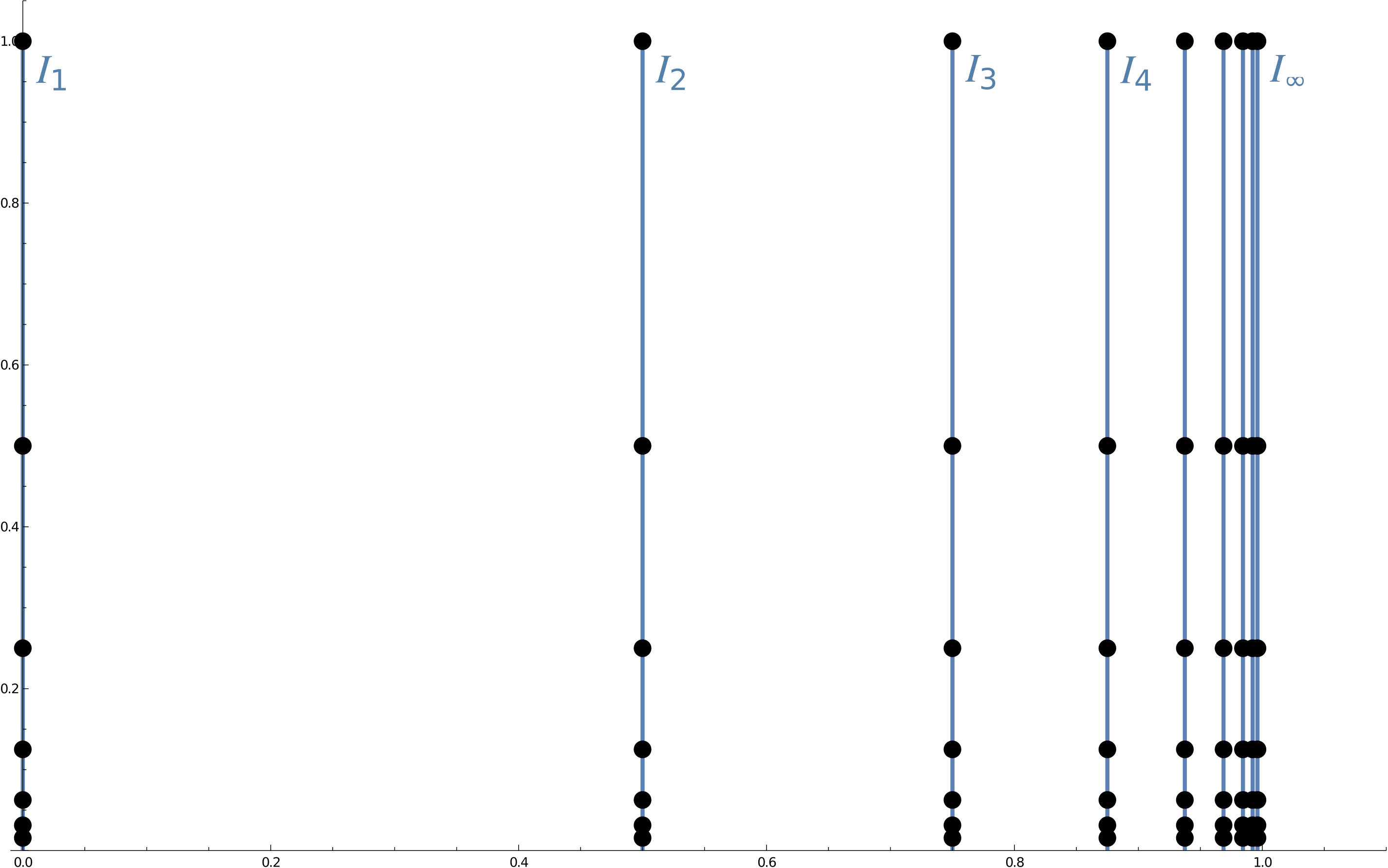}}}
\caption{The system defined in Example \ref{ex_quad}. The dots in (a) and (b) correspond to the iterations according to $f_1$ and $f_6$ respectively.}
\label{fig_quad}
\end{figure}

\end{exa}

\begin{thm}\label{themattractors}
Let $(X,\{f\})$ be a TDS, $\lambda$ a countable limit ordinal and $E,P\subseteq X$ respectively a proper and extended uniform $\lambda$-attractor.
Then:
\begin{enumerate} 
\item $E$ and $P$ are closed and invariant; 
\item $E$ is stable; 
\item $P$ is $\lambda$-attractive.
\end{enumerate}
\end{thm}
\begin{proof}
We assume $E=e\lambda(V)$ and $P=p\lambda(V)$, where $V\subseteq X$ is a uniformly $\lambda$-inward set.
To prove claim 1., it is enough to recall Proposition \ref{closed_attr} and to observe that, by Proposition \ref{prop_inv}, we have that $p\lambda(V)$ is invariant, and since $e\lambda(V)=\bigcup_{\beta\le \lambda} p\beta(V)$, it follows that $e\lambda(V)$ is invariant too. 

Let us now prove claim 2.
Suppose that $E$ is not stable. 
Then, there exists $\epsilon>0$ such that for all $\delta_k=1/k$ there exists $x_k\in X$ such that $d(x_k,E)<\delta_k$ and 
\begin{equation}\label{epsilon_n_k}
    d(f^{n_k}(x_k),E)\ge \epsilon,
    \end{equation}
for some $n_k\in \mathbb{N}$. 
 If $\{n_k\}_{k\in\mathbb{N}}$ was bounded, then some iteration, say $\bar{n}$, would make the inequality \eqref{epsilon_n_k} true for infinitely many values of $k$, but then by continuity of the map $f^{\bar{n}}$ we should have \begin{equation}\label{sublimit}
d(f^{\bar{n}}(\bar{x}),E)\ge\epsilon 
 \end{equation} 
 for any sublimit $\bar{x}$ of the sequence $\{x_k\}_{k\in\mathbb{N}}$. Since $E$ is closed and $d(\bar{x},E)=0$, we have $\bar{x}\in E$, and thus the inequality \eqref{sublimit} contradicts the invariance of $E$ established in point 1.
Therefore, we can assume, passing if needed to a subsequence, that $\{n_k\}_{k\in\mathbb{N}}$ is strictly increasing.

Since $E$ is closed and $E\subseteq V^\circ$, there exists $k_0\in\mathbb{N}$ such that $B_{1/k_0}(E)\subseteq V^\circ$. Hence $x_k\in V^\circ$ for all $k>k_0$ .  Since $X$ is compact, the sequence $f^{n_k}(x_k)$ has a subsequence $\{f^{n_{k_j}}(x_{k_j})\}_{j\in\mathbb{N}}$ converging to some $z$ as $j\to\infty$. Then $z\in \omega(V)\subseteq E$ and  $d(z,E)\ge \epsilon$, which is a contradiction.

Finally, let us prove claim 3.
Since $P$ is closed and $P\subseteq V^\circ$, there exists $\delta>0$ such that $B_\delta(P)\subseteq V^\circ$. Suppose that $P$ is not $\lambda$-attractive, so there exists $x\in B_\delta(P)$ such that $\{f\}^\beta(x)\neq \emptyset$ for all $\beta<\lambda$ and $d(\{f\}^\beta(x),P)$ does not converge to zero as $\beta\to \lambda$. That is, there exist $\epsilon>0$ and a strictly increasing sequence of ordinals $\{\beta_j\}_{j\in\mathbb{N}}$ such that
$
d\left( \{f\}^{\beta_j}(x),P \right)\ge \epsilon$ for all $j\in\mathbb{N}$,
and $\sup_{j\in\mathbb{N}}\beta_j=\lambda$. 
Then, there exists a subsequence $\{\beta_{j_k}\}_{k\in\mathbb{N}}$ such that $\{f\}^{\beta_{j_k}}(x)\longrightarrow y$, and
\[
d\left( \{f\}^{\beta_{j_k}}(x),P\right)\ge \epsilon,\qquad \forall k\in\mathbb{N}.
\]
Since $x\in V^\circ$ and $\sup_{k\in\mathbb{N}}\beta_{j_k}=\lambda$, by Proposition \ref{th_plim_seq}, it follows that $y\in p\lambda(V)$, but $d(y,P)\ge \epsilon$ which is a contradiction.
\end{proof}

\begin{rem}
    If $\lambda<\omega^2$, by Proposition \ref{prop_strong_inv}, we have that a proper $\lambda$-attractor is strongly invariant, which implies that, in this case, the extended $\lambda$-attractor is strongly invariant too.
\end{rem}

Let us come back to Example \ref{circle__} now that we have introduced some additional concepts. In that example, we can see that $\omega_f(R)=R$ (recall that $0$ is a repulsive fixed point for $g$), and in particular $\omega_f(R_k)=R_k$ for every $k\in\mathbb{N}_0$. Moreover, 
\begin{equation*} 
p(\omega\cdot k)(R)=\{(1,\theta_j):j\ge k\},
\end{equation*}
while 
$$p\omega^2(R)= \{O,(1,0)\}\cup \left(\mathcal{O}_g^{\mathbb{Z}}(1/2)\times 0\right).$$
which is therefore a uniform proper $\omega^2$-attractor. Notice that $O$, which is strictly $\omega^2$-recurrent with the points $(1/2,\theta_k)$ for every $k\in\mathbb{N}$,  belongs to $p\omega^2(R)$ and also to $[\partial\omega](R)$, but does not belong to any proper or extended transfinite attractor of order $<\omega^2$.

\begin{defi}[\textbf{Transfinite equicontinuity}]\label{def_equi_trans}
    For $\lambda\le\omega_1$, we say that $\{f\}$ is \emph{$\lambda$-equicontinuous} at $x\in X$ if, for every $\epsilon>0$, there exists $\delta=\delta(\epsilon,x)$ such that 
    \[
    d(x,y)<\delta \implies d(\{f\}^\beta(x),\{f\}^\beta(y))<\epsilon \qquad \forall\, 1\le\beta<\lambda.
    \]
\end{defi}
\begin{defi}
    For $\lambda<\omega_1$, we define the set $\mathcal{E}_\lambda\subseteq X$ of  $\lambda$-equicontinuous points of $(X,\{f\})$ as the set of all $x\in X$ such that $\{f\}$ is  $\lambda$-equicontinuous at $x$.
\end{defi}
Notice that, for $\lambda=\omega$, we have the classical definition of equicontinuity for a finite topological dynamical system (see for instance \cite{kurka2003topological}, p. 225).
Let us also recall the classical concept of \emph{sensitivity} for finite systems (see \cite[Def. 2.27]{kurka2003topological}):
\begin{defi}\label{def_sens}
    A finite dynamical system $(X,f)$ is \emph{sensitive} if 
    \[
    \exists\,\epsilon>0,\ \forall\,x\in X,\  \forall\,\delta>0, \ \exists\, y\in B_\delta(x), \ \exists\, n\ge 0,\  d(f^n(y),f^n(x))\ge \epsilon.
    \]
\end{defi}
This is the standard definition of sensitivity, but it does not work very well for transfinite systems, because in general transfinite iterations may not exist at $x$. However, a different version of it, which is equivalent in the finite case, generalizes better.
\begin{defi}\label{sensitive_2}
    We say that the finite dynamical system $(X,f)$ is \emph{sensitive}$^*$ if 
    \[
    \exists\,\epsilon>0,\ \forall\,x\in X,\  \forall\,\delta>0, \ \exists\, y_1,y_2\in B_\delta(x), \ \exists\, n\ge 0,\  d(f^n(y_1),f^n(y_2))\ge \epsilon.
    \]
\end{defi}
\begin{lem}\label{sensitiv_equiv}
Definitions \ref{def_sens} and \ref{sensitive_2} are equivalent.
\end{lem}
\begin{proof}
Since $x\in B_\delta(x)$, sensitive implies sensitive$^*$.
\\
To prove the converse implication, suppose that $(X,f)$ is sensitive$^*$  and pick $x\in X$ and $\delta>0$. 
Then there exist $y,z\in B_\delta(x)$ and $n\ge 0$ such that $d(f^n(y_1),f^n(y_2))\ge \epsilon$. Assuming that $d(f^n(y_1),f^n(x))< \epsilon/2$ and $d(f^n(x),f^n(y_2))< \epsilon/2$, it follows that 
    $$d(f^n(y_1),f^n(y_2))\le d(f^n(y_1),f^n(x))+d(f^n(x),f^n(y_2))<\frac \epsilon 2+\frac \epsilon 2,$$
    which is impossible. Then there exists a point $y\in B_\delta(x)$, coinciding with either $y_1$ or $y_2$, such that $d(f^n(x),f^n(y))\ge \epsilon/2$. By the arbitrariness of $\epsilon$, it follows that $(X,f)$ is sensitive in the sense of Definition \ref{def_sens}.
\end{proof}
We are now ready to define transfinite sensitivity by straightforwardly generalizing Def. \ref{sensitive_2}.
\begin{defi}[\textbf{Transfinite sensitivity}]\label{transsens__}
    A TDS $(X,\{f\})$ such that $\mathfrak{D}(X,\{f\})\ge\lambda$ is \emph{$\lambda$-sensitive} if 
    \[
    \exists\,\epsilon>0:\ \forall\,x\in X,\  \forall\,\delta>0,\  \exists\, \beta<\lambda, \ \exists\, y_1,y_2\in B_\delta(x)\cap X^{\beta+1}:d(\{f\}^\beta(y_1),\{f\}^\beta(y_2))\ge \epsilon.
    \]
\end{defi}
Recall that, for $n\in \mathbb{N}$, $X^{n+1}=X^\omega=X$. Therefore, 
by Lemma \ref{sensitiv_equiv}, a finite system is $\omega$-sensitive if and only if it is sensitive according to the standard Def. \ref{def_sens}. 

A well-known result in topological dynamics says that if an attractor contains only equicontinuous points, so does its basin (see for instance \cite{kurka2003topological}, Prop. 2.74). This does not hold for transfinite systems, even at the smallest limit ordinal above $\omega$, as we can see in the following example.
\begin{exa}
    Let $(\mathbb{T}=[0,1),R_a)$ be an irrational rotation on the circle and, for some space $X'$ such that $X'\cap \mathbb{T}=\emptyset$, let $(X',g)$ be an equicontinuous dynamical system.
    Set $X:=\mathbb{T} \cup X'$ and define a map $f:X\circlearrowleft$ as follows:
    \begin{equation*}
        f(x)=
        \begin{cases}
            R_a(x) \quad& \text{if } x\in \mathbb{T}\\
            g(x) \quad& \text{if } x\in X'
        \end{cases}
    \end{equation*}
    Let $x_0, y_0$ be two distinct points of $\mathbb{T}$ and set, for every $n\in\mathbb{Z}$, $x_n:=R_a^n(x_0)$ and $y_n:=R_a^n(y_0)$. Pick $x',y'\in X'$ and let us define the sequence $\{f_n\}_{n\in\mathbb{N}}$ of functions $f_n:X\circlearrowleft$ as follows:
    \begin{equation*}
        f_n(x)=
        \begin{cases}
             x' \quad& \text{if } x=x_n\\
             y' \quad& \text{if } x=y_n\\
            f(x) \quad& \text{otherwise }
        \end{cases}
    \end{equation*}
    Notice that, since the finite system $(\mathbb{T},R_\alpha)$ has no periodic points, we have $f_n \dot{\longrightarrow}  f$, and therefore $(X,\{f_n\}_{n\in\mathbb{N}})$ is a TDS (notice that the limit map $f$ is continuous). By construction if $x\in\mathbb{T}$ is such that $x\in\{x_n\}_{n\in\mathbb{Z}}$, it follows that $\{f\}^\omega(x)=x'$. Analogously, if $x\in\{y_n\}_{n\in\mathbb{Z}}$, then $\{f\}^\omega(x)=y'$.
    For all other points in $X$, since all the orbits in $\mathbb{T}$ are disjoint, we have that $\{f\}^\omega(x)=\emptyset$. 
    Set $Y:=p(\omega\cdot 2)(X)$. Then $Y\subseteq X'$, which means that $Y\subseteq\mathcal{E}_{\omega\cdot 2}$. Moreover, $x_0,y_0 \in\mathcal{B}(Y)$. However, in every open neighborhood of $x_0$ there is a point $z$ belonging to $\mathcal{O}(y_0)$, for which we have $\{f\}^\omega(z)=y'$, while $\{f\}^\omega(x_0)=x'$, so that the basin $\mathcal{B}(Y)$ does not consist of equicontinuous points. Finally, let us observe that, taking in the example $X'=\mathbb{T}$ and $g=R_\alpha$ (which we can do since an irrational rotation is an isometry, and therefore it verifies the equicontinuity assumption), $(X,\{f\})$ is an $(\omega\cdot 2)$-sensitive system. Indeed, in every nonempty open ball in $X$ there are points whose iterations of order $\omega$ are $d(x',y')$ apart.
\end{exa}

\section{Dynamical properties of transfinite attractors}\label{sec7}

 This Section is devoted to the investigation of some deeper properties of transfinite attractors. The purpose of the results developed here is to show that the classical theory occupies a small region of a much richer structural landscape (graphically represented in Figs. \ref{scheme_attrac_1} and \ref{scheme_attrac_2}), and that beyond level $\omega$ new invariance, stability, and reachability phenomena appear.
 Our first goal is to prove a transfinite analog of a well-known theorem by E. Akin,  which states that attractors are determined by their chain-recurrent points.
Let us start with a standard definition.
\begin{defi}
Given a relation $\mathcal{A}\subseteq X^2$, we say that a set $S\subseteq X$ is $\mathcal{A}$-invariant if, whenever $x\in S$ and $x\,\mathcal{A}\,y$, then $y\in S$.
\end{defi}

We can now state Akin's Theorem (see, for instance, \cite{kurka2003topological}, p. 82).
We recall that, when we say that $(X,f)$ is a topological dynamical systems we mean a finite system in which $X$ is compact, metric and $f$ is continuous. 

\begin{fact}[Akin]\label{akin__}
Let $(X,f)$ be a topological dynamical system. For $A\subseteq X$ attractor: 1. If $x\in A$ \text{ and } $x\,\mathcal{C}\,y$, then $y\in A$;
    2. The following inclusions hold:
    \begin{equation}\label{eq_akin_clas}
    A\subseteq\mathcal{N}(A\cap |\mathcal{N}|)\quad ,\quad A\supseteq\mathcal{C}(A\cap |\mathcal{C}|).
\end{equation}
\end{fact}

Here $\mathcal{C}$ is the standard chain relation (as defined in Eq. \eqref{chain_}). 
Since $\mathcal{N}\subseteq\mathcal{C}$, Eq. \eqref{eq_akin_clas} implies 
\begin{equation}\label{akino_}
  A=\mathcal{C}(A\cap |\mathcal{C}|).  
\end{equation}
This is the way in which the result is stated, for instance, in the given reference work \cite{kurka2003topological}, although the proof given there works in fact for the two inclusions in Eq. \eqref{eq_akin_clas}, which combined are in principle stronger that Eq. \eqref{akino_}.
Akin's Theorem does not transfer directly to transfinite systems, whether we generalize the concept of attractor with proper or extended $\lambda$-attractor. In fact, it fails very elementary, as it is shown in the following examples.

\begin{exa}\label{noAkin_}
In this example we see that in general the claim in point 1. of Akin's Theorem fails already at level $\omega\cdot 2$.

Take $x_0,x_1,z\in(0,1)$ such that $0<x_0<z<x_1<1$. Set $0<m<1$ and set $g_1(x):=m(x-x_0)+x_0$ and $g_2(x):=m(x-x_1)+x_1$. Let us define the map $f:\mathbb{I}\circlearrowleft$ as 
\begin{equation*}
 f(x)=\begin{cases}
        g_1(x) &\quad \text{if } x\in [0,z]\\
        x_1+\frac{x_1-g_1(z)}{x_1-z}(x-x_1) &\quad \text{if } x\in (z,x_1]\\
        g_2(x) &\quad \text{if } x\in (x_1,1]
    \end{cases}
\end{equation*}
Let $U=(a,b)$ be such that $z\in U$ and $x_0<g_1(b)<a<b<x_1$. For every $n\in\mathbb{N}$ we set $z_n:=g_1^n(z)$. Let $\{a_n\}_{n\in\mathbb{N}}$ and $\{b_n\}_{n\in\mathbb{N}}$ be two sequences of real numbers that converge to $z$ and such that 
\begin{itemize}
    \item $\{a_n\}_{n\in\mathbb{N}}$ is increasing and $a_n\in (a,z)$ for all $n\in\mathbb{N}$;
    \item $\{b_n\}_{n\in\mathbb{N}}$ is decreasing and $b_n\in (z,b)$ for all $n\in\mathbb{N}$.
\end{itemize}
For every $n\in\mathbb{N}$, we set $u_n:=g_1^n(a_n)$ and $v_n:=g_1^n(b_n)$ and we indicate by $U_n$ the interval $[u_n,v_n]$. Notice that $z_n\in U_n$ for all $n\in\mathbb{N}$. Pick $h\in (x_1,1)$. We define the sequence of functions $f_n:\mathbb{I}\circlearrowleft$ as follows (see Fig. \ref{fig_akin_c}):
\begin{equation*}
 f_n(x)=\begin{cases}
        f(x) &\quad \text{if } x\in \mathbb{I}\setminus U_n\\
        l_n(x-z_n)+h &\quad \text{if } x\in [u_n,z_n)\\
        r_n(x-z_n)+h &\quad \text{if } x\in [z_n,v_n]
    \end{cases}
\end{equation*}
where 
\[
l_n=\frac{h-g_1(u_n)}{z_n-u_n},\qquad r_n=\frac{h-g_1(v_n)}{z_n-v_n}.
\]
The maps $f_n$ are continuous for every $n\in\mathbb{N}$ and we have $f_n \dot{\longrightarrow} f$, with the limit map $f$ continuous as well. Then $(\mathbb{I},\{f\})$ is a sequentially continuous TDS.

Set $\lambda:=\omega\cdot 2$. Then $\mathbb{I}$ is a $\lambda$-inward set and $A:=p\lambda(\mathbb{I})=\{x_1\}$. Since $x_1\, \mathcal{C}\, x_0$ and $x_0\notin A$ we have that the claim 1. in Akin's Theorem does not hold for the proper $\lambda$-attractor $A$. Notice that, since $\mathcal{C}\subseteq \lambda\{\mathcal{C}\}$, this also implies that $\lambda$-attractors are not $\lambda\{\mathcal{C}\}$-invariant.

\begin{figure}[H]
\centering
\fbox{\includegraphics[width=10cm]{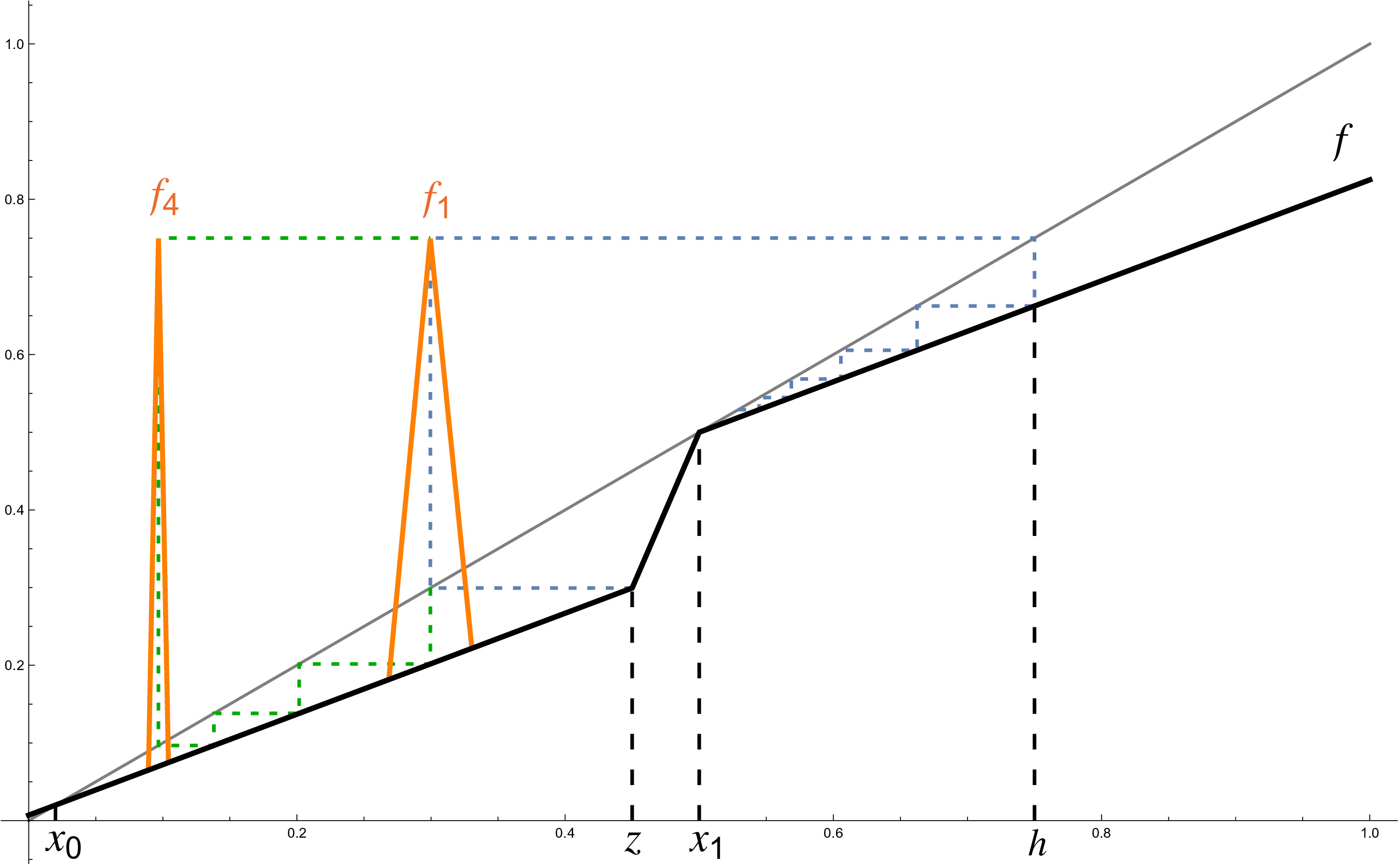}}
\caption{The maps $f_1$ and $f_4$ from the sequence $\{f_n\}_{n\in\mathbb{N}}$ defining the system in Example \ref{noAkin_}.}
\label{fig_akin_c}
\end{figure}
\end{exa}
\begin{rem}\label{subs___}
     Example \ref{noAkin_} can be used to show that TDSs are not stable up to subsequences, as stated in Remark \ref{_subs_}.
    Pick indeed $y\in (x_1,1)$.
    We define the sequence of functions $\{f_n\}_{n\in\mathbb{N}}$ by setting $h=y$ if $n$ is even, and $h=f(y)$ if $n$ is odd. 
    It follows that $f(y)\in \mathcal{O}_\infty(z)$, while $y\notin \mathcal{O}_\infty(z)$. 
    In particular, we have that $\{f\}^\omega(x)=\emptyset$ for all $x\in \mathbb{I}\setminus \mathcal{O}^\mathbb{Z}(z)$. 
    Moreover, setting $g_n:=f_{2n}$ for every $n\in\mathbb{N}$, and considering the TDS
    $(X,\{g_n\}_{n\in\mathbb{N}})$ where $g_n\dot{\longrightarrow} f$, we have that, for $x\in \mathcal{O}^\mathbb{Z}(z)$, $$f(y)=\{f\}^\omega(x)\ne\{g\}^\omega(x)=y.$$
\end{rem}

Let us now show that also the claim in point 2. of Akin's Theorem fails in general. For this, consider again Example \ref{ex_attr_cyc}, where we have $A:=e\omega^2(\mathbb{I})=p\omega^2(\mathbb{I})=\{x_0,z,f(z),f^2(z),\ldots\}$. In this case, since $\mathcal{C}(A)=\{x_0\}$, we have that claim 1. of Akin's Theorem holds for $A$. On the other hand,  $\mathcal{C}(A\cap |\mathcal{C}|)=\mathcal{C}(\{x_0\})=\{x_0\}\subsetneq A$,
which means that equality \eqref{akino_} is not verified.

\vspace{0.2cm}

Our aim is to identify the ``correct" transfinite generalization of the dynamical relations $\mathcal{C}$ and $\mathcal{N}$ to obtain a satisfactory transfinite version of Akin's Theorem. We thus want to define two dynamical relations, depending on a limit ordinal $\lambda<\omega_1$, which we will call $\lambda \{\mathcal{G}\}$ and $\lambda\{\mathcal{D}\}$, having the following characteristics:
\begin{enumerate}
  \item they are at least as strong as, respectively, $\mathcal{C}$ and $\mathcal{N}$ for $\lambda=\omega$, in the sense that: $$x\,\omega\{\mathcal{G}\}\,y\implies x\,\mathcal{C}\,y\quad , \quad x\,\omega\{\mathcal{D}\}\,y\implies x\,\mathcal{N}\,y $$
    \item they verify Akin's Theorem above, that is we have that $A$ is $\lambda\{\mathcal{G}\}$-invariant and
    \begin{equation}\label{akin___}
    A\subseteq \lambda\{\mathcal{D}\}(A\cap |\lambda\{\mathcal{D}\}|)\quad,\quad A\supseteq \lambda\{\mathcal{G}\}(A\cap |\lambda\{\mathcal{G}\}|) 
    \end{equation} 
    where $A$ is a (uniform) extended $\lambda$-attractor. 
    \end{enumerate}
    
    \begin{rem}\label{akin_ref}
    We remark that, in the particular case $\lambda=\omega$, this will provide a version of Akin's Theorem for finite systems in which $\mathcal{N}$ and $\mathcal{C}$ are replaced by stronger relations, which means that the first inclusion in Eq.\eqref{akin___}, assuming $\lambda=\omega$, will be in fact a slight refinement of the first inclusion in Eq.\eqref{eq_akin_clas}. 
    \end{rem}

This will be done in two steps. First of all we define the relations $\lambda \{\mathcal{B}\}$ and $\lambda\{\mathcal{F}\}$; the latter has the desirable property that proper $\lambda$-attractors are $\lambda\{\mathcal{F}\}$-invariant (whereas, as we saw, they are not $\lambda\{\mathcal{C}\}$-invariant). Building on those, we define the relations $\lambda\{\mathcal{D}\}$ and $\lambda\{\mathcal{G}\}$ appearing in Eq.\eqref{akin___}.
Finally, we will refine the definition of transfinite attractor, arriving gradually at a special class of proper attractors which verify Eq. \eqref{eq_akin_clas} with $\lambda\{\mathcal{B}\}$ and $\lambda\{\mathcal{F}\}$ replacing respectively $\lambda\{\mathcal{D}\}$ and $\lambda\{\mathcal{G}\}$.

\begin{defi} 
Let $\lambda$ be a countable limit ordinal. Let us define the following relations:
\begin{enumerate}
\item $(x,y)\in \lambda\mathcal{\{B\}}\iff$ for every pair of open neighborhoods $U(y)$ and $V(x)$ and for every $\beta<\lambda$, there exists $z\in V(x)$ and $\beta<\gamma<\lambda$ such that $\{f\}^\gamma(z)\in U(y)$.

\item $(x,y)\in \lambda\mathcal{\{D\}}\iff (x,y)\in \beta\mathcal{\{B\}}\text{ for some limit ordinal $\beta\le\lambda$}$.

\item $(x,y)\in \lambda\mathcal{\{F\}}\iff \forall \epsilon>0$, $\forall \beta<\lambda$, there exists an $(\epsilon,\lambda)$-chain 
\[
(\{x_0,x_1,\ldots,x_n\},\{\lambda_0,\ldots,\lambda_{n-1}\})
\]
from $x$ to $y$ such that $\beta<\lambda_{n-1}<\lambda$.

\item $(x,y)\in \lambda\mathcal{\{G\}}\iff (x,y)\in \beta\mathcal{\{F\}}\text{ for some limit ordinal $\beta\le\lambda$}$.
\end{enumerate}
\end{defi}

(In the case of ordinary dynamical systems, that is when $\lambda=\omega$, the relation $\lambda\{\mathcal{F}\}$ is a weaker form of what is called ``the relation $\mathcal{R}$" in the literature, see, e.g., \cite{ding2008chain}.)

Even if $d$ appears in their definition, the relations $\lambda\{\mathcal{F}\}$ and $\lambda\{\mathcal{G}\}$ are independent of the metric, because, like $\lambda\{\mathcal{C}\}$, they can be defined requiring that $\lambda\{\mathcal{F}\}$-chains and $\lambda\{\mathcal{G}\}$-chains belong to a certain neighborhood of the set $\{(x,x):x\in X\}$ in the product space $X^2$.
It is straightforward to see that $\lambda\{\mathcal{B}\}\subseteq \lambda\{\mathcal{N}\}$ and $\lambda\{\mathcal{F}\}\subseteq \lambda\{\mathcal{C}\}$. Let us show two further inclusions:
\begin{prop}\label{inclusions_DGBF}
Let $(X,\{f\})$ be a TDS. Then, for every countable limit ordinal $\lambda$, we have:
\begin{equation*}
    \lambda\{\mathcal{B}\}\subseteq\lambda\{\mathcal{F}\}\quad\quad,\quad\quad\lambda\{\mathcal{D}\}\subseteq\lambda\{\mathcal{G}\}.
\end{equation*}
\end{prop}
\begin{proof}
The first inclusion follows from Remark \ref{etabeta}.
Moreover, if $(x,y)\in\lambda\{\mathcal{D}\}$, then $(x,y)\in\beta\{\mathcal{B}\}$ for some $\beta\le\lambda$, so by the first inclusion  $(x,y)\in\beta\{\mathcal{F}\}$, which implies  $(x,y)\in\lambda\{\mathcal{G}\}$.
\end{proof}

Let us now prove that $\lambda\{\mathcal{F}\}$ shares some of the nice properties that $\mathcal{C}$ has for finite dynamical systems (Propositions \ref{clo_tra_} and \ref{lambdaF_inv}). Moreover, we will prove (Lemma \ref{lambdaf_prop}) that $\lambda\{\mathcal{F}\}$ also has some desirable additional properties.
\begin{prop}\label{clo_tra_}
The relation $\lambda\{\mathcal{F}\}$ is closed and transitive in $X\times X$.
\end{prop}
\begin{proof}
    The transitivity of $\lambda\{\mathcal{F}\}$ is trivial. Let $(x,y)\in \overline{\lambda\{\mathcal{F}\}}$, we want to show that $(x,y)\in \lambda\{\mathcal{F}\}$. Pick $\epsilon>0$ and $\beta<\lambda$.  
    Since $f$ is continuous in $X$, for $x\in X$ there exists $0<\delta<\epsilon/2$ such that $d(x,z)<\delta$, implies $d(f(x),f(z))<\epsilon/2$.
    By hypothesis, there exist $n\in\mathbb{N}$ and an $(\epsilon/2,\lambda)$-chain 
    $$
    (\{x_0,x_1,\ldots,x_n\}, \{\lambda_0,\ldots,\lambda_{n-1}\})
    $$ 
    from some point $x_0\in B_\delta(x)$ to some point $x_n\in B_\delta(y)$ such that $\beta<\lambda_{n-1}<\lambda$.
    
    If $\lambda_0\ge \omega$, using an analogous argument of the proof of Proposition \ref{prop_N_C} for the relation $\lambda\{\mathcal{C}\}$, we obtain that 
    $$
    (\{x,f(x_0), x_1,\ldots, x_{n-1},y\}, \{1,\lambda_0,\ldots,\lambda_{n-1}\})
    $$ 
    is an $(\epsilon,\lambda)$-chain from $x$ to $y$ with $\beta<\lambda_{n-1}<\lambda$. If $\lambda_0<\omega$ we have that $$(\{x,f(x_0),x_1,\ldots,x_{n-1},y\}, \{1,\lambda_0-1,\lambda_1,\ldots,\lambda_{n-1}\})$$ is an $(\epsilon,\lambda)$-chain from $x$ to $y$ with $\beta<\lambda_{n-1}<\lambda$.
    By the arbitrariness of $\epsilon$ and $\beta<\lambda$, we conclude that $(x,y)\in \lambda\{\mathcal{F}\}$. 
\end{proof}

\begin{lem}\label{lambdaf_prop}
    Let $V$ be a closed uniformly $\lambda$-inward set. Then $\lambda\{\mathcal{F}\}(V)= p\lambda(V)$.
\end{lem}
\begin{proof}
    Pick $x\in V$ and let us assume that $x\,\lambda\{\mathcal{F}\}\,y$. We want to show that $y\in p\lambda(V)$. 
    Since $V$ is uniformly $\lambda$-inward, there exists $\delta>0$ such that $\inf_{\beta\le \lambda} d(\{f\}^\beta(\overline{V}),\partial V)>\delta$.
    Let $\{\epsilon_j\}_{j\in\mathbb{N}}$ be a strictly decreasing sequence of real numbers that converges to $0$, and such that $\epsilon_j<\delta$ for all $j\in\mathbb{N}$. 
    Let $\{\beta_j\}_{j\in\mathbb{N}}$ be a strictly increasing sequence of ordinals such that $\sup_{j\in\mathbb{N}}\beta_j=\lambda$.
    Since $x\, \lambda\{\mathcal{F}\}\, y$, for every $j\in\mathbb{N}$, there exists an $(\epsilon_j,\lambda)$-chain 
    $$
    (\{x_0,x_1,\ldots,x_{n_j}\},\{  \lambda_0,\ldots,\lambda_{n_j-1}\})
    $$ 
    from $x$ to $y$ with $\beta_j<\lambda_{n_j-1}<\lambda$.
    Since $\epsilon_j<\delta$, we have that $x_i\in V$ for all $i=0,\ldots,n_j-1$.
    Moreover, since the sequence of ordinals $\{\beta_j\}_{j\in\mathbb{N}}$ is strictly increasing, there exists a subsequence $\{j_k\}_{k\in\mathbb{N}}$ such that $\{\lambda_{n_{j_k}-1}\}_{k\in\mathbb{N}}$ is strictly increasing. 
    
    It follows that the sequence of points $\{x_{n_{j_k}-1}\}_{k\in\mathbb{N}}\subseteq V$ and $\{\lambda_{n_{j_k}-1}\}_{k\in\mathbb{N}}$ verify the assumptions of Proposition \ref{th_plim_seq}. 
    In particular, since $\{f\}^{\lambda_{n_{j_k}-1}}(x_{n_{j_k}-1})\in B_{\epsilon_{j_k}}(y)$ for all $k\in\mathbb{N}$, it follows that $\{f\}^{\lambda_{n_{j_k}-1}}(x_{n_{j_k}-1})\longrightarrow y$ as $k\to\infty$. 
    Then $y\in p\lambda(V)$. 
    
    Conversely, take $y\in p\lambda(V)$. 
    We want to show that there exists $x\in V$ such that $x\, \lambda\{\mathcal{F}\}\, y$.
    By Proposition \ref{th_plim_seq} there exist a sequence of points $\{y_j\}_{j\in\mathbb{N}}\subseteq V$ and a strictly increasing sequence of ordinals $\{\beta_j\}_{j\in\mathbb{N}}$ such that $\sup_{j\in\mathbb{N}}\beta_j=\lambda$ and $\{f\}^{\beta_j}(y_j)\longrightarrow y$ as $j\to\infty$. 
    Then, there exists a subsequence $\{y_{j_k}\}_{k\in\mathbb{N}}$ that converges to some $x\in X$, and since $V$ is closed we have that $x\in V$. Since $f$ is continuous, for every $\epsilon>0$, there exists $0<\delta_\epsilon<\epsilon$ such that $d(f(x),f(z))<\epsilon$ whenever $z\in B_{\delta_\epsilon}(x)$. Let us consider the following two cases:
     \begin{itemize}
         \item Suppose that $\lambda=\omega$. Pick $\epsilon>0$ and $\beta<\omega$.
         Let $q\in\mathbb{N}$ be so large that $y_{j_q}\in B_{\delta_\epsilon}(x)$, $\{f\}^{\beta_{j_q}}(y_{j_q})=f^{\beta_{j_q}}(y_{j_q})\in B_\epsilon(y)$ and $\beta<\beta_{j_q}-1<\omega$. Then the pair
         $$
         (\{x, f(y_{j_q}), y\}, \{1, \beta_{j_q}-1\})
         $$ 
         is an $(\epsilon,\omega)$-chain from $x$ to $y$. By the arbitrariness of $\epsilon$ and $\beta<\omega$, we have $x\,\omega\{\mathcal{F}\}\,y$;
         
         \item Suppose that $\lambda>\omega$. Pick $\epsilon>0$ and $\beta<\lambda$. 
         Let $q\in\mathbb{N}$ be so large that $y_{j_q}\in B_{\delta_\epsilon}(x)$, $\{f\}^{\beta_{j_q}}(y_{j_q})\in B_\epsilon(y)$ and $\beta<\beta_{j_q}<\lambda$. Then the pair
         $$
         (\{x, f(y_{j_q}), y\}, \{1, \beta_{j_q}\})
         $$ 
         is an $(\epsilon,\lambda)$-chain from $x$ to $y$. Indeed, since $\beta_{j_q}\ge \omega$, we have that 
         $$
         d(\{f\}^{\beta_{j_q}}(f(y_{j_q})),y)=d(\{f\}^{\beta_{j_q}}(y_{j_q}),y)<\epsilon.
         $$
         We conclude that $x\,\lambda\{\mathcal{F}\}\,y$.
     \end{itemize}
    
\end{proof}

\begin{rem}\label{__refin__}
    Notice that in the previous lemma, the inclusion $\lambda\{\mathcal{F}\}(V)\subseteq p\lambda(V)$ is valid without assuming that $V$ is a closed set.
    Note also that, when specialized to the case $\lambda=\omega$, the lemma is a slight refinement of the well-known inclusion, holding for closed $V$,
    $$\omega(V)\subseteq \mathcal{C}(V).$$
\end{rem}

\begin{prop}\label{lambdaF_inv}
    A uniform proper $\lambda$-attractor is $\lambda\{\mathcal{F}\}$-invariant. 
\end{prop}
\begin{proof}
   Let $P=p\lambda(V)$ be a uniform proper $\lambda$-attractor. Since $P\subseteq V$, by Lemma \ref{lambdaf_prop}, it follows that $$\lambda\{\mathcal{F}\}(P)\subseteq \lambda\{\mathcal{F}\}(V)\subseteq p\lambda(V)=P.$$ 
\end{proof}
As we said, the relation $\lambda\{\mathcal{F}\}$ does not coincide with $\mathcal{C}$ when $\lambda=\omega$. In particular, one can have $\omega \{\mathcal{F}\}\subsetneq \mathcal{C}$. In a certain sense, however, $\omega \{\mathcal{F}\}$ is a more fundamental relation  than $\mathcal{C}$ for finite systems, because, just like the latter:
\begin{itemize}
    \item it is closed and transitive (by Proposition \ref{clo_tra_});
    \item  attractors are $\omega\{\mathcal{F}\}$-invariant (by Proposition \ref{lambdaF_inv}).
\end{itemize}  
Moreover it also verifies:
\begin{itemize}
    \item $\omega\{\mathcal{F}\}(V)=\omega(V)$ when $V$ is a (closed) inward set (by Lemma \ref{lambdaf_prop}). 
\end{itemize}

Of course, in the case $\lambda=\omega$, we cannot replace, in Lemma \ref{lambdaf_prop}, the relation $\omega\{\mathcal{F}\}$ by $\lambda\{\mathcal{C}\}$ because in general we have $$V\text{ is inward }\implies \mathcal{C}(V)\supsetneq  \omega(V).$$

Under some conditions, in finite systems, $x\,\mathcal{C}\,y$ implies $x\,\omega\{\mathcal{F}\}\,y$. For instance, this is true if $X$ is connected and every point in $X$ is chain recurrent with itself (this is an immediate consequence of Corollary 14 in \cite{richeson2008chain}).
In general, however, $\omega\{\mathcal{F}\}$ is stronger than $\mathcal{C}$. 
The following result concerns indeed the cases in which $(x,y)\in\mathcal{C}$ but $(x,y)\notin\omega\{\mathcal{F}\}$. It says that two points can belong to $C\setminus\omega\{\mathcal{F}\}$ only if they are in chain relation in quite a special sense: the $\epsilon$-chains connecting them can always be built without using the last $\epsilon$-correction, because with the last iterations of the map one can hit the target point. More precisely, we have the following 

\begin{prop}\label{omegaF_C}
    Let $(X,f)$ be a topological dynamical system. Assume that we have $(x,y)\in\mathcal{C}$
    and $(x,y)\notin \omega\{\mathcal{F}\}$. Then there exists a certain $\overline{k}\in\mathbb{N}$ such that, for every $\epsilon>0$, there is an $(\epsilon,\omega)$-chain 
    $$(\{x_0,\dots,x_n\},\{k_0,\ldots,k_{n-2},\overline{k}\})$$ from $x$ to $y$ such that
    \begin{equation}\label{k_barr_}
       f^{\overline{k}}(x_{n-1})=x_n=y. 
    \end{equation}
\end{prop}
\begin{proof}
In fact, we will prove a stronger result. 
Let $\{\epsilon_j\}_{j\in\mathbb{N}}$ be a strictly decreasing sequence of positive reals converging to zero. Let be $(x,y)\in \mathcal{C}$. We have as well $(x,y)\in\omega\{\mathcal{C}\}$. Notice that, if $f^m(y)=y$ for some $m\ge 1$, then clearly $(x,y)\in\omega\{\mathcal{F}\}$, so we can assume that $y$ is not periodic. Suppose that  there exists $k\in\mathbb{N}$ such that, for every $j\in \mathbb{N}$, there is an $(\epsilon_j,\omega)$-chain $$(\{x_0^j,\dots,x_n^j\},\{k^j_0,\dots,k^j_{n-1}\})$$ between $x$ and $y$ such that $k_{n-1}^j<k$. Notice that this does not exclude $x\,\omega\{\mathcal{F}\}\,y$, it is simply a necessary condition for having $(x,y)\notin\omega\{\mathcal{F}\}$.

Set $u^j:=f^{k^j_{n-1}}(x^j_{n-1})$ and $w^j:=f^{k^{j}_{n-2}}(x^{j}_{n-2})$. There is a positive integer $\overline{k}\le k$ such that $k_{n-1}^j=\overline{k}$ along a subsequence $\{j_m\}_{m\in\mathbb{N}}$. By assumption, $u^{j_m}\to y$ when $m\to\infty$. The sequence $\{x_{n-1}^{j_m}\}_{m\in\mathbb{N}}$ converges, up to a subsequence $\{j_{m_p}\}_{p\in\mathbb{N}}$, to some $x_\infty$, and since $\epsilon^j\to 0$, by triangle inequality we have as well $w^{j_{m_p}}\to x_\infty$ when $p\to\infty$. To avoid too heavy a notation, we simply indicate by $j$ the generic term of the latter subsequence.  

Pick $\epsilon>0$ and take $j$ so large that $\epsilon_j<\epsilon$ and $d(u^j,x_\infty)<\epsilon$. In these assumptions, we have that $$(\{x^j_0,\dots,x^j_{n-2}, x_{\infty},x^j_n\},\{k_0^j,\dots,k_{n-2}^j,\overline{k}\})$$ is an $(\epsilon,\omega)$-chain from $x$ to $y$ because, by continuity of the map $f^{\overline{k}}$, we have $f^{\overline{k}}(x_\infty)=y$. Since Eq. \eqref{k_barr_} is verified and $\overline{k}$ is independent of $\epsilon$, we are done.
\end{proof}

\begin{rem}\label{anomaly}
The previous result sheds some light on what might be regarded as an ``anomaly" of finite systems. Indeed, in finite systems, the attractors are not only $\omega\{\mathcal{F}\}$-invariant (just as $\lambda$-attractors are $\lambda\{\mathcal{F}\}$-invariant), but also $\mathcal{C}$-invariant. On the other hand, $\lambda$-attractors are not, in general, $\lambda\{\mathcal{C}\}$-invariant for $\lambda>\omega$. 
Notice however that, when $Y\subseteq X$ is an attractor, if $x\in Y$ and $x\,\mathcal{C}\,y$, there are, according to Proposition \ref{omegaF_C}, two possibilities:
\begin{enumerate}
    \item the $(\epsilon,\omega)$-chains are forced to reach $y$ with iterations of unbounded order (for $\epsilon\to 0)$, which means that they are in fact in $\omega\{\mathcal{F}\}$ relation;
    \item the $(\epsilon,\omega)$-chains hit the point $y$ without needing the last $\epsilon$-correction, so that in fact $y$ is in orbit relation with some point connected to $x$ by $(\epsilon,\omega)$-chains of the type above (or with $x$ itself), and therefore $y$ belongs to $Y$ because of $f$-invariance of attractors. Notice that this works because, in finite systems, an attractor $A$ is $\omega$-invariant (as of course $f^k(A)\subseteq A$ for every $k<\omega$), that is it is invariant \emph{at the same level} at which it is attractive.
\end{enumerate} 
\end{rem} 

We are now ready to prove one of the two generalizations of Akin's Theorem  which we give for transfinite systems. This version holds for uniform extended $\lambda$-attractors and needs some additional conditions. To get a version for proper attractors (Theorem \ref{akin_perf_}), which in a sense is more relevant, we will need to refine further our concept of transfinite attractor.
\begin{thm}\label{__akin_}
    Let $E:=e\lambda(V)$ be a uniform extended $\lambda$-attractor. Then $E$ is $\lambda\{\mathcal{G}\}$-invariant. Moreover, if there exists a countable ordinal $\gamma$ such that $\gamma\cdot \omega\le\lambda$ and $E\subseteq \{f\}^\gamma(E)$, we have 
    \begin{equation*}
    E\subseteq\lambda\{\mathcal{D}\}(E\cap |\lambda\{\mathcal{D}\}|)\quad ,\quad E\supseteq\lambda\{\mathcal{G}\}(E\cap |\lambda\{\mathcal{G}\}|).
    \end{equation*}
\end{thm}
\begin{proof}
     Let us prove that $\lambda\{\mathcal{G}\}(E)\subseteq E$. Let $x\in E$ be such that $x \,\lambda\{\mathcal{G}\}\, y$. This means that there exists $\beta_0\le \lambda$ limit ordinal such that $x \,\beta_0\{\mathcal{F}\}\, y$. By Lemma \ref{lambdaf_prop}, since $x\in E\subseteq V$, it follows that $y\in p\beta_0(V)\subseteq E$. Since $E$ is $\lambda\{\mathcal{G}\}$-invariant, it follows that 
    \[
    \lambda\{\mathcal{G}\}(E\cap |\lambda\{\mathcal{G}\}|)\subseteq \lambda\{\mathcal{G}\}(E)\subseteq E.
    \]
    It remains to prove the other inclusion.
    Take $y\in E$ and assume that there exists $\gamma<\omega_1$ such that $\gamma\cdot \omega\le\lambda$ and $E\subseteq \{f\}^\gamma(E)$. 
    Therefore, we construct a sequence $\{y_n\}_{n\in\mathbb{N}_0}$ of points of $E$ such that 
    \[
    y_0=y \quad \text{and}\quad \{f\}^{\gamma}(y_n)=y_{n-1}\qquad \forall n\in\mathbb{N}.
    \]
    Then, there exists a subsequence $\{y_{n_k}\}_{k\in\mathbb{N}}$ converging to some $x\in X$. Since $E$ is closed, we have that $x\in E$. 
    We want to show that $x \,\lambda\{\mathcal{D}\}\, y$ and $x\in |\lambda\{\mathcal{D}\}|$. 
    
    Pick $\epsilon>0$ and $\beta<\gamma\cdot \omega$. 
    Let $k_1,k_2\in\mathbb{N}$ be sufficiently large that $y_{n_{k_1}},y_{n_{k_2}}\in B_\epsilon(x)$ and $\gamma\cdot (n_{k_2}-n_{k_1})>\beta$. Since
    \[
    \{f\}^{\gamma\cdot (n_{k_2}-n_{k_1})}(y_{n_{k_2}})=y_{n_{k_1}}\in B_\epsilon(x),
    \]
    it follows that $x\in |(\gamma\cdot\omega)\{\mathcal{B}\}|$, which implies that $x\in |\lambda\{\mathcal{D}\}|$.

    Pick $\epsilon>0$ and $\beta<\gamma\cdot \omega$, then there exists $k_0>0$ sufficiently large that $y_{n_{k_0}}\in B_\epsilon(x)$ and $\gamma\cdot n_{k_0}>\beta$. 
    Since $\{f\}^{\gamma\cdot n_{k_0}}(y_{n_{k_0}})=y$, it follows that $x\,\lambda\{\mathcal{D}\}\, y$.
\end{proof}
 Since $\lambda\{\mathcal{D}\}\subseteq \lambda\{\mathcal{F}\}$, the following result immediately follows from Theorem \ref{__akin_}:
\begin{thm*}\label{cor_akin_}
In the same assumptions of Theorem \ref{__akin_}, we have: $$E=\lambda\{\mathcal{G}\}(E\cap |\lambda\{\mathcal{G}\}|),$$
that is an extended uniform $\lambda$-attractor is determined by its $\lambda\{\mathcal{G}\}$-recurrent points.
\qed   
\end{thm*}

Notice that, for a uniform proper $\lambda$-attractor $P$, a certain form of the previous Theorem holds. Indeed, if there exists a countable ordinal $\gamma$ such that $\gamma\cdot \omega\le\lambda$ and $P\subseteq \{f\}^\gamma(P)$, by Lemma \ref{lambdaf_prop} and an argument identical to that of the proof of Theorem \ref{__akin_}, we have that $P$ is $\lambda\{\mathcal{F}\}$-invariant, and 
    \begin{equation*}
    P\subseteq (\gamma\cdot\omega)\{\mathcal{B}\}(P\cap |(\gamma\cdot\omega)\{\mathcal{B}\}|)\quad ,\quad P\supseteq\lambda\{\mathcal{F}\}(P\cap |\lambda\{\mathcal{F}\}|).
    \end{equation*}
    This version is not very satisfactory, because a proper attractor is meaningful at its specific transfinite level, while the two inclusions above concern in general dynamical relations at two distinct ordinal levels: there is no suitable ``time-scale" at which the transfinite attractor is exactly determined by its recurrent points (whatever version of recurrence we use), so that the main feature of the classical Theorem by Akin is lost.

\begin{defi}
    For any $S\subseteq X$ and $\delta>0$ we set 
    $$\lambda\{\mathcal{C}\}_\delta(S):=\{ x\in X : \text{ there exists a $(\delta, \lambda)$-chain from a certain $y\in S$ to $x$} \}.$$
\end{defi}

\begin{defi}
    We say that a proper $\lambda$-attractor $P$ is \emph{minimal} if none of its proper subsets $Y\subsetneq P$ is a proper $\lambda$-attractor.
\end{defi}
The following Proposition generalizes to transfinite systems a classical theorem holding for finite systems. In that case, it can be simply stated as follows (see for instance \cite{kurka2003topological}, Proposition 2.69): 
\begin{fact}
Let $(X,f)$ be a topological dynamical system and $Y\subseteq X$ an attractor. Then $Y$ is chain transitive if and only if it is a minimal attractor.   
\end{fact}
In case of transfinite systems, the dynamical relation appearing in the two implications is not the same, but of course both have the standard chain relation as a particular case at level $\omega$. 
\begin{prop}
    Let $P$ be a uniform proper $\lambda$-attractor. The following hold:
    \begin{enumerate}
        \item If $P$ is $\lambda\{\mathcal{F}\}$-transitive, then $P$ is minimal;
        \item If $P$ is minimal, then $P$ is $\lambda\{\mathcal{C}\}$-transitive.
    \end{enumerate}  
\end{prop}
\begin{proof}
    \begin{enumerate}
        \item Suppose that $P$ is not minimal. Then there exists $P'\subsetneq P$ proper $\lambda$-attractor. By Proposition \ref{lambdaF_inv}, $P'$ is $\lambda\{\mathcal{F}\}$-invariant. Thus, if $y\in P'$ and $x\in P\setminus P'$, it follows that $(y,x)\notin \lambda\{\mathcal{F}\}$. Hence $P$ is not $\lambda\{\mathcal{F}\}$-transitive. 
        \item The proof of the second implication is slightly more complicated. A key point is that, for $Y\subseteq X$, a point $y$ in the boundary of $\lambda\{\mathcal{C}\}_\delta(Y)$ is arbitrarily close to points connected with $\lambda$-chains to $Y$ but, for $\beta\ge\omega$, $\{f\}^\beta(y)$ might not have points admitting $\{f\}^\beta$ iterations in its vicinity, which means that $\lambda\{\mathcal{C}\}_\delta(Y)$ is not necessarily $\lambda$-inward (notice instead that, in the finite case, $\mathcal{C}_\delta(Y)$ is always inward, because iterations of order $<\omega$ always exist). Therefore we have to take care of the boundary before proving $\lambda$-inwardness. 
        
        By uniformity of the attractor, there is a $\lambda$-inward set $V$ and $\bar{\delta}>0$ such that $P=p\lambda(V)$ and $d(\{f\}^\beta(V),\partial V)>\bar{\delta}$ for every $\beta<\lambda$.
        For any $Y\subseteq X$, assume $0<\epsilon<\delta/2$ and let us show that the set 
    $$S_\delta(Y):=\lambda\{\mathcal{C}\}_\delta(Y)\setminus \overline{B_\epsilon(\partial\lambda\{\mathcal{C}\}_\delta(Y))}$$
    is $\lambda$-inward.  
    Notice that $\lambda\{\mathcal{C}\}_\delta(Y)$ is an open set, so that $S_\delta(Y)$ is an open set too.
    Let $y\in \overline{S_\delta(Y)}$ and assume $\beta<\lambda$. If $\{f\}^\beta(y)=\emptyset$, then inwardness is vacuously verified. If $\{f\}^\beta(y)\neq \emptyset$, since 
    $$
    y\,\lambda\{\mathcal{H}\}\,\{f\}^\beta(y),
    $$ 
    we have also, by Proposition \ref{incl_top_rel} and observing that $\lambda\{{C}\}\subseteq\lambda\{{C}\}_\delta$,  
    $$
    y\,\lambda\{\mathcal{C}\}_\delta\, \{f\}^\beta(y),
    $$ 
    so that $\{f\}^\beta(y)\in \lambda\{\mathcal{C}\}_\delta(Y)$. Suppose that $\{f\}^\beta(y)\in \overline{B_\epsilon(\partial\lambda\{\mathcal{C}\}_\delta(Y))}\cap \lambda\{\mathcal{C}\}_\delta(Y)$. 
    Since $y\in \lambda\{\mathcal{C}\}_\delta(Y)$, there exist $n\in\mathbb{N}$ and a $(\delta,\lambda)$-chain 
    $$(\{y_0,\ldots,y_{n-1},y\}, \{\lambda_0,\ldots,\lambda_{n-1}\})$$ from a certain $y_0\in Y$ to $y$. 
    Notice that $d(\{f\}^\beta(y),z)\le\epsilon<\delta/2$ for some $z\in \partial\lambda\{\mathcal{C}\}_\delta(Y)$. 
    It follows that  $$(\{y_0,\ldots,y_{n-1},y,z\},\{\lambda_0,\ldots,\lambda_{n-1},\beta\})$$ is a $(\delta,\lambda)$-chain from $y_0$ to $z$. 
    Then $z\in \lambda\{\mathcal{C}\}_\delta(Y)$, which is impossible since $z\in \partial\lambda\{\mathcal{C}\}_\delta(Y)$ and $\lambda\{\mathcal{C}\}_\delta(Y)$ is an open set. It follows that $\{f\}^\beta(y)\in S_\delta(Y)=S_\delta(Y)^\circ$, so that $S_\delta(Y)$ is $\lambda$-inward.

    Now suppose that $P$ is minimal but is not $\lambda\{\mathcal{C}\}$-transitive, that is there exist $x_0,x_1\in P$ such that $(x_0,x_1)\notin \lambda\{\mathcal{C}\}$. Then there exists $\delta_0>0$ such that there is not a $(\delta_0,\lambda)$-chain from $x_0$ to $x_1$.
    We have that $P=p\lambda(V)$ for some uniformly $\lambda$-inward set $V$, and there is $\bar{\delta}>0$ such that     
    \[
    \inf_{\beta<\lambda} d\left(\{f\}^\beta(\overline{V}),\partial V\right)>\bar{\delta}.
    \]
    Therefore, we can take $0<\delta<\min\{\delta_0,\bar{\delta}\}$ to have that 
    \[
    V':=\overline{S_\delta(x_0)}\subseteq V.
    \]
    Moreover, since $S_\delta(x_0)$ is $\lambda$-inward, $V'$ is $\lambda$-inward too, and since $V'\subseteq \lambda\{\mathcal{C}\}_{\delta_0}(x_0)$ we have that $x_1\notin V'$. It follows that $P'=p\lambda(V')\subsetneq P$ is a proper $\lambda$-attractor, which is a contradiction because $P$ was minimal.
    \end{enumerate}
\end{proof}

\vspace{0.2cm}

Let us now address more closely the differences between transfinite attractors and finite attractors. Some of them can be seen as a result of the differences between $\lambda$-inwardness and finite inwardness. In particular, when $V\subseteq X$ is an inward set in the topological dynamical system $(X,f)$, the following properties hold (assuming only $f$ continuous and $X$ compact):
\begin{itemize}
\item $f^k(\overline{V})$ is a closed set for every $k\in\mathbb{N}$;
    \item $\exists\delta>0$ such that, for every $k\in\mathbb{N}$, $ d(\partial V,f^k(\overline{V}))>\delta$;
    \item $f^k(\overline{V})\subseteq f^h(\overline{V})$ whenever $k\ge h$.
\end{itemize}
The straightforward generalization of these properties, for $V$ a $\lambda$-inward set in the transfinite system $(X,\{f\})$, is:
\begin{itemize}
\item $\{f\}^\beta(\overline{V})$ is a closed set for every $\beta<\lambda$;
    \item $\exists\delta>0$ such that, for every $\beta<\lambda$, $ d(\partial V,\{f\}^\beta(\overline{V}))>\delta$;
    \item $\{f\}^\beta(\overline{V})\subseteq \{f\}^\eta(\overline{V})$ whenever $\lambda>\beta\ge \eta$.
\end{itemize}
None of these hold if we assume just that $X$ is  compact metric and $(X,\{f\})$ is $\lambda$-regular.

The objects that we introduce now are, roughly speaking, the finite attractors produced by maps $g$ of the form $g=\{f\}^{\beta}$ for some additively indecomposable ordinal $\beta<\omega_1$ (we say ``roughly speaking" because the closure in Eq. \eqref{partial__} is taken in $X$ rather than in the domain $X^{\beta+1}$ of the map $g$). Since each of them only depends on the dynamics of the iteration of order $\beta$, we call them \emph{partial attractors}. We will study some of their properties on our way towards the class of the ``truly well-behaved" transfinite attractors, i.e. perfect attractors, which we will encounter later. 
Partial attractors do behave more similarly to finite attractors with respect to the above mentioned three properties, but they fail in general to capture the whole transfinite asymptotic behavior of the system, as we will see. 

\begin{defi}[\textbf{Partial attractors}]
    Let $\beta$ be a countable additively indecomposable ordinal, that is $\beta=\omega^\eta$ for some $0\le\eta<\omega_1$. Set $\lambda:=\beta\cdot \omega=\omega^{\eta+1}$. Let $V\subseteq X$ be a closed $\lambda$-inward set.
    We say that the set
    \begin{equation}\label{partial__}
    \{\partial \beta\}(V):=\bigcap_{n=0}^\infty \overline{\bigcup_{k>n} \{f\}^{\beta\cdot k}(V)} 
    \end{equation}
    is the \emph{partial attractor} of $V$ generated by $\beta$.
    \end{defi}
    
 \begin{lem}\label{lem_part_attr}
     Let $\beta=\omega^\eta$ for some $\eta<\omega_1$ and $\lambda=\beta\cdot \omega$. If $V$ is $\lambda$-closed, then $$\{\partial\beta\}(V)=\bigcap_{k=1}^\infty \{f\}^{\beta \cdot k}(V).$$
 \end{lem}
 \begin{proof}
     Since $V$ is a closed $\lambda$-inward set, it follows from the composition rules \eqref{composition1}-\eqref{composition2} that 
     \begin{equation}\label{inclusions__}
         \{f\}^{\beta\cdot k}(V)=\{f\}^{\beta\cdot j}(\{f\}^{\beta\cdot h}(V))\subseteq \{f\}^{\beta\cdot h}(V)
     \end{equation}
     whenever $k=h+j$ for non-negative integers $h,j$. Thus $\{\partial\beta\}(V)=\bigcap_{k=1}^\infty \overline{\{f\}^{\beta \cdot k}(V)}.$
     Finally, observe that $X^{\beta\cdot k+1}\cap V$ is a closed set in the relative topology for all $k\in\mathbb{N}_0$, and since $V$ is $\lambda$-closed we have that $\{f\}^{\beta\cdot k}(V)$ is closed in $X$.
 \end{proof}

\begin{prop}\label{partial_1_}
    Let $(X,\{f\})$ be a $\lambda$-TDS with $\lambda=\beta\cdot\omega$, where $\beta=\omega^\eta$ for some $\eta<\omega_1$. Let $V$ be a closed, $\lambda$-closed, $\lambda$-inward set and let $S:=\{\partial\beta\}(V)$ be the partial attractor of $V$ generated by $\beta$.
    If $S$ is $(\beta+1)$-saturated, then $\{f\}^{\beta}(S)=S$.
\end{prop}
\begin{proof}
By Lemma \ref{lem_part_attr} and by the composition rules \eqref{composition1}-\eqref{composition2}, we have that
\begin{equation}\label{inv_S}
  \{f\}^\beta(S)=\{f\}^\beta\big(\bigcap_{k=1}^\infty \{f\}^{\beta\cdot k}(V)\big)\subseteq \bigcap_{k=1}^\infty \{f\}^{\beta\cdot k+\beta}(V)= \bigcap_{k=2}^\infty \{f\}^{\beta\cdot k}(V)=S.
\end{equation}
We want to prove the converse inclusion. 
Take $y\in\ S$ and assume that $y\notin \{f\}^\beta(S)$. 
If there exists $K\ge 0$ such that $y\notin \{f\}^\beta \big(\{f\}^{\beta\cdot K}(V)\big)= \{f\}^{\beta\cdot (K+1)}(V)$, we have a contradiction, since $y\in S$. 

Thus, suppose that there exist a sequence $\{x_n\}_{n\in\mathbb{N}}\subseteq V$ and a strictly increasing sequence $\{k_n\}_{n\in\mathbb{N}}$ of positive integers such that $\{f\}^\beta(x_n)=y$ and $x_n\in \{f\}^{\beta\cdot k_n}(V)$ for every $n\in\mathbb{N}$. 
Then there exists a subsequence $\{x_{n_j}\}_{j\in\mathbb{N}}$ such that $x_{n_j}\to x$ as $j\to\infty$. 
Notice that $S=\cap_{k\to\infty} S_k$ where $S_k=\cap_{i=1}^k \{f\}^{\beta\cdot i}(V)$. Therefore, for every $k\in\mathbb{N}$, we have that $x_{n_j}\in S_k$ for $j$ large enough, which implies that $x\in S_k$. Then $x\in S$, and since $S$ is $(\beta+1)$-saturated and the system is $\lambda$-regular, we have that $\{f\}^\beta(x)=y$.
\end{proof}
 Notice, that if $\{f\}^\beta$ is injective, then strong $\beta$-invariance of $S$ is trivial, since the inclusion in Eq. \eqref{inv_S} becomes an equality.

\begin{prop}\label{lem_preatt_proper}
    Let $\beta=\omega^\eta$ for some $\eta<\omega_1$. If $\lambda=\beta\cdot \omega$, then \begin{equation}\{\partial\beta\}\label{inclus_part_}
    (V)\subseteq p\lambda(V).
    \end{equation}
\end{prop}
\begin{proof}
    Assume $y\in \{\partial\beta\}(V)$. Then, from the fact that 
    $$\forall n,\ y\in \overline{\bigcup_{k>n} \{f\}^{\beta\cdot k}(V)},$$ 
    it follows that there exist a sequence of  points $\{x_n\}_{n\in\mathbb{N}}\subseteq V$ and an increasing sequence of positive integers $\{k_n\}_{n\in\mathbb{N}}$ such that $\{f\}^{\beta\cdot k_n}(x_n)\to y$. 
    Since $\{\beta\cdot k_n\}_{n\in\mathbb{N}}$ is a strictly increasing sequence of ordinals and $\sup_{n\in\mathbb{N}}\{\beta\cdot k_n\}=\beta\cdot \omega$, by Proposition \ref{th_plim_seq}, it follows that $y\in p\lambda(V)$.
\end{proof}

In general, the inclusion \eqref{inclus_part_} can be strict. A simple case of this kind is Example \ref{ex_quad}, where we can see that, for instance, $(1,1/2)\in p\omega^2(X)$ and at the same time $(1,1/2)\notin \{\partial\omega\}(X)$. Indeed, 
    $$\{\partial\omega\}(X)=\bigcap_{n=0}^\infty \overline{\bigcup_{k>n} \{f\}^{\omega\cdot k}(X)}=\bigcap_{n=0}^\infty \overline{\bigcup_{k>n} \{(x_k,1)\}}=\{(1,1)\}.$$

\begin{prop}\label{extended_preacc}
    Let $\lambda$ be a countable limit ordinal. Set 
    $
    \Gamma=\{\omega^\eta: \eta<\omega_1\text{ and } \omega^{\eta+1}\le\lambda\}.
    $ 
    We have  
    \begin{equation}\label{inclus_part___}
    \bigcup_{\beta\in\Gamma}\{\partial\beta\}(V)\subseteq e\lambda(V).
    \end{equation}
\end{prop}
\begin{proof}
    By Lemma \ref{lem_preatt_proper}, it follows that 
    \[  \bigcup_{\beta\in\Gamma}\{\partial\beta\}(V)\subseteq \bigcup_{\beta\in\Gamma} p(\beta\cdot \omega)(V)\subseteq e\lambda(V).
    \]
\end{proof}
Propositions \ref{partial_1_} and \ref{lem_preatt_proper} hold only when $\lambda$ has a special form, so that it can be limit of a sequence of type $\{\omega^\eta\cdot k\}_{k\in\mathbb{N}}$. If $\lambda$ is the limit of a sequence of a more general type, as for instance $\lambda=\sup_{k\in\mathbb{N}}(\omega^\eta)^k$, then one cannot use the composition rules \eqref{composition1}-\eqref{composition2} to make sure that further iterations in the sequence are nested subsets of $V$ like done in Eq. \eqref{inclusions__}.
Moreover, in general both inclusions \eqref{inclus_part_} and \eqref{inclus_part___} are strict. 
This tells us that partial attractors, although rather well-behaved, capture too little of the transfinite dynamics to be considered satisfactory from the point of view of a general theory of transfinite attractors. 
To proceed in this direction, we need to make some further assumptions in addition to uniformity, but not \emph{as strong} as those leading to partial attractors. 

\begin{defi}\label{complet__}
    We say that $V\subseteq X$ is \emph{completely $\lambda$-inward} if $\{f\}(\overline{V})\subseteq V^\circ$ and for every $\beta<\lambda$ we have
    \begin{equation}\label{inclusion_comp}
    \{f\}^\beta(\overline{V})\subseteq\bigcap_{\eta<\beta}\{f\}^\eta(\overline{V}).
    \end{equation}
\end{defi}

It is easily seen that, for $\lambda<\omega_1$, every completely $\lambda$-inward set is uniformly $\lambda$-inward and therefore $\lambda$-inward, and that every $\omega$-inward set in $(X,\{f\})$ is completely $\omega$-inward and inward in the usual sense for the finite system $(X,f)$.

\begin{rem}\label{rem_Vclosed}
    Let $V$ be a uniformly $\lambda$-inward set for some limit ordinal $\lambda<\omega_1$. Since there is $\delta>0$ such that $d(\cup_{\beta<\lambda}\{f\}^\beta(\overline{V}),\partial V)>\delta$, there is a closed set $C$ such that $\{f\}^\beta(\overline{V})\subsetneq C \subsetneq V$ (for instance the closed ball $\overline{B}_{\delta/2}(\cup_{\beta<\lambda}\{f\}^\beta(\overline{V}))$).
    Moreover, $p\lambda(V)=p\lambda(C)$ and $e\lambda(V)=e\lambda(C)$. 
\end{rem}

By Remark \ref{rem_Vclosed}, every uniform $\lambda$-attractor is the $\lambda$-limit of a \emph{closed} uniformly $\lambda$-inward set, and therefore, in the following we will assume that every uniformly $\lambda$-inward set is closed. 

\begin{defi}
   We say that a proper $\lambda$-attractor $Y$ is \emph{complete} if $Y$ is the proper $\lambda$-limit of a $\lambda$-closed, completely $\lambda$-inward set.

   We say that a proper $\lambda$-attractor $Y$ is \emph{algebraically closed} if the countable ordinal $\lambda$ is additively indecomposable, that is $\lambda=\omega^\beta$ for some $\beta<\omega_1$.

   We say that $Y$ is \emph{almost perfect} if it is both complete and algebraically closed.
\end{defi}

Note that the definition of a complete attractor $p\lambda(V)$ requires not only complete $\lambda$-inwardness of $V$ but also its $\lambda$-closedness, which is a priori not guaranteed.  

\begin{defi}[\textbf{Perfect attractors}]\label{perf_attr__}
    We say that a proper $\lambda$-attractor $Y$ is \emph{perfect} if it is both almost perfect and strongly $\lambda$-invariant. 
\end{defi}

\begin{lem}\label{caract_com_attr}
    Let $(X,\{f\})$ be a $\lambda$-TDS and $P$ a complete $\lambda$-attractor. Then, 
    $P=\bigcap_{\eta<\lambda}\{f\}^\eta(V)$.
\end{lem}
\begin{proof}
    Since $V$ is a $\lambda$-closed, completely $\lambda$-inward set, it follows that 
    $$P=\bigcap_{\eta<\lambda}\overline{\bigcup_{\eta<\beta<\lambda}\{f\}^\beta(V)}= \bigcap_{\eta<\lambda}\overline{\{f\}^\eta(V)}=\bigcap_{\eta<\lambda}\{f\}^\eta(V),$$ where the last equality follows by the fact that $V$ is a closed and $\lambda$-closed set.
\end{proof}

\begin{prop}\label{inv_perfattr}
Let $(X,\{f\})$ be a $\lambda$-TDS and $Y$ a complete $\lambda$-attractor with $$\lambda=\omega^{\alpha_1}+\ldots + \omega^{\alpha_k},$$ where $k\in\mathbb{N}$ and $\alpha_1\ge \ldots\ge \alpha_k$. Then $Y$ is $\omega^{\alpha_k}$-invariant.
\end{prop}

\begin{proof}
Assume $\beta<\omega^{\alpha_k}$. Notice that, if $\eta<\lambda$, we have that $\eta+\beta<\lambda$. By Lemma \ref{caract_com_attr} and by the composition law \eqref{composition1}-\eqref{composition2}, it follows that 
\[
\{f\}^\beta(Y)=\{f\}^\beta\big( \bigcap_{\eta<\lambda}\{f\}^{\eta}(V) \big)\subseteq \bigcap_{\eta<\lambda}\{f\}^{\eta+\beta}(V) \subseteq \bigcap_{\eta<\lambda}\{f\}^{\eta}(V)=Y, 
\]
where the last inclusion follows by the fact that $\{f\}^{\eta+\beta}(V)\subseteq \{f\}^\eta(V)$ for every $\eta<\lambda$.
\end{proof}
Our next goal is to prove a sufficient condition for perfectness of transfinite attractors. To achieve this, let us give first the following
\begin{defi}\label{reachable_}
    We say that $S\subseteq X$ is \emph{$\lambda$-reachable} for some $\lambda<\omega_1$ limit ordinal, if whenever $x\in S\cap \overline{X^{\beta+1}}$ we have $x\in X^{\beta+1}$.
\end{defi}
\begin{prop}\label{cor_lambda_inv}
   Let $(X,\{f\})$ be a $\lambda$-TDS and $Y$ a complete $\lambda$-attractor with $$\lambda=\omega^{\alpha_1}+\ldots + \omega^{\alpha_k},$$ where $k\in\mathbb{N}$ and $\alpha_1\ge \ldots\ge \alpha_k$. If $Y$ is $\omega^{\alpha_k}$-reachable, then $Y$ is strongly $\omega^{\alpha_k}$-invariant.
\end{prop}
\begin{proof}
    Suppose that $Y$ is $\omega^{\alpha_k}$-reachable and assume $\beta<\omega^{\alpha_k}$. 
    By Proposition \ref{inv_perfattr} we have that $\{f\}^\beta(Y)\subseteq Y$. 
    Let $y\in Y$ and assume that $y\notin \{f\}^\beta(Y)$. 
    By Lemma \ref{caract_com_attr}, this means that $y\notin \{f\}^\beta\big( \bigcap_{\eta<\lambda}\{f\}^{\eta}(V) \big)$. 
    
    Suppose that $S:=\{\eta<\lambda : y\notin \{f\}^\beta\big(\{f\}^\eta(V)\big)\}$ is nonempty and let $\gamma$ be the minimum of $S$. 
    Then,
    \[
    y\notin \{f\}^\beta\big(\{f\}^\eta(V)\big)=\{f\}^{\eta+\beta}(V)\quad \text{ for all $\gamma\le\eta<\lambda$}.
    \]
    Since $\eta+\beta< \lambda$ and $y\in  \bigcap_{\eta<\lambda}\{f\}^{\eta}(V)$, we have a contradiction.
    
    Suppose now that $S=\emptyset$. Therefore, there exist a sequence $\{x_n\}_{n\in\mathbb{N}}\subseteq V$ and a sequence $\{\beta_n\}_{n\in\mathbb{N}}$ of ordinals such that $\sup_{n\in\mathbb{N}}\beta_n =\lambda$, and 
    \[
    \{f\}^\beta(x_n)=y, \qquad x_n\in \{f\}^{\beta_n}(V) \quad\text{for all } n\in\mathbb{N}.
    \]
    Then, there exists a subsequence $\{x_{n_j}\}_{j\in\mathbb{N}}$ such that $x_{n_j}\to x$ as $j\to\infty$. 
    By Proposition \ref{th_plim_seq} it follows that $x\in Y$.  
    Since $Y$ is $\omega^{\alpha_k}$-reachable and $x\in Y\cap \overline{X^{\beta+1}}$, we have that $x$ belongs to $X^{\beta+1}$, and thus 
    $\{f\}^\beta(x)\neq \emptyset$.
    Since the system is $\lambda$-regular, it follows that $\{f\}^\beta(x)=y$. 
\end{proof}

\begin{thm}\label{strong_inv_perf_}
     Let $(X,\{f\})$ be a $\lambda$-TDS and $Y$ an almost perfect $\lambda$-attractor. If $Y$ is $\lambda$-reachable, then $Y$ is perfect.
\end{thm}
\begin{proof}
It follows from Proposition \ref{cor_lambda_inv} assuming $\lambda=\omega^\beta$ for some ordinal $\beta<\omega_1$.
\end{proof}

\begin{defi}
    Let $\lambda$ be a limit ordinal. We say a set $S\subseteq X$ is \emph{$\lambda$-stable}, if for every $\epsilon>0$ there exists $\delta>0$ such that, if $x\in B_\delta(S)$, then $d(\{f\}^{\beta}(x),S)<\epsilon$ for all $\beta<\lambda$.
\end{defi}
Of course if $\lambda=\omega$ we have the classical definition of a stable set (see for instance \cite[Def. 1.18]{kurka2003topological}).

\begin{prop}
    Let $(X,\{f\})$ be a $\lambda$-TDS and $Y$ be a complete $\lambda$-attractor. Then the following hold:
    \begin{enumerate}
     \item $Y$ is $\lambda$-attractive.
        \item Write $\lambda$ as $$\lambda=\omega^{\alpha_1}+\ldots + \omega^{\alpha_k},$$ where $k\in\mathbb{N}$ and $\alpha_1\ge \ldots\ge \alpha_k$; then $Y$ is  $\omega^{\alpha_k}$-stable.
          In particular, if $Y$ is algebraically closed then $Y$ is $\lambda$-stable.
    \end{enumerate} 
  
\end{prop}

\begin{proof}
   Item 1. follows by Theorem \ref{themattractors}. Let us prove item 2. by considering the following two cases:
   \begin{itemize}
       \item Suppose that there exists a sequence $\{\eta_i\}_{i\in\mathbb{N}}$ of limit ordinals such that 
   $\sup_{i\in\mathbb{N}}\eta_i=\lambda$. 
   By Lemma \ref{caract_com_attr}, since $Y=\bigcap_{\eta<\lambda}\{f\}^\eta(V)$ and $V$ is completely $\lambda$-inward, it follows that $Y=\bigcap_{i\in\mathbb{N}} \{f\}^{\eta_i}(V)$. 
   If there exists $m\in\mathbb{N}$ such that $Y=\bigcap_{i=1}^m\{f\}^{\eta_i}(V)=\{f\}^{\eta_m}(V)$, then $Y=p\eta_{m+1}(V)$, which is impossible, since $Y$ is a proper $\lambda$-attractor and $\eta_{m+1}<\lambda$. 
   It follows that there exists a strictly decreasing sequence $\{r_i\}_{i\in\mathbb{N}}$ of positive real numbers such that 
   $$
   Y\subsetneq\{f\}^{\eta_i}(V)\subseteq B_{r_i}(Y)
   $$ 
   for every $i\in\mathbb{N}$, and $r_i\to 0$ as $i\to\infty$. 
   Pick $\epsilon>0$. 
   Then, there exists $i_\epsilon\in\mathbb{N}$ such that $r_{i_\epsilon}<\epsilon$ and $Y\subsetneq \{f\}^{\eta_{i_\epsilon}}(V)\subseteq B_{r_{i_\epsilon}}(Y)$. Hence, there exists $0<\delta_\epsilon<\epsilon$ such that $B_{\delta_\epsilon}(Y)\subseteq \{f\}^{\eta_{i_\epsilon}}(V)$. In particular, for every $x\in B_{\delta_\epsilon}(Y)$, there exists $z\in V$ such that $\{f\}^{\eta_{i_\epsilon}}(z)=x$. Then, for every $\beta<\omega^{\alpha_k}$, we have that 
   $$
   \{f\}^\beta(x)\in \{f\}^{\eta_{i_\epsilon}+\beta}(V)\subseteq \{f\}^{\eta_{i_\epsilon}}(V)\subseteq B_{r_{i_\epsilon}}(Y)\subseteq B_\epsilon(Y),
   $$ 
   where the first inclusion follows by the fact that $\eta_{i_\epsilon}+\beta<\lambda$ and $V$ is completely $\lambda$-inward. By the arbitrariness of $\epsilon$ we have the claim.

   \item Otherwise, there there exists $0\le\beta<\lambda$ such that $\lambda=\beta+\omega=\sup_{k\in\mathbb{N}} (\beta+k)$. This means that $\alpha_k=1$, thus we have to prove that $Y$ is $\omega$-stable. 
   By Lemma \ref{caract_com_attr} and since $V$ is completely $\lambda$-inward, we have that $Y=\bigcap_{k\in\mathbb{N}}\{f\}^{\beta+k}(V)$. 
   Pick $\epsilon>0$, then there exists 
   $$k_\epsilon=\min\{ k\in\mathbb{N}:\{f\}^{\beta+k}(V)\subseteq B_\epsilon(Y)\}.$$
   If $Y\subsetneq \{f\}^{\beta+k_\epsilon}(V)$, then there exists $\delta>0$ such that $B_\delta(Y)\subseteq \{f\}^{\beta+k_\epsilon}(V)$. Hence, since $V$ is completely $\lambda$-inward, for every $x\in B_\delta(Y)$ we have that 
   $$f^n(x)\in \{f\}^{\beta+k_\epsilon+n}(V)\subseteq \{f\}^{\beta+k_\epsilon}(V)\subseteq B_\epsilon(Y) \quad \forall n\in\mathbb{N}.$$
   If $Y=\{f\}^{\beta+k_\epsilon}(V)$, then $B_\epsilon(Y)\subseteq \{f\}^{\beta+k_\epsilon-1}(V)$. This implies that for every $x\in B_\epsilon(Y)$, we have that $f^n(x)\in \{f\}^{\beta+k_\epsilon+n-1}(V)\subseteq \{f\}^{\beta+k_\epsilon}(V)=Y$ for all $n\in\mathbb{N}$.
   \end{itemize}
   
\end{proof}

\begin{prop}\label{akin_compl_attr}
    Let $(X,\{f\})$ be a $\lambda$-TDS and $Y$ a complete $\lambda$-attractor. Then $Y$ is $\lambda\{\mathcal{F}\}$-invariant. Moreover, if writing $\lambda$ as $$\lambda=\omega^{\alpha_1}+\ldots + \omega^{\alpha_k},$$ where $k\in\mathbb{N}$ and $\alpha_1\ge \ldots\ge \alpha_k$ are countable ordinals, $Y$ is strongly $\omega^{\alpha_k}$-invariant, then the following holds 
    \begin{equation}\label{eq_akin_compl_attr}
    Y\subseteq \omega^{\alpha_k}\{\mathcal{B}\}(Y\cap |\omega^{\alpha_k}\{\mathcal{B}\}|)\quad ,\quad Y\supseteq \lambda\{\mathcal{F}\}(Y\cap |\lambda\{\mathcal{F}\}|).
    \end{equation}
\end{prop}
\begin{proof}
Since $Y$ is a proper $\lambda$-attractor, by Proposition \ref{lambdaF_inv} we have that $Y$ is $\lambda\{\mathcal{F}\}$-invariant. 

Take $y\in Y$ and let $\{\beta_n\}_{n\in\mathbb{N}}$ be such that $\sup_{n\in\mathbb{N}}\beta_n=\omega^{\alpha_k}$. 
Suppose that $Y$ is strongly $\omega^{\alpha_k}$-invariant. Therefore, $\{f\}^{\beta_n}(Y)=Y$ for all $n\in\mathbb{N}$. 
Thus, we can construct a sequence $\{y_n\}_{n\in\mathbb{N}_0}$ of points of $Y$ such that 
    \[
    y_0=y \quad \text{and}\quad \{f\}^{\beta_n}(y_n)=y_{n-1}\qquad \forall n\in\mathbb{N}.
    \]
    Then, there exists a subsequence $\{y_{n_k}\}_{k\in\mathbb{N}}$ that converges to some point  $x$. 
    Since $Y$ is closed, we have that $x\in Y$. 
    We want to show that $x\, \omega^{\alpha_k}\{\mathcal{B}\}\, y$ and $x\in |\omega^{\alpha_k}\{\mathcal{B}\}|$. 
    
    Pick $\epsilon>0$ and $\beta<\omega^{\alpha_k}$.
    Let $p,q\in\mathbb{N}$ be sufficiently large that $y_{n_p},y_{n_q}\in B_\epsilon(x)$ and $\beta_{n_q}>\beta$. 
    Set $\gamma:=\sum_{i=0}^{n_q-n_p-1} \beta_{n_q-i}$. We have $\{f\}^\gamma(y_{n_q})=y_{n_p}\in B_\epsilon(x)$ and $\gamma\ge \beta_{n_q}$.
    Since $\omega^{\alpha_k}$ is additively indecomposable, we have $\gamma<\omega^{\alpha_k}$.
    By the arbitrariness of $\epsilon$ and $\beta<\omega^{\alpha_k}$, it follows that $x\in |\omega^{\alpha_k}\{\mathcal{B}\}|$.
    
    Pick $\epsilon>0$ and $\beta<\lambda$, then there exists $q>0$ so large that $y_{n_q}\in B_\epsilon(x)$ and $\gamma:= \sum_{i=0}^{n_q-1}\beta_{n_q-i}$ is such that $\beta<\gamma<\omega^{\alpha_k}$. 
    Since 
    $
    \{f\}^{\gamma}(y_{n_q})=y,
    $
    we conclude that $x\, \omega^{\alpha_k}\{\mathcal{B}\}\, y$.
\end{proof}

\begin{rem}\label{corol_perfect_}
  Notice that, by Proposition \ref{cor_lambda_inv}, Proposition \ref{akin_compl_attr} holds if we replace the assumption that $Y$ is strongly $\omega^{\alpha_k}$-invariant by the assumption that $Y$ is $\omega^{\alpha_k}$-reachable.
\end{rem}

\begin{rem}
    Notice that the left inclusion in \eqref{eq_akin_compl_attr} also holds with $\lambda\{\mathcal{N}\}$ instead of $\omega^{\alpha_k}\{\mathcal{B}\}$. Indeed,
    $$\omega^{\alpha_k}\{\mathcal{B}\}(Y\cap |\omega^{\alpha_k}\{\mathcal{B}\}|)\subseteq \lambda\{\mathcal{N}\}(Y\cap |\lambda\{\mathcal{N}\}|).$$
\end{rem}
\begin{thm}
\label{akin_perf_}
  Let $(X,\{f\})$ be a $\lambda$-TDS and $Y$ be a perfect $\lambda$-attractor. Then $Y$ is $\lambda\{\mathcal{F}\}$-invariant. Moreover, we have
  \begin{equation}\label{akin_tran}
    Y\subseteq \lambda\{\mathcal{B}\}(Y\cap |\lambda\{\mathcal{B}\}|)\quad ,\quad Y\supseteq \lambda\{\mathcal{F}\}(Y\cap |\lambda\{\mathcal{F}\}|).
    \end{equation}
\end{thm}
\begin{proof}
It follows from Proposition \ref{akin_compl_attr} assuming $\lambda=\omega^\beta$ for some ordinal $\beta<\omega_1$.
\end{proof}

Let us state explicitly the particular case $\lambda=\omega$, in which, similarly to what happens with Theorem \ref{__akin_}, the first inclusion is a refinement of the corresponding inclusion in Akin's Theorem:
\begin{cor}\label{akin_ref1}  Let $(X,f)$ be a topological dynamical system, with $X$ compact metric and $f$ continuous. Let $Y\subseteq X$ be an attractor. Then $Y$ is $\omega\{\mathcal{F}\}$-invariant. Moreover, we have
  \begin{equation}\label{akin_tran_omega}
    Y\subseteq \omega\{\mathcal{B}\}(Y\cap |\omega\{\mathcal{B}\}|)\quad ,\quad Y\supseteq \omega\{\mathcal{F}\}(Y\cap |\omega\{\mathcal{F}\}|).
    \end{equation}
    \qed
\end{cor}

As we saw, perfect attractors behave almost like $\omega$-attractors (i.e., attractors in finite dynamical systems), but we have to emphasize that they are \emph{very special cases} in transfinite systems. 
Indeed, to have a perfect attractor, the transfinite dynamics of $\{f\}$ has to ensure that three a priori independent conditions are verified: all the iterations of a certain inward set verify the inclusion property characterizing complete inwardness, the ordinal at which the attractor emerges has to have a certain form, and we also need strong transfinite invariance. 
This means that, unless we make some rather strong assumptions, transfinite attractors are significantly ``wilder" than finite ones. 
None of the transfinite attractors we saw up to this point (except, of course, $\omega$-attractors) were perfect. In the next example we describe a TDS in which a perfect $\lambda$-attractor emerges at level $\lambda=\omega^2$.

\begin{exa}\label{exa_perfect}
Let $\{x_i\}_{i\in\mathbb{N}}\subseteq \mathbb{I}$ be a strictly increasing sequence such that $x_1=0$ and $x_i\longrightarrow 1$ as $i\to\infty$. Set 
    $$I_i:=\{x_i\}\times \mathbb{I}\quad \forall\, i\in\mathbb{N},\qquad I_\infty:=\{1\}\times\mathbb{I},$$
    and
    $$
    X:=I_\infty \cup \bigcup_{ i\in\mathbb{N}}I_i    \subseteq \mathbb{I}^2.
    $$
    Assume on $X$ the Euclidean metric inherited from $\mathbb{R}^2$, so that $X$ is a compact metric space.
    Take $m,z\in (0,1)$. Let $g:\mathbb{I}\circlearrowleft$ be the function given by 
    \[
    g(x)=
    \begin{cases}
       m x \quad& \text{if } x\in [0,z]\\
       \frac{1-mz}{1-z}(x-1)+1 \quad& \text{if } x\in (z,1]
    \end{cases}
    \]
    Let us define the map $f:X\circlearrowleft$ as 
    $$f(x,y)=(x,g(y)).$$ 
    Let $U_0:=[a_0,b_0]\subseteq (0,z)$ and $V_0:=[c_0,d_0]\subseteq (0,z)$ be such that $g^k(U_0)\cap V_0=\emptyset$ for every $k\in\mathbb{Z}$. 
    For every $k\in\mathbb{Z}$, we set $U_k:=[a_k,b_k]=g^k(U_0)$, $V_k:=[c_k,d_k]=g^k(V_0)$ and we denote by
    \[
    \mathcal{U}=\bigcup_{k\in\mathbb{Z}}U_k, \qquad  \mathcal{V}=\bigcup_{k\in\mathbb{Z}}V_k.
    \] 
    Let $\{u_n\}_{n\in\mathbb{N}_0}$ and $\{v_n\}_{n\in\mathbb{N}_0}$ be two sequences of functions $u_n:U_n\rightarrow\mathbb{I}$ and $v_n:V_n\rightarrow\mathbb{I}$ given by
    \[
    u_n(x)=\frac{x-a_n}{b_n-a_n}, \qquad  
    v_n(x)=\frac{x-c_n}{d_n-c_n}.
    \]
    Let us now define the sequence $\{f_n\}_{n\in\mathbb{N}}$ of functions $f_n:X\circlearrowleft$ as follows (see Fig. \ref{fig_perf_att}):
    \[
    f_n(x,y)=
    \begin{cases}
       f(x,y) \quad& \text{} y\in \mathbb{I}\setminus (U_n\cup V_n) \\
       
       f\left( x_{i+1}, u_n(y)\right) \quad& \text{if } y\in U_n \text{ and } x=x_i \text{ for some }i\in\mathbb{N}\\

        f\left( 1, u_n(y)\right) \quad& \text{if } y\in U_n \text{ and } x=1 \\
       
       f\left(1, v_n(y)\right) \quad& \text{if } y\in V_n 
    \end{cases}
    \]
    Since the endpoints (and thus the diameters) of $U_n$ and $V_n$ go to zero as $n$ goes to $\infty$, we have that $f_n \dot{\longrightarrow}  f$, with the limit map $f$ continuous, and therefore we have defined the TDS $(X,\{f\})$. 

    Since the map $g$ is surjective, it follows that $f^k(X)=X$ for all $k\in\mathbb{N}_0$. 
    Moreover, observing that 1 is a repulsive fixed point for $g$ and $(0,1]$ is the basin of attraction of 1 for $g^{-1}$, we have $\omega_g(\mathbb{I})=\mathbb{I}$, which implies that $\omega_f(X)=X$. 
    
    Notice that, if $y\in \mathbb{I}\setminus (\mathcal{U}\cup\mathcal{V})$, we have that $\{f\}^\omega(x,y)=\emptyset$. 
    On the other hand, if $(x,y)\in \{x_i\}\times U_k$ or $(x,y)\in \{1\}\times U_k$ for some $i\in\mathbb{N}$ and $k\in\mathbb{Z}$, then $\{f\}^\omega(x,y)\neq \emptyset$. 
    Indeed, there exists $y_0\in U_0$ such that $g^k(y_0)=y$, and if $x=x_i$ for some $i\in\mathbb{N}$, we have:
    $$
        \mathcal{O}_n(x,y)=\left \{f(x,y),\ldots,f^{n-k-1}(x,y),\left(x_{i+1},u_n(y)\right),\ldots \right \} \quad \forall\,n>k.
    $$
    Since $u_n(y)=\frac{y_0-a}{b-a}=u_0(y_0)$ and it does not depend by $n$, it follows that $\{f\}^\omega(x,y)=\left(x_{i+1},u_0(y_0)\right)$. 
    Since the map $u_0$ is surjective, we have:
    \begin{equation}\label{it_om_U_1}
    \{ \{f\}^\omega(x_i,y) : y\in \mathcal{U}\}=I_{i+1}\qquad \forall i\in\mathbb{N}.
    \end{equation}
    Moreover, if $x=1$, with the same argument, we have: 
     \[
    \mathcal{O}_n(x,y)=\left \{f(1,y),\ldots,f^{n-k-1}(1,y),\left(1,u_n(y)\right),\ldots \right \} \quad \forall\,n>k,
    \]
    which means that
    \begin{equation}\label{it_om_U_2}
        \{ \{f\}^\omega(1,y) : y\in \mathcal{U}\}=I_\infty.
    \end{equation}
    With an analogous argument, we have:
    \begin{equation}\label{it_om_V}
        \{ \{f\}^\omega(x,y) : y\in \mathcal{V}\}=I_\infty \qquad \forall\, x\in \{1,x_0,x_1,\ldots\}.
    \end{equation}
    By \eqref{it_om_U_1}, \eqref{it_om_U_2}, and \eqref{it_om_V}, it follows that $\{f\}^\omega(X)=X\setminus I_1$ and since the map $u_0$ is continuous we have that $\{f\}^\omega$ is continuous as well. In addition, since $X$ is strongly $f$-invariant, we have $\{f\}^{\omega+k}(X)=\{f\}^\omega(X)$ for every $k\in\mathbb{N}$. 
    This implies that $(\omega\cdot 2)(X)=X\setminus I_1$. 
    With the same argument, it follows that 
    \begin{equation}\label{eq_ite_ex_attper}
         \{f\}^{\omega\cdot h+k}(X)=X\setminus (I_1\cup\ldots\cup I_h)= (\omega \cdot(h+1))_f(X)\qquad
    \forall \,h\in\mathbb{N},\ \forall\,k\in\mathbb{N}_0,
    \end{equation}
    and that $\{f\}$ is $\omega^2$-regular. 
    Moreover, Eq. \eqref{eq_ite_ex_attper} implies that 
    $$
    p\omega^2(X)=\bigcap_{\eta<\omega^2}\,\,\overline{\bigcup_{\eta<\beta<\omega^2}\{f\}^\beta(X)}=\bigcap_{h\in\mathbb{N}}\overline{X\setminus (I_1\cup\ldots\cup I_h)}=I_\infty.
    $$ 
    is a proper $\omega^2$-attractor.
    
     
    By construction, $X$ is a completely $\omega^2$-inward set and, by Eq. \eqref{eq_ite_ex_attper}, $X$ is $\omega^2$-closed, which means that $I_\infty$ is an almost perfect $\omega^2$-attractor. Moreover, since $I_\infty$ is strongly $f$-invariant, by \ref{it_om_V} it follows that $I_\infty$ is strongly $\omega^2$-invariant. Then $I_\infty$ is a perfect $\omega^2$-attractor.
    Notice that $I_\infty$ is not $\omega^2$-reachable. For instance, the point $(1,1)$ does not belong to $X^{\omega+1}$, although it is the limit of the sequence of points $\{(x_i,b_{-i})\}_{i\in \mathbb{N}}$, all admitting iterations of order $\omega$. Therefore, the example shows that the converse of Theorem \ref{strong_inv_perf_} is not true, in the sense that $\lambda$-reachability is not a necessary condition for perfectness.

\vspace{0.2cm}

    A schematic classification of transfinite attractors is given in Figs. \ref{scheme_attrac_1}-\ref{scheme_attrac_2}.

    \vspace{0.5cm}

\begin{figure}[H]
 \centering
 \fbox{\includegraphics[width=10cm]{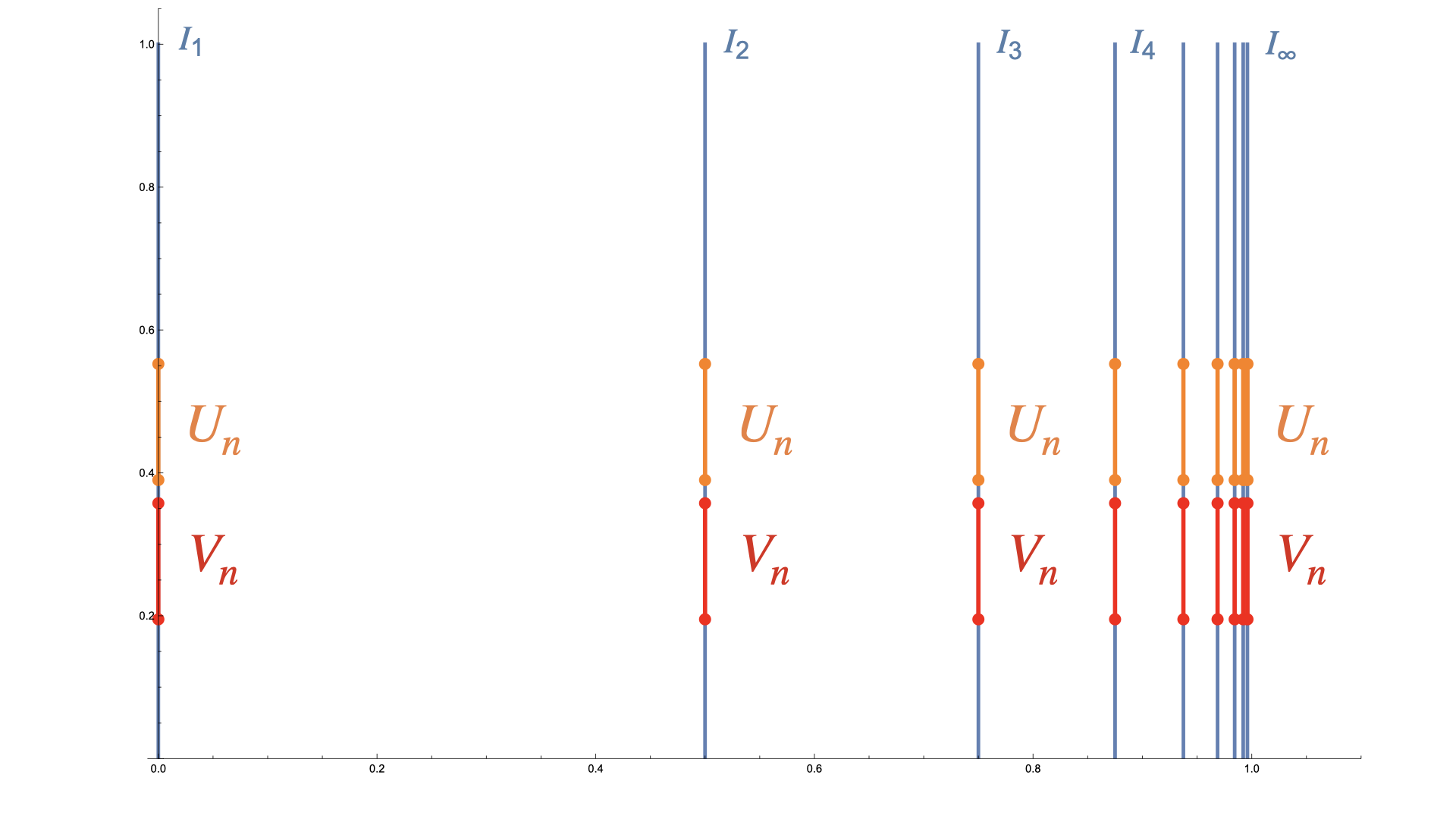}}
 \caption{The system defined in Example \ref{exa_perfect}. Each sub-interval $U_n$ in $I_k$ is mapped onto $I_{k+1}$ by $f_n$. Each sub-interval $V_n$ in $I_k$ is mapped onto $I_\infty$ by $f_n$. We have thus, for every $k$, $\{f\}^\omega(I_k)=I_{k+1}\cup I_\infty$.}\label{fig_perf_att}
 \end{figure}
\end{exa}

\vspace{0.5cm}

\begin{figure}[H]
    \centering
    \includegraphics[width=0.9\linewidth]{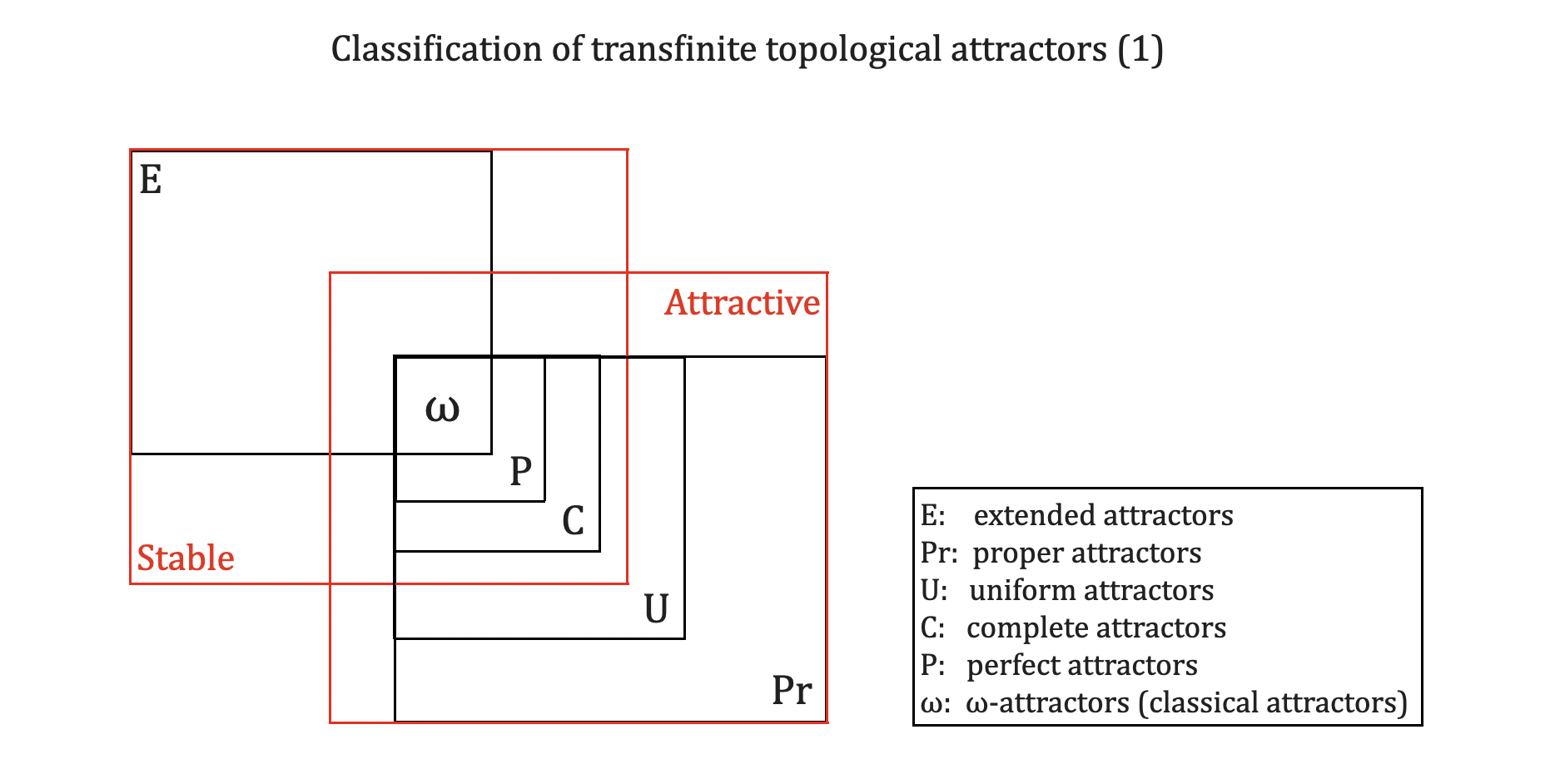}
   \caption{Proper and extended transfinite attractors intersect at $\omega$-attractors, that is attractors in finite systems; proper attractors are progressively particularized by uniform, complete and perfect attractors. Attractive and stable objects are highlighted.}
    \label{scheme_attrac_1}
\end{figure}

\vspace{0.5cm}

\begin{figure}[H]
    \centering
    \includegraphics[width=0.9\linewidth]{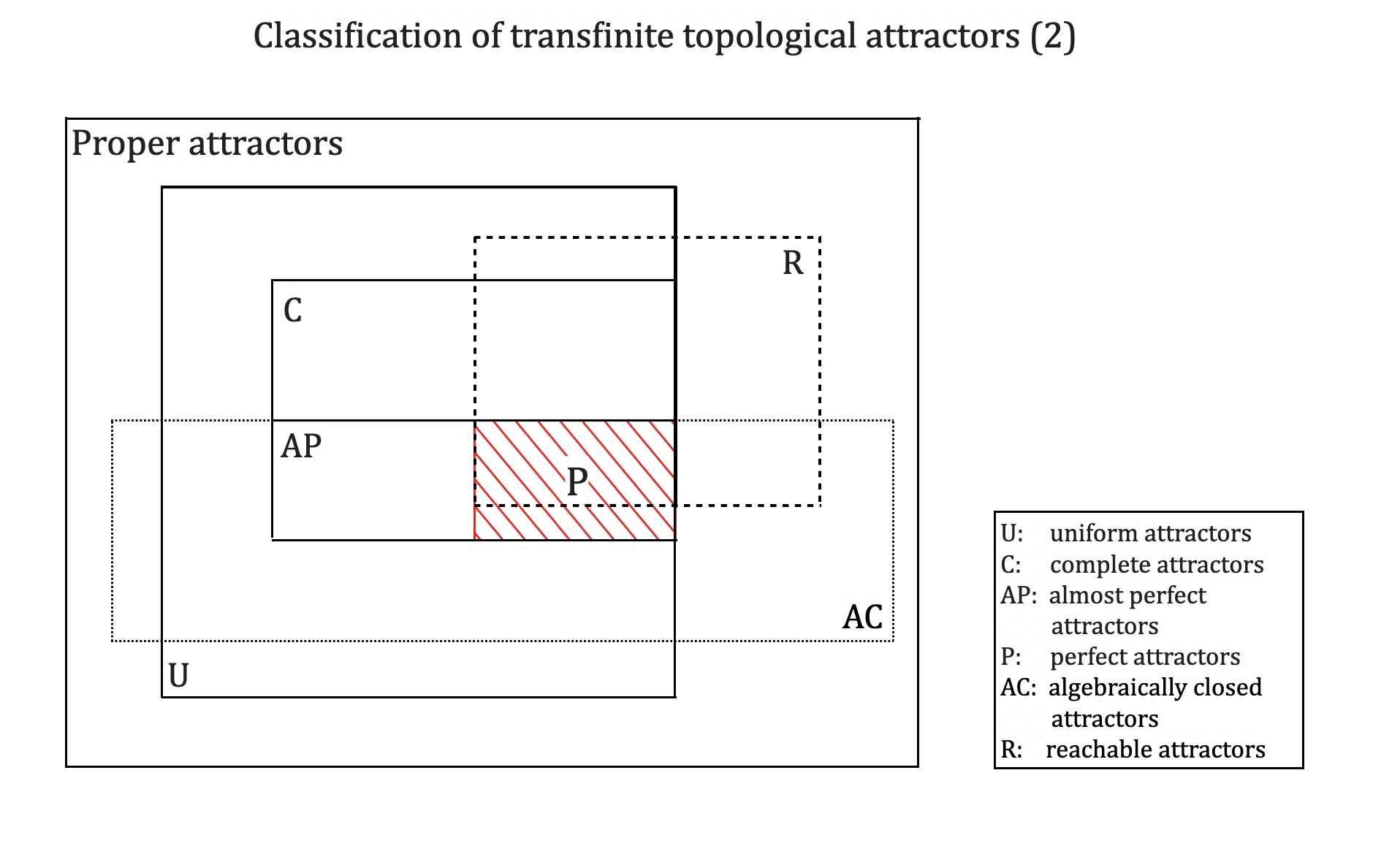}
    \caption{Scheme of proper transfinite attractors. Perfect attractors are attractive, stable, algebraically closed and transfinitely strongly invariant.}
    \label{scheme_attrac_2}
\end{figure}

\section{Transfinite conjugacy}\label{sec8}
In this Section, we generalize the basic ideas and results about conjugacy and topological conjugacy to TDSs. 
Considering two compact metric spaces $X$, $Y$, whose metric is denoted respectively by $d_X$ and $d_Y$, we start by giving the following definition:
\begin{defi}[\textbf{Transfinite conjugacy}]
\label{conjugacy}
    Let $\omega\le\lambda\le\omega_1$ be a limit ordinal. We say that two TDSs, $(X,\{f\})$ and $(Y,\{g\})$, are \emph{$\lambda$-conjugate} if there exists a bijection $h:X\to Y$ such that, for every $\beta<\lambda$, we have 
    \begin{equation}\label{conj___}
     h(X^{\beta+1})=Y^{\beta+1}   
    \end{equation} and, for every $\beta<\lambda$,
    \begin{equation}\label{conj_}
        \{f\}^\beta=h^{-1}\circ \{g\}^\beta\circ h. 
    \end{equation}

    We say that $(X,\{f\})$ and $(Y,\{g\})$ are \emph{topologically $\lambda$-conjugate} if they are conjugate and the map $h$ is a homeomorphism.
    \\
    We say that $(X,\{f\})$ and $(Y,\{g\})$ are (topologically) \emph{completely conjugate} if there is $\lambda\le\omega_1$ such that $\lambda=\mathfrak{D}(X,\{f\})=\mathfrak{D}(Y,\{g\})$ and that $(X,\{f\})$ and $(Y,\{g\})$ are (topologically) $\lambda$-conjugate.
\end{defi}
Both being $\lambda$-conjugate and completely conjugate are equivalence relations. The equivalence classes of these relations having representative $(X,\{f\})$  will be indicated respectively by $\lambda[X,\{f\}]$ and $[X,\{f\}]$.

    Since for every TDS $(X,\{f\})$ we have $X^\omega=X$, the systems $(X,\{f\})$ and $(Y,\{g\})$ are (topologically) $\omega$-conjugate if and only if the finite systems $(X,f)$ and $(Y,g)$ are (topologically) conjugate in the ordinary sense. It is also immediate to see that, for $\lambda \geq \beta$, $\lambda$-conjugacy implies $\beta$-conjugacy. 
    
    \begin{rem}
        
Since two TDSs with the same limit map can have  totally different transfinite dynamics, it is immediate to see that $\lambda$-conjugacy distinguishes sequences $\{f_n\}$ and $\{g_n\}$ in cases where their pointwise limit maps $f$ and $g$ are topologically conjugate in the classical sense. 
  \end{rem}

\begin{rem}
    We could have assumed that the equality \eqref{conj_} be valid just for ordinals of the form $\beta=\omega^\eta<\lambda$ for some countable ordinal $\eta$. Indeed, every ordinal can be written as a finite sum of ordinals of type $\omega^\eta$, so every transfinite iteration is a finite composition, by means of the rules \eqref{composition1}-\eqref{composition2}, of iterations of type $\omega^\eta$. From this, equality \eqref{conj_} for arbitrary ordinals smaller than $\lambda$ follows by induction on the number of terms in the composition.
\end{rem}

\begin{thm}\label{_conj__}
    If $(X,\{f\})$ and $(Y,\{g\})$ are topologically $\lambda$-conjugate TDSs and $\beta\le\lambda$, then:
    \begin{itemize}
        \item[1.] for every $x\in X$, we have $h(\{\mathcal{O}\}(x))=\{\mathcal{O}\}(h(x))$ and the sets $\{\mathcal{O}\}(x)$ and $\{\mathcal{O}\}(h(x))$ are order-isomorphic with respect to $<_\infty$; in particular, every transfinite cycle $\mathcal{C}\subseteq X$ is homeomorphic to a transfinite cycle $\mathcal{C}'\subseteq Y$ of the same order.
        \item[2.] for every (uniformly, completely) $\beta$-inward set $V\subseteq X$ we have that $h(V)\subseteq Y$ is a (uniformly, completely) $\beta$-inward set;
        \item[3.] for every pair of subsets $V\subseteq X$ and $U\subseteq Y$ such that $U=h(V)$, we have $$p\beta(U)=h(p\beta(V))\text{ and }e\beta(U)=h(e\beta(V)).$$ In particular, every proper (extended) $\beta$-attractor in $(X,\{f\})$ is homeomorphic to a proper (extended) $\beta$-attractor in $(Y,\{g\}$; moreover, for attractors, both being uniform complete are poreserved by conjugacy;
        \item[4.] If $A\subset X$ is a proper $\beta$-attractor in $X$, then $h(\mathcal{B}(A))=\mathcal{B}(h(A))$. 
        \item[5.] $A\subseteq X$ is a perfect $\lambda$-attractor for $(X,\{f\})$ if and only if $h(A)$ is a perfect $\lambda$-attractor for $(Y,\{g\})$; 
        \item [6.] For $\beta$ a limit ordinal smaller than $\lambda$, a set $S\subset X$ is $\beta$-reachable in $(X,\{f\})$ if and only if $h(S)$ is $\beta$-reachable in $(Y,\{g\})$.
        \item[7.] $\lambda$-minimality and $\lambda$-transitivity are preserved by topological $\lambda$-conjugacy. 
    \end{itemize}
\end{thm}

\begin{proof}
\begin{itemize}
    \item[1.] First of all, since for every countable ordinal $\beta<\lambda$ we have $h(X^{\beta+1})=Y^{\beta+1}$, it follows: 
    $$
    \{f\}^\beta(x)\ne\emptyset\iff x\in X^{\beta+1}\iff h(x)\in Y^{\beta+1}\iff \{g\}^\beta(h(x))\ne\emptyset,
    $$ 
    so that \begin{equation}\label{order_iso_orbit}
   \eta:=D(x)=D\big(h(x)\big).
    \end{equation}
    Therefore we have:
    $$
    h\left(\{\mathcal{O}\}(x)\right)=h\left(\Big\{\{f\}^\beta(x)\Big\}_{\beta<\eta}\right)=\Big\{\{g\}^\beta\left(h(x)\right)\Big\}_{\beta<\eta}=\{\mathcal{O}\}(h(x)).
    $$
    Since $\{\mathcal{O}\}(x)$ and $\{\mathcal{O}\}(h(x))$ are both strictly well-ordered by $<_\infty$, the last chain of equalities and \eqref{order_iso_orbit} imply that they are order-isomorphic with respect to $<_\infty$. Notice that here we only used $\lambda$-conjugacy and not topological $\lambda$-conjugacy.
    \item[2.] i) The equivalence: $$V \text{ $\lambda$-inward}\iff h(V)\text{ $\lambda$-inward}$$ follows from the fact that both the interior and the closure are preserved by homeomorphisms, so, for $\beta<\lambda$, we have $\{g\}^\beta(\overline{h(V)})=\{g\}^\beta(h(\overline{V}))$ and $(h(V))^\circ=h(V^\circ)$. Thus,
    $$\{f\}^\beta(\overline{V})=h^{-1}(\{g\}^\beta(\overline{h(V)}))\subseteq V^\circ\iff \{g\}^\beta(\overline{h(V)})\subseteq (h(V))^\circ.$$
    
    ii) Let us prove the equivalence: $$V \text{ uniformly $\lambda$-inward}\iff h(V)\text{ uniformly $\lambda$-inward}.$$
    Set $h(V)=U$ and $S:=\{f^\eta(\overline{V}),\ \eta<\lambda\}$. Suppose that $V$ is $\lambda$-inward but not uniformly $\lambda$-inward, so that $ d_X(\partial V, S)=0$. Notice that $h(\partial V)=\partial h(V)=\partial U$ and that 
    $$
    T:=h(S)=\{g^\eta(\overline{U}),\ \eta<\lambda\}.
    $$  
    Then there are two sequences of points $\{x_n\}_{n\in\mathbb{N}}\subseteq S$ and $\{y_n\}_{n\in\mathbb{N}}\subseteq \partial V$ such that $\lim_{n\to\infty} d_X(x_n,y_n)=0$. 
    Since $\partial V$ is compact, $y_n$ converges to some $y\in \partial V$ up to a subsequence, so we have as well $x_{n_m}\longrightarrow y$ as $m$ goes to $\infty$ for a suitable subsequence $\{x_{n_m}\}_{m\in\mathbb{N}}$. 
    Then 
    $$
    \{h(x_{n_m})\}_{m\in\mathbb{N}}\subseteq T\ \text{  and  }\  h(x_{n_m})\xrightarrow{m\to\infty} h(y)\in \partial U
    $$ 
    which implies $d_Y(\partial U, T)=0$. The argument applied conversely (using $h^{-1}$ instead of $h$) completes the proof.
    \\
    iii) To derive the equivalence $$V \text{ completely $\lambda$-inward}\iff h(V)\text{ completely $\lambda$-inward}$$ observe that, assuming $V$ completely inward, for $\eta\le\beta<\lambda$ we have 
    $$
    h^{-1}(\{g\}^\beta(h(\overline{V})))=\{f\}^\beta(\overline{V})\subseteq\bigcap_{\eta<\beta}\{f\}^\eta (\overline{V})=\bigcap_{\eta<\beta}h^{-1}(\{g\}^\eta(h(\overline{V}))),
    $$
    so that
    $$
    \{g\}^\beta(h(\overline{V}))\subseteq h\left(\cap_{\eta<\beta}h^{-1}(\{g\}^\eta(h(\overline{V})))\right)\subseteq \bigcap_{\eta<\beta}\{g\}^\eta(h(\overline{V})).
    $$
    \item [3.] Assume $\beta\le\lambda$, $V\subseteq X$ and $h(V)=U$.
    
    By Proposition \ref{th_plim_seq}, $x\in p\beta(V)$ if and only if there is a sequence $\{x_n\}_{n\in\mathbb{N}}$ and a strictly increasing sequence of ordinals $\{\beta_n\}_{n\in\mathbb{N}}$ converging to $\beta$ such that  $\lim_{n\to\infty}\{f\}^{\beta_n}(x_n)=x.$
    In this case, setting $y_n=h(x_n)$ for every $n\in\mathbb{N}$, by conjugacy and exploiting the continuity of $h$, we have 
    $$
    \lim_{n\to\infty}h(\{f\}^{\beta_n}(x_n))= \lim_{n\to\infty}\{g\}^{\beta_n}(y_n)=h(x). 
    $$
    This proves the inclusion $h(p\beta(V))\subseteq p\beta(U)$ , while the converse inclusion follows from an analogous argument applied using $h^{-1}$ instead of $h$. Therefore, proper $\beta$-limits are preserved by conjugacy. By point 2., and since $\lambda$-closedness is preserved by homeomorphisms,  this implies that the $h$-images of proper (uniform, complete) $\beta$-attractors in $(X,\{f\})$ are proper (uniform, complete) $\beta$-attractors in $(Y,\{g\})$. 
    
    The equality $e\beta(U)=h(e\beta(V))$ follows from the result just proven recalling that the extended $\beta$-limit is defined as a union of proper $\beta$-limits (see Def. \ref{def_lambda_limits}). 

    \item[4.] By point 3., we know that $h(A)$ is a proper $\beta$-attractor in $Y$. We have $$x\in\mathcal{B}(A)\iff \lim_{\eta\to\beta}d_X(\{f\}^\eta(x),A)=0,$$
    which implies by conjugacy, and exploiting compactness of $A$: 
    $$ \lim_{\eta\to\beta}d_Y(h(\{f\}^\eta(x)),h(A))=\lim_{\eta\to\beta}d_Y(\{g\}^\eta(h(x)),h(A))=0$$ and thus $h(x)\in \mathcal{B}(h(A))$. This proves $h(\mathcal{B}(A))\subseteq \mathcal{B}(h(A))$. The converse inclusion follows analogously with the argument applied to $h^{-1}$.
    
    \item[5.]  By item 3. above, $h(A)$ is a proper attractor in $Y$ and, by item 2., it is complete.  Moreover, $A$ algebraically closed means that $\lambda=\omega^\eta$ for a suitable $\eta<\omega_1$, and this implies that $h(A)$ is also algebraically closed. 

    Take $\beta<\lambda$. Since $A$ is strongly $\lambda$-invariant, by $\lambda$-conjugacy we have:

    $$ h(A)=h(\{f\}^\beta(A))=\{g\}^\beta(h(A)),$$

    so that the $\lambda$-attractor $h(A)$ is strongly $\lambda$-invariant and therefore perfect. 

    \item[6.] Take $\beta<\lambda$.  We have $x\in S\cap \overline{X^{\beta+1}}$ if and only if $$h(x)\in h(S)\cap h(\overline{X^{\beta+1}})= h(S)\cap \overline{h(X^{\beta+1})}=  h(S)\cap\overline{Y^{\beta+1}}.$$

    Moreover, $x\in X^{\beta+1}$ if and only if $h(x)\in h(X^{\beta+1})=Y^{\beta+1}$.
    
    \item[7.] Assume that $(X,\{f\})$ is $\lambda$-minimal, so that every point in $X$ has a dense $\lambda$-orbit. Since being dense is preserved by homeomorphisms, it follows from item 1. above that every $\lambda$-orbit in $Y$ is dense as well, which implies that $(Y,\{g\})$ is $\lambda$-minimal.
    
    Assume now that $(X,\{f\})$ is $\lambda$-transitive. 
    Pick $(y,y')\in Y^2$ and two nonempty open neighbors $V(y), V'(y')$. Then, setting $$h^{-1}(y)=x,\ h^{-1}(y')=x',\ h^{-1}(V)=U,\ h^{-1}(V')=U',$$ we have that $U, U'$ are nonempty open neighbors of respectively $x$ and $x'$, and that there are two points $z\in U$ and $z'\in U'$ and some $\beta<\lambda$ such that $\{f\}^\beta(z)=z'$. Then, setting $h(z)=:w$ and $h(z')=:w'$, we have $w\in h(U)$, $w'\in h(U')$ and, recalling item 1. above,
    $$
    \{g\}^\beta(w)=h(\{f\}^\beta(z))=h(z')=w'.
    $$
\end{itemize}

\end{proof}
We now give a stronger notion of transfinite conjugacy, that simply requires that two TDSs consist of step-by-steo conjugate finite systems, so that we have in fact a sequence of pairs of conjugate systems in the standard sense.

More precisely, let us give the following Definition.
\begin{defi}\label{seq_conj___}
    We say that two TDSs $(X,\{f\})$ and $(Y,\{g\})$ are \emph{sequentially conjugate} if there is a bijection $h:X\to Y$ such that, for every $n\in\mathbb{N}$, we have $f_n=h^{-1}\circ g_n \circ h$. We say that $(X,\{f\})$ and $(Y,\{g\})$ are \emph{topologically sequentially conjugate} if they are completely conjugate and the map $h$ is a homeomorphism.
\end{defi}
    Being sequentially conjugate is an equivalence relation. The equivalence class of this relation having representative $(X,\{f\})$ will be indicated by $S[X,\{f\}]$.

    To see that sequential conjugacy is stronger than topological $\lambda$-conjugacy, consider again Example \ref{exa1}. Suppose that, for every $n\in\mathbb{N}$, we can replace $f_n$ with a map $g_n$ such that: 
    \begin{enumerate} 
    \item $g_n=f_n$ for $x\in\mathbb{I}\setminus(u_n,v_n)$; 
    \item 
    $g_n(z_n)=f_n(z_n)=h$. 
    \end{enumerate}
    We have $g_n\dot{\longrightarrow}f$ and, in the TDS $(X,\{g\})$ we have $\{g\}^\beta(x)=\{f\}^\beta(x)$ for every $x\in\mathbb{I}$ and every $\beta<\omega_1$, so the systems $(X,\{f\})$ and $(X,\{g\})$ are $(\omega\cdot 2)$-conjugate and also completely conjugate by the identity map. However, it is clear that we can define $g_n$ on $(u_n,z_n)$ and $(z_n,v_n)$ in such a way that it cannot be conjugate with $f_n$ (for instance because it has a different number of fixed points). The system $(X,\{g\})$ can of course be made sequentially continuous if the maps $g_n$ are chosen to be continuous. 

    The next result shows that the ordinal degree (locally at each $x$, and therefore globally) is an invariant in the standard sense, being preserved under sequential conjugacy. 

\begin{thm}\label{sequenc_conj__}
If $(X,\{f\})$ and $(Y,\{g\})$ are (topologically) sequentially conjugate, then they are (topologically) completely conjugate.
\end{thm}
\begin{proof}
Pick $x\in X$ and set $h(x)=y$. Let us observe first that, if $h$ is a conjugacy, then, for every $n\in\mathbb{N}$, we have $h(\mathcal{O}_{f_n}(x))=\mathcal{O}_{g_n}(y)$, so that 
$
h(\mathcal{O}_\infty(x))=\mathcal{O}_\infty(y).
$
Moreover, $h$ is also an order-isomorphism with respect to $<_n$ between $\mathcal{O}_{f_n}(x)$ and $\mathcal{O}_{g_n}(y)$, that is, for every $a,b\in \mathcal{O}_\infty(x)$ and every $n\in\mathbb{N}$, we have $a<_{x,n}b \iff h(a)<_{y,n} h(b)$ and this implies 
$
a<_{x,\infty}b \iff h(a)<_{y,\infty} h(b).
$
Therefore, $h([\mathcal{O}](x))=[\mathcal{O}](y)$, with $h$  an order-isomorphism with respect to the order induced by ordinal iteration index, so that, for every $\beta<\omega_1$ such that $[f]^\beta(x)$ is defined, we have:
\begin{equation}\label{basic_conj}      
[f]^\beta(x)=(h^{-1}\circ [g]^\beta \circ h)(x)=(h^{-1}\circ [g]^\beta)(y). 
\end{equation}
Moreover, assuming that, for some $a,b,c\in X$ and $\beta,\eta<\omega_1$ we have $[f]^\beta(a)=b$ and $[f]^\eta(b)=c$, then, using \eqref{basic_conj}, it follows: 
\[
h([f]^{\eta+\beta}(a))=h([f]^\beta([f]^\eta(a)))=[g]^\beta([g]^\eta(h(a)))=[g]^{\eta+\beta}(h(a)),
\]
whenever $[f]^{\eta+\beta}$ exists. It follows that also the transfinite iterations defined in Def. \ref{cycle}, and in particular the composition law \eqref{composition1}-\eqref{composition2}, are preserved by conjugacy. 
Therefore, it follows that $$h(\{\mathcal{O}\}(x))=\{\mathcal{O}\}(y),$$ again with $h$ being an order-isomorphism with respect to ordinal iteration index, so that the conjugacy law \eqref{conj_} follows. 
In particular, we have $D(x)=D(y)$. 
Since $x$ was arbitrary, this means that $\mathfrak{D}(X,\{f\})=\mathfrak{D}(Y,\{g\})=:\lambda$. 
Moreover,  for every $\beta<\lambda$, 
$$
h(X^{\beta+1})=Y^{\beta+1}\  \text{ and }\ \   h(\{f\}^\beta(x))=\{g\}^\beta(y).
$$ 
If $h$ is a homeomorphism, the same reasoning shows that $(X,\{f\})$ and $(Y,\{g\})$ are topologically completely conjugate.
\end{proof}
Summarizing, we have always:
$$\text{Sequential conjugacy $\implies$ Complete conjugacy $\implies\ \lambda$-conjugacy}$$
so that, for every TDS and every limit ordinal $\lambda<\omega_1$, 
$$
S[X,\{f\}]\subseteq [X,\{f\}]\subseteq \lambda[X,\{f\}].
$$

\clearpage 

\section*{Main definitions}\label{main_def_}
\begin{itemize}
\item Transfinite Dynamical System (TDS) -- Def. \ref{defi_tds}
\item Transfinite cycle -- Def. \ref{cycles}
\item Ordinal Degree of a point and of a TDS  -- Def. \ref{degree_}

\item $\lambda$-continuous function -- Def. \ref{cont___}
    \item $\lambda$-regular and $\lambda^*$-regular system -- Def. \ref{regularity_} and Def. \ref{lambda_star}
    \item $\lambda$-normal and $\lambda^*$-normal system -- Def. \ref{regularity_} and Def. \ref{lambda_star}
\item Transfinite dynamical relations --  Def. \ref{def_tdr}
\item Proper and extended $\lambda$-limit -- Def. \ref{def_lambda_limits}
\item $\lambda$-inward set and uniformly $\lambda$-inward set -- Def. \ref{inward__}
\item $\lambda$-attractor and uniform $\lambda$-attractor-- Def. \ref{defi_attractors} 
\item Transfinite equicontinuity -- Def. \ref{def_equi_trans}
\item Transfinite sensitivity -- Def. \ref{transsens__}
\item Completely $\lambda$-inward set -- Def. \ref{complet__}
\item Perfect $\lambda$-attractor -- Def. \ref{perf_attr__}
    \item $\lambda$-conjugacy, complete conjugacy -- Def. \ref{conjugacy}
    \item Sequential conjugacy -- Def. \ref{seq_conj___}

\end{itemize}

\clearpage

\bibliographystyle{plain}
\bibliography{transfinite}

\end{document}